\documentclass[10pt]{article}
\usepackage[margin=0.82in]{geometry}
\usepackage[T1]{fontenc}
\usepackage[utf8]{inputenc}
\usepackage[english]{babel}
\usepackage{csquotes}
\usepackage{amsmath}
\usepackage{amsfonts}
\usepackage{mathrsfs}
\usepackage{mathtools}
\usepackage{amssymb}
\usepackage{amsthm}
\usepackage{bbm}
\usepackage{dsfont}
\usepackage{hyperref}

\newtheorem{theorem}{Theorem}[section]
\newtheorem{lemma}[theorem]{Lemma}
\newtheorem{proposition}[theorem]{Proposition}
\newtheorem{corollary}[theorem]{Corollary}
\newtheorem{conjecture}[theorem]{Conjecture}
\newtheorem{definition}[theorem]{Definition}

\newtheorem{assumption}{Assumption}
\newtheorem{remark}{Remark}

\DeclareMathOperator{\diam}{diam}

\DeclareMathOperator{\ric}{Ric}

\DeclareMathOperator{\ricoll}{Ric_{dc}}

\DeclareMathOperator{\mlsi}{MLSI}
\DeclareMathOperator{\csi}{CSI}

\DeclareMathOperator{\hwi}{HWI}

\DeclareMathOperator{\po}{PI}
\DeclareMathOperator{\cd}{CD}
\DeclareMathOperator{\am}{AM}

\DeclareMathOperator{\ricollp}{Ric_{dc,p}}
\DeclareMathOperator{\ricollone}{Ric_{dc,1}}
\DeclareMathOperator{\ricollinf}{Ric_{dc,\infty}}

\DeclareMathOperator{\ricllyp}{Ric_{cc,p}}
\DeclareMathOperator{\riclly}{Ric_{cc}}
\DeclareMathOperator{\ricllyinf}{Ric_{cc,\infty}}
\DeclareMathOperator{\ricllyone}{Ric_{cc,1}}
\DeclareMathOperator{\rice}{Ric_{e}}

\DeclarePairedDelimiter{\abs}{\lvert}{\rvert}
\DeclarePairedDelimiter{\tond}{(}{)} 
\DeclarePairedDelimiter{\quadr}{[}{]}
\DeclarePairedDelimiter{\graf}{\{}{\}} 
 
\DeclarePairedDelimiter{\floor}{\lfloor}{\rfloor}

\DeclarePairedDelimiter{\tv}{\lVert}{\rVert_{\text{TV}}}

\newcommand{\numberset}{\mathbb}
\newcommand{\N}{\numberset{N}}
\newcommand{\Z}{\numberset{Z}}
\newcommand{\R}{\numberset{R}}

\newcommand{\Q}{\numberset{Q}}

\newcommand{\dd}{\mathrm{d}}

\newcommand{\indic}[1]{\mathds{1}_{#1}}

\newcommand{\curlyW}{\mathscr{W}}
\newcommand{\curlyB}{\mathscr{B}}

\newcommand{\curlyC}{\mathscr{C}}

\newcommand{\curlyA}{\mathscr{A}}

\newcommand{\curlyD}{\mathscr{D}}

\newcommand{\expe}[1]{\mathbb{E}\quadr*{#1}}

\newcommand{\ccpl}{c^{\mbox{\scriptsize{cpl}}}}
\newcommand{\metacsigma}{m(\eta)c(\eta,\sigma)}
\newcommand{\gip}{\gamma_i^+}
\newcommand{\gim}{\gamma_i^-}
\newcommand{\gjp}{\gamma_j^+}
\newcommand{\gjm}{\gamma_j^-}
\newcommand{\gxp}{\gamma_x^+}
\newcommand{\gxm}{\gamma_x^-}
\newcommand{\gyp}{\gamma_y^+}
\newcommand{\gym}{\gamma_y^-}

\newcommand{\en}{\mathcal{E}}
\newcommand{\pp}{\mathcal{P}}
\newcommand{\hh}{\mathcal{H}}
\newcommand{\ii}{\mathcal{I}}

\newcommand{\kollone}{K_{dc,1}}
\newcommand{\kollp}{K_{dc,p}}
\newcommand{\kollinf}{K_{dc,\infty}}

\newcommand{\kllyp}{K_{cc,p}}

\newcommand{\kllyinf}{K_{cc,\infty}}
\newcommand{\kllyone}{K_{cc,1}}

\newcommand{\ptil}{\tilde{P}}

\newcommand{\ttt}[1]{\mathcal{T}\tond*{#1}}

\newcommand{\gipn}{\gamma_i^{+,N}}
\newcommand{\gimn}{\gamma_i^{-,N}}

\title{Contractive coupling rates and curvature lower bounds \\for  Markov chains}
\author{Francesco Pedrotti}

\begin{document}
	
	\maketitle

	\begin{abstract}
		Contractive coupling rates have been recently introduced by Conforti as a tool to establish convex Sobolev inequalities (including modified log-Sobolev and Poincar\'{e} inequality) for some classes of Markov chains. 
			In this work, for most of the examples discussed by Conforti, we use contractive coupling rates to prove stronger inequalities, in the form of  curvature lower bounds (in entropic and discrete Bakry--\'{E}mery sense) and geodesic convexity of some entropic functionals.
			In addition, we recall and give straightforward generalizations of some notions of coarse Ricci curvature, and we discuss some of their properties and relations with the concepts of couplings and coupling rates: as an application, we show exponential contraction of the $p$-Wasserstein distance for the heat flow in the aforementioned examples.
	\end{abstract}

	\tableofcontents

	\section{Introduction}
	
	In this work, we are mostly concerned with finite state space continuous time Markov chains and we assume that they are irreducible and reversible: we use the letter $\Omega$ for the state space, $L$ for the generator and $m$ for the invariant measure. 
	A fundamental problem in the theory of Markov chains consists in estimating the speed of convergence to the stationary distribution and giving upper bound for its mixing time. Convex Sobolev inequalities are a particularly useful tool to address this task. Given a convex function $\phi \colon \R_{\ge 0} \to \R$ such that $\phi \in C^1(\R_{> 0})$, we define the $\phi$-entropy of a function
	$\rho\colon\Omega \to\R_{>0}$ by 
	\begin{equation}
		\label{eq:phi-entropy}
		\begin{split}
			\hh^\phi \tond*{\rho} \coloneqq \mathbb{E}_m\quadr*{\phi \circ\rho} - \phi\tond*{\mathbb{E}_m\quadr*{\rho}}.
		\end{split}
	\end{equation}
	Notice that by Jensen's inequality $\hh^\phi(\rho) \geq 0$ and $\hh^\phi(C) = 0$ for any constant $C>0$. 
	Therefore, if we denote by $\rho = \frac{d\mu}{dm}$ the density of a probability measure $\mu$ with respect to $m$, we can think of 
	$	\hh^\phi\tond*{\rho}$ as a (non-symmetric) measure of distance of $\mu$ from  $m$.
	We also recall the definition of the Dirichlet form $\en \colon \R^\Omega \times \R^\Omega \to \R$ via
	\begin{align*}
		\en\tond*{f,g}   \coloneqq -\mathbb{E}_m \quadr*{f (Lg)},
	\end{align*}
	and of the $\phi-$Fisher information
	\begin{equation}\label{eq:phi-fisher-inf}
		\ii^\phi\tond*{\rho}  \coloneqq \en\tond*{\rho, \phi'\circ \rho}.
	\end{equation}
	We then say that a $\phi$-convex Sobolev inequality holds with constant $K>0$ (notation: $\csi_{\phi}(K)$) if  for all positive functions $\rho\colon \Omega \to \R_{>0}$ we have that
	\begin{align}\label{eq:conv-sob-ineq}
		K \, \hh^\phi(\rho) \leq \ii^\phi(\rho).
	\end{align}
	The interest behind this inequality lies in the fact that, denoting by $P_t = e^{tL}$ the semigroup associated to the generator $L$,
	we have the well-known identity
	\[
	\frac{d}{dt}\hh^\phi\tond*{P_t\rho} = - \ii^\phi \tond*{P_t \rho}
	\]
	for any $\rho\colon \Omega \to \R_{\ge 0}$.
	Thus, by Gr\"{o}nwall Lemma, \eqref{eq:conv-sob-ineq}  is equivalent to the exponential decay of the entropy along the heat flow
	\[
	\hh^\phi\tond*{P_t \rho} \leq e^{-Kt}\hh^\phi\tond*{\rho},
	\]
	and therefore quantifies the speed of convergence to equilibrium of the Markov chain. 
	Classical choices of the function $\phi$ include the function $\phi_\alpha$ for $\alpha \in [1,2]$ defined by
	\begin{equation}\label{eq:beckner-func}
		\phi_\alpha (t) = \begin{cases}
			t\log t - t+1 & \text{ if } \alpha=1,
			\\
			\frac{t^\alpha-t}{\alpha-1}-t + 1 & \text{ if } \alpha\in (1,2].
		\end{cases}
	\end{equation}
	
	When $\phi = \phi_1$, we get the relative entropy  and inequality  \eqref{eq:conv-sob-ineq} is the celebrated modified log-Sobolev inequality \cite{bob-tet-2006} (notation: $\mlsi(K)$); when $\phi = \phi_2$, we find the variance and \eqref{eq:conv-sob-ineq} is the Poincar\'{e} inequality (notation: $\po(K)$). For $\alpha \in (1,2)$, inequalities \eqref{eq:conv-sob-ineq} are known as Beckner inequalities, which interpolate between modified log-Sobolev and Poincar\'{e} \cite{bob-tet-2006,jun-yue-2017}.

	\paragraph*{Curvature of Markov chains}

	In the setting of Riemannian manifolds, positive lower bounds for the Ricci curvature have been linked to many functional inequalities: this has motivated the seminal independent works of Sturm \cite{stu-2006} and Lott and Villani \cite{lot-vil-2009}, who extended the notion of curvature lower bound and many of its consequences (including some logarithmic Sobolev inequalities) to a large class of geodesic metric measure spaces. 
	In spite of its generality, this theory does not apply to Markov chains on discrete spaces; for this reason, several adapted notions of curvature have been proposed, based on different equivalent characterisations of Ricci curvature for Riemannian manifolds.
	Among these, we recall in particular the entropic curvature by Erbar and Maas \cite{erb-maa-2012}, which is based on displacement convexity of the relative entropy with respect to an adapted Wasserstein-like metric $\curlyW$ introduced in \cite{maa-2011} (see also the work of Mielke \cite{mie-2013}). This theory shares many similarities with the classical Lott--Sturm--Villani one, and among its merits it is such that many of the desired functional inequalities follow from positive lower bounds on the Ricci curvature, including in particular the modified log-Sobolev inequality.
	Moreover, as shown in \cite{erb-maa-2014}, the role of the classical relative entropy (with respect to $m$) can be taken over by other $\phi$-entropy functionals as defined in \eqref{eq:phi-entropy}, provided that one changes accordingly a parameter function in the definition of $\curlyW$: once again, from positive geodesic convexity one can derive many consequences, including the convex Sobolev inequality \eqref{eq:conv-sob-ineq}. 
	Unfortunately, establishing positive lower bounds for the entropic Ricci curvature (or more generally $K$-geodesic convexity of an entropic functional $\hh^\phi$) of a Markov chain can be challenging, and in many interesting examples good estimates are not available.

	\paragraph*{Coupling rates and Conforti's results}
	While studying the entropic curvature of a Markov chain is a difficult task, in general even finding good estimates on the best constant for the modified log-Sobolev inequality (or other convex Sobolev inequalities) can be difficult.
	For this reason, in the recent paper \cite{con-2022}, Conforti introduced a method based on the new notion of \emph{coupling rates} to study general convex Sobolev inequalities, and applied it to some interesting classes of Markov chains.
	Coupling rates are a modification of the familiar notion of coupling: roughly speaking, they are used to ``couple'' the action of the generator $L$ from two different states. While couplings have been extensively used to establish fast convergence of Markov chains (``probabilistic'' approach to fast mixing), their use to establish convex Sobolev inequalities (which belong to the ``analytic'' approach to fast mixing) is less common, and the results of \cite{con-2022} give an interesting connection between these two families of methods.

	\paragraph*{Our contribution and organization of the paper}
	In this work, we show that the coupling rates introduced by Conforti are a powerful tool to establish entropic curvature lower bounds and other related inequalities for some classes of Markov chains.
	As an illustration of the applications of these methods, we state below one particular instance of our results. We refer to Sections \ref{sec:gen-ineq}--\ref{sec:applications} below for precise definitions.
	\begin{theorem}[Cf. Sections \ref{sec:appl-glauber}, \ref{sec: Hardcore model}]
		Denote by $\rice$ the entropic curvature of a continuous time reversible Markov chain \cite{erb-maa-2012, fat-maa-2016, erbar-henderson-menz-tetali}. 
		\begin{itemize}
			\item For the Curie--Weiss model with size $N$ and parameter $\beta > 0$, in the limit $N\to \infty$ we have
			\[
			\rice \geq (1-\beta) + (1-2\beta) e^{-\beta}
			\] 
			for $\beta \leq \frac{1}{2}$.
			\item For the Ising model in dimension $d$ with parameter $\beta > 0$, we have \begin{equation*}
				\rice \geq 1 + e^{-2\beta d} - 3d(1-e^{-2\beta}) e^{2\beta d} 
			\end{equation*}
			if $  2d\tond*{1-e^{-2\beta}} e^{4d\beta} \leq 1$.
			\item For the hardcore model on a graph with maximum degree $\Delta$ and parameter $\beta>0$, set $\kappa_* = 1-\beta(\Delta-1)$ and $ \overline{\kappa_*} = \min \graf*{\beta, 1-\beta \Delta}$. Then
			\[
			\rice \geq  \frac{\kappa_*}{2}+ \overline{\kappa_*}
			\]
			provided that $\beta \Delta \leq 1$.
		\end{itemize}
	\end{theorem}
	We remark that in all the examples above we find new estimates for the entropic curvature of those Markov chains: these estimates imply in particular $\mlsi$ with the same constant obtained in \cite{con-2022}, but also other interesting functional inequalities, such as exponential contractivity along the heat flow of the popular Wasserstein-like metric $\curlyW$ of \cite{maa-2011} (see Section \ref{sec:log-mean} and references therein for more details). Therefore, in this sense, we provide a strengthening of the results of \cite{con-2022}.
	
	To conclude this section, we briefly present below the organization of this paper while giving a more complete overview of our contributions. 
	\begin{itemize}
		\item In Section \ref{sec:gen-ineq} we give some preliminary definitions and define the general abstract inequality that we consider: it reads $\curlyB(\rho, \psi) \geq K \curlyA (\rho, \psi)$ for a constant $K$ and all $\rho \colon \Omega \to \R_{>0}$ and $\psi\colon \Omega \to \R$, and it depends on an additional weight function $\theta\colon \R_{>0}\times \R_{>0} \to \R_{>0}$.
		This inequality was first introduced by Erbar and Maas in their works on the entropic curvature of a Markov chain \cite{erb-maa-2012} and, more generally, on the geodesic convexity of some entropic functionals \cite{erb-maa-2014}. To motivate the interest in studying this inequality and for the sake of completeness, in the following subsections we recall some results by Erbar and Maas in connection with specific choices of the weight function $\theta$.
		In particular:
		\begin{itemize}
			\item In Section \ref{sec:log-mean} we recall that when $\theta$ is the logarithmic mean then the inequality corresponds to an entropic curvature lower bound for the Markov chain, as proved in \cite{erb-maa-2012}. For the convenience of the reader, we also recall from \cite{erb-maa-2012} some necessary definitions and consequences of such curvature lower bound, including $\mlsi(2K)$.
			\item In Section \ref{sec:weight-conv}, by the results of \cite{erb-maa-2014}, we extend the considerations of the previous section to the case where the relative entropy is replaced by some other $\phi$-entropy functionals, and in particular we explain how we recover a family of convex Sobolev inequalities as in \eqref{eq:conv-sob-ineq} (consistently with \cite{con-2022}), together with other functional inequalities.
			\item In Section \ref{sec:disc-bak-eme} we recall that if $\theta$ is the arithmetic mean then the inequality of interest corresponds to a lower bound for another notion of discrete curvature, namely the discrete Bakry--\'{E}mery one (see \cite{Sch-1999}).
		\end{itemize} 
		\item In Section \ref{sec:coup-rates} we recall the definition of coupling rates. Moreover, under some basic assumptions on the weight function $\theta$, we use coupling rates to provide a lower bound for the quantity $\curlyB(\rho,\psi)$: this is the content of  Lemma \ref{lem: lower bound for B with coupling rates}, which will be crucial for our method and the applications to the particular Markov chains considered later in this paper. This section also includes some heuristic considerations, explaining the favorable role of ``contractivity'' of the couplings in proving the abstract inequality $\curlyB(\rho,\psi) \ge K \curlyA(\rho,\psi)$, which can also guide possible future applications of this tool.
		
		\item In Section \ref{sec:applications} we illustrate the considerations of the previous section by considering most of the examples discussed in \cite{con-2022} and correspondingly establishing the general abstract inequality $\curlyB(\rho, \psi) \geq K \curlyA (\rho, \psi)$ introduced in Section \ref{sec:gen-ineq}. More specifically:
		\begin{itemize}
			\item In Section \ref{sec: Glauber dynamics} we consider Glauber dynamics, which includes in particular the Ising model and the Curie--Weiss model.
			\item In Section \ref{sec: Bernoulli Laplace model} we consider a simplified version of the Bernoulli--Laplace model. 
			\item In Section \ref{sec: Hardcore model} we consider the classical hardcore model.
			\item In Section \ref{sec: Interacting random walks} we consider the case of interacting random walks on the discrete grid $\N^d$.
		\end{itemize}
		In particular, by choosing $\theta$ to be the logarithmic mean, we find new estimates for the entropic curvature of the Ising model, the Curie--Weiss model and the hardcore model, and we recover the best known lower bound for the entropic curvature of the Bernoulli--Laplace model.
		\item Finally, in Section \ref{sec:coup-curv}
		we explain how the coupling rates constructed by Conforti are naturally connected with the notion of coarse Ricci curvature by Ollivier \cite{oll-2009}. We recall some well-known definitions and properties of the coarse Ricci curvature, and we also provide some natural generalizations. In particular, inspired by the properties of the coupling rates constructed by Conforti, we consider a stronger notion of coarse Ricci curvature which, roughly speaking, is based on simultaneous contraction of $1$-Wasserstein distance $W_1$ and non-expansion of the $\infty$-Wasserstein distance $W_\infty$ along the Markov chain dynamics (where both Wasserstein distances are defined with respect to the natural graph distance $d$ on $\Omega$). Correspondingly, we raise the question of connecting positive lower bounds for this notion of curvature to a modified log-Sobolev inequality, formulating a weaker version of a conjecture by Peres and Tetali.
		
		As a further application of the discussion in Section \ref{sec:coup-curv}, and as already done by Conforti for the specific case of the interacting random walks of Section \ref{sec: Interacting random walks}, we show that in all the other examples discussed in Section \ref{sec:applications} Conforti's coupling rates imply an exponential contraction of the Wasserstein distance of the form
		\[
		W_p\tond*{P_t \rho, P_t \sigma} \leq e^{-\frac{Kt}{p}} W_p\tond*{\rho, \sigma}
		\]
		for all starting densities $\rho, \sigma$ and $p\geq 1$.
		
		As a final remark, we emphasize how, in some sense, this section shows that the connections between probabilistic and analytic methods emerging in 
		\cite{con-2022} carry over at the level of curvature. In fact, while contractive couplings are naturally linked to the coarse Ricci curvature, we use them to establish (for some classes of Markov chains) lower bounds on the entropic curvature, which is a rather analytic notion of curvature. 
	\end{itemize}

	\section{Preliminaries and main inequality}\label{sec:gen-ineq}
	
	Following \cite{con-2022}, we work with a so-called ``mapping representation'' of the Markov chain, which we briefly recall.  
	We are given a finite set of moves $G$ (where a move $\sigma \in G$ is a function $\sigma \colon \Omega \to \Omega$) together with a transition rate function $c\colon \Omega \times G \to \R_{\geq0}$, so that $c(\eta,\sigma)$ represents the rate of using the move $\sigma$ starting from the state $\eta$. Such a mapping representation has already proved useful before in establishing functional inequalities and curvature lower bounds for Markov chains \cite{pra-pos-2013, cap-pra-pos-2009, erb-maa-2012, fat-maa-2016, erbar-henderson-menz-tetali}.
	Typically (and if not otherwise specified) we use the letter $\eta$ for a state and $\sigma, \gamma, \bar{\gamma}$ for moves, and to lighten the notation we write for example $\sigma\eta$ instead of $\sigma(\eta)$ for the state reached after jumping with the move $\sigma$ from the state $\eta$. 
	We make the assumption that for each move $\sigma$ there exists a unique inverse $\sigma^{-1}\in G$ such that $\sigma^{-1} \sigma \eta= \eta$ whenever $m(\eta) c(\eta, \sigma) >0$ (recall that we denote by $m$ the unique invariant measure). We also denote by $e\colon \Omega \to \Omega$ the `null move'' (i.e. the identity map); without loss of generality, we assume that $e\notin G$ and we denote by $G^* \coloneqq G \cup \graf*{e}$ the enlarged set of moves, as in \cite{con-2022}. With this notation, we can write explicitly the action of the generator $L$ of the continuous time Markov chain in the form
	\begin{equation}\label{eq:generator}
		L\psi(\eta)=\sum_{\sigma \in G} c(\eta,\sigma)\tond*{\psi(\sigma \eta)-\psi(\eta)}
	\end{equation} 
	for any bounded $\psi\colon \Omega\to \R$. Notice that in \eqref{eq:generator} we could also take the sum for $\sigma \in G^*$, and that in this context the rates $c(\eta, e)$ can be arbitrarily defined.
	The state space $\Omega$ is at most countable, and we assume that 
	\[
	\sum_{\eta\in \Omega, \sigma \in G} m(\eta) c(\eta, \sigma)<\infty.
	\]
	We also use the notation
	\[
	\nabla_\sigma \psi(\eta) \coloneqq \psi(\sigma \eta) - \psi(\eta)
	\]
	for the discrete gradient
	and
	\[
	S \coloneqq \graf*{(\eta,\sigma) \in \Omega \times G \mid c(\eta,\sigma)>0}.
	\]
	We assume that 
	for all bounded functions $F\colon \Omega \times G \to \R$ we have
	\begin{equation}
		\label{eqn: reversibility with moves}
		\sum_{\eta \in \Omega, \sigma \in G} m(\eta) c(\eta, \sigma) F(\eta, \sigma) = \sum_{\eta \in \Omega, \sigma \in G} m(\eta) c(\eta, \sigma) F(\sigma \eta, \sigma^{-1}).
	\end{equation}
	As observed in e.g. \cite[Def 3.1]{fat-maa-2016} (cf. also \cite{pra-pos-2013, erb-maa-2012, erbar-henderson-menz-tetali, con-2022}), the condition \eqref{eqn: reversibility with moves} expresses the reversibility of the Markov chain, and every irreducible and reversible Markov chain admits a mapping representation satisfying the above requirements. In practice, it is typically useful to consider such a description where the set of moves $G$ is small.  In all the concrete examples of Markov chains considered in Section \ref{sec:applications}, following \cite{con-2022}, we will work with mapping representations satisfying the above conditions.

	We now proceed to introduce in an abstract way the main inequality of interest in this paper. For this, we first need an additional ingredient, i.e. a weight function $\theta \colon \R_{> 0} \times \R_{> 0} \to \R_{\geq 0}$. In this paper, we always work under the following basic
	\begin{assumption}
		The weight function $\theta$ is such that:
		\label{ass: weak assumption theta}
		\begin{enumerate}
			\item $\theta$ is not identically $0$;
			\item $\theta(s,t) = \theta(t,s)$;
			\item $\theta$ is differentiable;
			\item $\theta$ is concave.
		\end{enumerate}
	\end{assumption}

	Given a weight function $\theta$ satisfying the above, the inequality reads
	\begin{equation}
		\label{eq:main-ineq}
		\curlyB(\rho,\psi)\geq K \curlyA(\rho,\psi)
	\end{equation}
	for all positive functions $\rho\colon \Omega \to \R_{>0}$,  functions $\psi\colon\Omega\to \R$ and for a constant $K\in \R$ independent of $\rho,\psi$, where we have 
	\begin{align*}
		\curlyA(\rho,\psi) & = \frac{1}{2}\sum_{(\eta,\sigma)\in S} m(\eta)c(\eta,\sigma)\theta \tond*{\rho(\eta), \rho(\sigma \eta)} \quadr*{\psi(\eta)-\psi(\sigma \eta)}^2, \\
		\curlyB(\rho,\psi) &= \curlyC(\rho,\psi)-\curlyD(\rho,\psi) ,
	\end{align*}
	with
	\begin{align*}
		\curlyC(\rho,\psi) &= \frac{1}{4}\sum_{(\eta,\sigma)\in S}m(\eta)c(\eta, \sigma)\graf*{ \nabla \theta(\rho(\eta), \rho(\sigma\eta)) \cdot \begin{pmatrix}
				L \rho(\eta)\\
				L \rho(\sigma \eta)
		\end{pmatrix}} \quadr*{\psi(\eta)-\psi(\sigma \eta)}^2,\\
		\curlyD(\rho,\psi) &= \frac{1}{2}\sum_{(\eta,\sigma)\in S} m(\eta) c(\eta, \sigma) \theta(\rho(\eta),\rho(\sigma \eta)) \tond*{\psi(\eta)-\psi(\sigma \eta)}\tond*{L\psi(\eta)-L\psi(\sigma \eta)}.
	\end{align*}
	
	This inequality was introduced in the work of Erbar and Maas \cite{erb-maa-2012,erb-maa-2014}: depending on the choice of $\theta$, it has different interpretations and consequences, which we discuss in the next subsections.

	\subsection{Logarithmic mean and entropic curvature}\label{sec:log-mean}
	The main reason for studying inequality \eqref{eq:main-ineq} comes from the work  \cite{erb-maa-2012} and for choosing as $\theta$ the logarithmic mean
	\begin{align}\label{eq:log-mean}
		\theta_1(s,t) \coloneqq \int_0^1 s^{1-p}t^p \, dp = \begin{cases}
			\frac{s-t}{\log s - \log t} & \text{ if } s\neq t,
			\\
			s & \text{ if } s = t.
		\end{cases}
	\end{align}
	
	Let us now denote by $\pp(\Omega)$ the set of probability densities on $\Omega$ with respect to $m$, i.e. functions $\rho\colon \Omega \to \R_{\geq 0}$ such that
	$\mathbb{E}_m(\rho)=1$, and by $\pp_*(\Omega)$ the set of strictly positive densities.
	In \cite{maa-2011}, Maas introduced a Wasserstein-like metric  $\curlyW$ on $\pp(\Omega)$ via a discrete variant of the Benamou--Brenier formula, and showed that, as in the classical setting, for any $\rho \in \pp(\Omega)$ the heat flow $t\to P_t\rho = e ^{tL} \rho$ is the gradient flow of the relative entropy functional in $\tond*{\pp(\Omega), \curlyW}$ started at $\rho$, where the relative entropy is the restriction of the functional $\hh^{\phi_1}$ (as in \eqref{eq:phi-entropy}) to $\pp(\Omega)$. Writing it explicitly, for $\rho \in \pp(\Omega)$ we have that
	\[
	\hh^{\phi_1}(\rho) = \sum_{\eta\in \Omega} m(\eta) \rho(\eta)\log\tond*{\rho(\eta)},
	\]
	with the convention that $0\log 0= 0$ in the above sum.
	In this setting, inequality \eqref{eq:main-ineq} can be interpreted as a lower bound for the Hessian of $\hh^{\phi_1}$ with respect to $\curlyW$ (see \cite[Thm. 4.5]{erb-maa-2012}), or equivalently as a statement of $K$-geodesic convexity (see also the independent work of Mielke \cite{mie-2013}). 
	As gradient flows of geodesically $K$-convex functionals enjoy many useful properties, functional inequalities can be subsequently derived for the Markov chains.
	In the next Proposition we collect in particular some results proved in \cite{erb-maa-2012} (cf. Proposition $4.7$ and Theorems $7.3, 7.4$ therein). To be precise, \cite{erb-maa-2012} considers actually the case where the generator corresponds to a Markov Kernel (whose rows sum to $1$), but these definitions and properties easily extend to non-normalised transition rates, as considered in the subsequent literature \cite{erb-maa-2014, fat-maa-2016, erbar-henderson-menz-tetali}.
	\begin{proposition}\label{prop:coneqs-entr-low-bound}
		Assume that $\theta$ is the logarithmic mean and that inequality \eqref{eq:main-ineq} holds for some constant $K\in \R$ and for all $\psi \colon \Omega\to \R$, $\rho \in \pp_*(\Omega)$. Then: 
		\begin{itemize}
			\item the $\hwi(K)$ inequality
			\[
			H^{\phi_1}\tond*{{\rho}} \leq \curlyW(\rho, \mathbf{1})\sqrt{I^{\phi_1}(\rho)}-\frac{K}{2}\curlyW\tond*{\rho, \mathbf{1}}^2
			\]
			holds for all $\rho \in \pp\tond*{\Omega}$;
			\item for any $\rho, \sigma \in \pp\tond*{\Omega}$
			\[
			\curlyW\tond*{P_t\rho, P_t\sigma}\leq e^{-Kt} \curlyW \tond*{\rho, \sigma};
			\]
			\item if $K>0$ then the modified log-Sobolev inequality $\mlsi(2K)$
			\[
			2K\hh^{\phi_1}\tond*{\rho} \leq \ii^{\phi_1}\tond*{\rho}
			\]
			holds for all $\rho \in \pp_*\tond*{\Omega}$.
		\end{itemize}
	\end{proposition}
	Following \cite{erb-maa-2012}, when \eqref{eq:main-ineq} holds for the logarithmic mean, we say that the \emph{entropic curvature} of the Markov chain is bounded from below by $K$, and we use the notation 
	\[
	\rice \geq K.
	\]
	For more consequences of entropic curvature lower bounds we refer the reader to \cite{erb-maa-2012, erb-fat-2018}.
	\subsection{Weight functions and convex Sobolev inequalities}\label{sec:weight-conv}
	In some cases, it is possible to let other $\phi$-entropies take over the role of the relative entropy in the previous section, following \cite{erb-maa-2014} and by choosing an appropriate weight function $\theta$.
	First, we consider a function $\phi\colon \R_{\geq 0} \to \R$
	and correspondingly we make the following
	\begin{assumption} \label{ass:mild-theta-phi}
		The function $\phi\colon\R_{\geq 0} \to \R$ is such that 
		\begin{enumerate}
			\item $\phi$ is continuous and $\phi\in C^2(\R_{>0})$;
			\item $\phi$ is strictly convex.
		\end{enumerate}
		Moreover, the weight function $\theta = \theta_\phi$ defined by
		\begin{equation} \label{eq:theta-phi}
			\theta(s,t) \coloneqq \begin{cases}
				\frac{s-t}{\phi'(s)-\phi'(t)}& \text{ if } s\neq t,
				\\
				\frac{1}{\phi''(s)} & \text{ if } s=t
			\end{cases}
		\end{equation}
		satisfies Assumption \ref{ass: weak assumption theta}.
	\end{assumption}
	A first motivation for defining $\theta$ as in \eqref{eq:theta-phi} comes from the following result due to \cite{erb-maa-2014} (see Theorem $4.8$ therein).
	\begin{proposition}\label{prop:rel-bak-em-esti}
		Suppose that Assumption \ref{ass:mild-theta-phi} is satisfied and that inequality	\eqref{eq:main-ineq} holds for some $K>0$. Then the convex Sobolev inequality  
		\eqref{eq:conv-sob-ineq} holds with constant $2K$ (notation: $\csi_{\phi}(2K)$). 
	\end{proposition}
	For completeness, we provide the proof of this proposition in Appendix \ref{sec:proof-prop-rel-bak-eme-estim}, since it was proved in \cite{erb-maa-2014} under slightly more restrictive assumptions, as a consequence of stronger geodesic convexity results, as explained later in this section. The idea for the proof of Proposition \ref{prop:rel-bak-em-esti} is that, when restricting to the specific choice $\psi \coloneqq \phi'\circ \rho$, inequality \eqref{eq:main-ineq} is equivalent to the second order differential inequality
	\begin{equation}\label{eq:bak-eme-method}
		\frac{\dd^2}{\dd t^2} \Bigg|_{t=0} \hh^\phi(P_t\rho) \geq 2K \ii^\phi(\rho). 
	\end{equation}
	From this, it is standard to deduce the convex Sobolev inequality \eqref{eq:conv-sob-ineq} with constant $2K$, essentially by integration, following what is known as the 
	``Bakry--\'{E}mery argument''.
	Actually, this is exactly the approach used by Conforti in \cite{con-2022}, that is, he uses couplings rates to establish \eqref{eq:bak-eme-method} and subsequently deduces the convex Sobolev inequality $\csi_\phi(2K)$. Since \eqref{eq:bak-eme-method} is a particular case of \eqref{eq:main-ineq} for a specific choice of $\psi$, it is clear that proving \eqref{eq:main-ineq} under Assumption \ref{ass:mild-theta-phi} gives a stronger result as compared to \eqref{eq:bak-eme-method} and is in general more challenging to achieve as one has to deal with two unknown functions ($\rho$ and $\psi$) as opposed to just one (i.e. $\rho$).
	Another difference with the work of Conforti lies in the assumptions on the convex function $\phi$: indeed, our Assumption \ref{ass:mild-theta-phi} requires that $\theta$ is concave. This assumption was present also in \cite{jun-yue-2017} and (as already observed there) implies in particular that $\frac{1}{\phi''}$ is concave, which is a classical assumption for the continuous setting. On the other hand, Conforti does not assume concavity of $\theta$, but requires instead that the function 
	\begin{equation}\label{eq:conf-conv}
		(s,t) \to \tond*{s-t}\cdot \tond*{\phi'(s)-\phi'(t)}
	\end{equation}
	is convex.

	While both assumptions are enough to deduce convex Sobolev inequalities, it is possible to make another more demanding one and deduce stronger consequences from inequality \eqref{eq:main-ineq}.
	\begin{assumption}\label{ass:strong-theta}
		Assumption \ref{ass:mild-theta-phi} is satisfied. Moreover, with $\theta$ as in \eqref{eq:theta-phi}, we have that:
		\begin{itemize}
			\item $\phi \in C^\infty(\R_{>0})$;
			\item $\theta \in C^\infty(\R_{>0}\times\R_{>0})$;
			\item $\theta$ extends to a continuous  function defined on $\R_{\geq 0} \times \R_{\geq 0}$;
			\item $\theta(r,s) \leq \theta(r,t)$ for all $0\leq s\leq t$ and $0\leq r$.
		\end{itemize}
		
	\end{assumption}
	
	If the above is satisfied, then in \cite{erb-maa-2014} the authors showed it is possible to adapt some of the results of Section \ref{sec:log-mean}. Replacing the $\curlyW$ metric with a suitable modified metric $\curlyW_\phi$ (where the new weight function $\theta$ replaces the logarithmic mean), it holds that for any starting density $\rho\in \pp(\Omega)$ the heat flow $t \to P_t \rho$ is the gradient flow of the $\phi$-entropy $\hh^\phi$ in $\tond*{\pp(\Omega), \curlyW_\phi}$. Moreover, as in the previous section, inequality \eqref{eq:main-ineq} si equivalent to $K$-geodesic convexity of $\hh^\phi$ in $\tond*{\pp(\Omega), \curlyW_\phi}$ and the following result holds (cf. Propositions $4.2$, $4.6$ and Theorem $4.8$ in \cite{erb-maa-2014}).
	\begin{proposition}\label{prop:coneqs-phi-entr-low-bound}
		Under Assumption \ref{ass:strong-theta}, suppose that inequality \eqref{eq:main-ineq} holds for some constant $K\in \R$ and for all $\psi \colon \Omega\to \R$, $\rho \in \pp_*(\Omega)$. Then: 
		\begin{itemize}
			\item the  inequality
			\[
			H^{\phi}\tond*{{\rho}} \leq \curlyW_\phi(\rho, \mathbf{1})\sqrt{I^\phi(\rho)}-\frac{K}{2}\curlyW_\phi\tond*{\rho, \mathbf{1}}^2
			\]
			holds for all $\rho \in \pp\tond*{\Omega}$;
			\item for any $\rho, \sigma \in \pp\tond*{\Omega}$
			\[
			\curlyW_\phi\tond*{P_t\rho, P_t\sigma}\leq e^{-Kt} \curlyW_\phi \tond*{\rho, \sigma};
			\]
			\item if $K>0$ then the $\phi$-convex Sobolev inequality $\csi_{\phi}(2K)$
			\[
			2K\hh^{\phi}\tond*{\rho} \leq \ii^{\phi}\tond*{\rho}
			\]
			holds for all $\rho \in \pp_*\tond*{\Omega}$.
		\end{itemize}
	\end{proposition}

	All the functions $\phi_\alpha$ defined in \eqref{eq:beckner-func} satisfy Assumption \ref{ass:strong-theta} (see also \cite[Lemma 16]{jun-yue-2017}): for the corresponding weight function, we use the notation $\theta_\alpha$. Notice in particular that for $\alpha = 1$ we recover the logarithmic mean of the previous subsection, while for $1<\alpha <2$ we have 
	\begin{equation}\label{eq:theta-alpha}
		\theta_\alpha (s,t) = \begin{cases}
			\frac{\alpha-1}{\alpha} \frac{s-t}{s^{\alpha-1}-t^{\alpha-1}} & \text{ if } s\neq t,
			\\
			\frac{1}{\alpha} s^{2-\alpha} & \text{ if } s = t. 
		\end{cases}
	\end{equation}

	\begin{remark}[Case $\alpha =2$]\label{rmk:alpha=2}
		The case $\alpha = 2$ is   particular and should be studied separately. In this case, indeed, the weight function satisfies $\theta_2 \equiv \frac{1}{2}$. Therefore, the quantities $\curlyB(\rho,\psi)$ and $\curlyA(\rho,\psi)$ become independent of $\rho$, which makes establishing inequality \eqref{eq:main-ineq} significantly simpler. Actually, for $\theta = \theta_2$ establishing \eqref{eq:main-ineq} for all $\rho,\psi$ is equivalent to establishing \eqref{eq:bak-eme-method} 
		for all $\rho$, as done by Conforti. Therefore, in this case it is usually possible to establish inequality \eqref{eq:main-ineq} with a better constant than what would happen just under Assumption \ref{ass: weak assumption theta}. In all the examples of Section \ref{sec:applications}, this can be achieved by a simple modification of the arguments after substituting $\theta \equiv \frac{1}{2}$, or alternatively, given the equivalence of \eqref{eq:main-ineq} and \eqref{eq:bak-eme-method}, by just applying the results of \cite{con-2022}.
		For this reasons, for the results of Section \ref{sec:applications} applied to $\theta = \theta_\alpha$ we focus on $\alpha \in [1,2)$ when comparing to \cite{con-2022}.
	\end{remark}

	\subsection{Arithmetic mean and discrete Bakry--\'{E}mery curvature}
	\label{sec:disc-bak-eme}
	If $\theta$ is the arithmetic mean, it is well known  that inequality \eqref{eq:main-ineq} is equivalent to a lower bound for the discrete Bakry--\'{E}mery curvature (for example, it was already observed in \cite{maa-2017}). For completeness, and since we did not find a detailed proof in the literature, we recall the definitions and include a proof of this fact.
	In analogy with the classical setting discussed in great detail in \cite{bak-gen-led-2014}, for $f,g\colon \Omega \to \R$ define
	\begin{align*}
		\Gamma(f,g)(\eta) \coloneqq  \frac{1}{2} \sum_{\sigma\in G} c(\eta,\sigma) \tond*{f(\sigma \eta)-f(\eta)}\tond*{g(\sigma \eta) - g(\eta)},
	\end{align*}
	$\Gamma(f) \coloneqq \Gamma(f,f)$ and
	\begin{align*}
		\Gamma_2(f) \coloneqq \frac{1}{2} L \Gamma(f) - \Gamma(f, Lf).
	\end{align*}
	\begin{definition}[\cite{Sch-1999}]
		We say that the curvature condition $\cd(K,\infty)$ is satisfied if for all $f\colon \Omega \to \R$
		\begin{equation}\label{eq:CDK}
			\Gamma_2(f) \geq K \Gamma(f).	
		\end{equation}
	\end{definition}
	
	\begin{proposition}
		Suppose that $\theta$ is the arithmetic mean. Then for any $K\in\R$  inequality \eqref{eq:main-ineq} holds if and only if  $\cd(K, \infty)$ holds.
	\end{proposition}
	\begin{proof}
		Notice that, using reversibility,
		\begin{align*}
			\curlyA(\rho,\psi) &= \frac{1}{4}\sum_{(\eta,\sigma)\in S} m(\eta) c(\eta, \sigma) \tond*{\rho(\eta)+\rho(\sigma \eta)} \quadr*{\psi(\eta)-\psi(\sigma \eta)}^2
			\\
			& =  \frac{1}{2}\sum_{\eta\in \Omega} m(\eta) \rho(\eta)\quadr*{\sum_{\sigma\in G} c(\eta, \sigma) \quadr{\psi(\eta)-\psi(\sigma \eta)}^2}
			\\
			& = \sum_{\eta\in \Omega} m(\eta)\rho(\eta) \Gamma(\psi)(\eta).
		\end{align*}
		
		Moreover, using reversibility multiple times (cf. also \eqref{eqn: reversibility with moves}),
		\begin{align*}
			\curlyC(\rho, \psi) &= \frac{1}{8} \sum_{(\eta,\sigma)\in S}m(\eta) c(\eta,\sigma) \tond*{L\rho(\eta)+L\rho(\sigma\eta)}\quadr*{\psi(\eta) - \psi(\sigma \eta)}^2 
			\\
			& = \frac{1}{4}\sum_{(\eta,\sigma)\in S}m(\eta) c(\eta,\sigma) {L\rho(\eta)}\quadr*{\psi(\eta) - \psi(\sigma \eta)}^2 
			\\
			& = \frac{1}{4}\sum_{\substack{\eta\in \Omega\\ \sigma,\gamma\in G}}m(\eta) c(\eta,\sigma)c(\eta,\gamma) \tond*{\rho(\gamma\eta) - \rho( \eta)}\quadr*{\psi(\eta) - \psi(\sigma \eta)}^2 
			\\
			& = \frac{1}{4}\sum_{\substack{\eta\in \Omega\\ \sigma,\gamma\in G}}m(\eta) c(\gamma\eta,\sigma)c(\eta,\gamma) \rho(\eta) \quadr*{\psi(\gamma\eta) - \psi(\sigma \gamma \eta)}^2
			\\
			& -\frac{1}{4}\sum_{\substack{\eta\in \Omega\\ \sigma,\gamma\in G}}m(\eta) c(\eta,\sigma)c(\eta,\gamma) \rho(\eta) \quadr*{\psi(\eta) - \psi(\sigma  \eta)}^2
			\\
			& = \frac{1}{4}\sum_{\eta\in \Omega}m(\eta)\rho(\eta) \sum_{\gamma\in G}c(\eta,\gamma)
			\sum_{\sigma\in G} \Big\{ c(\gamma\eta,\sigma)  \quadr*{\psi(\gamma\eta) - \psi(\sigma \gamma \eta)}^2
			\\
			& -c(\eta,\sigma)\quadr*{\psi(\eta)-\psi(\sigma\eta)}^2\Big\}
			\\
			& = \sum_{\eta\in \Omega} m(\eta)\rho(\eta) \frac{1}{2} L \Gamma \psi(\eta),
		\end{align*}
		
		and
		\begin{align*}
			\curlyD(\rho,\psi) & = \frac{1}{4}\sum_{(\eta,\sigma)\in S}m(\eta) c(\eta, \sigma) \tond*{\rho(\eta)+\rho(\sigma\eta)} \tond*{\psi(\eta)- \psi(\sigma \eta)} \tond*{L\psi(\eta)-L\psi(\sigma \eta)}
			\\
			& =
			\frac{1}{2}\sum_{\eta\in\Omega} m(\eta) \rho(\eta) \sum_{\sigma\in G} c(\eta,\sigma)\tond*{\psi(\eta)- \psi(\sigma \eta)} \tond*{L\psi(\eta)-L\psi(\sigma \eta)}
			\\
			& =
			\sum_{\eta\in \Omega} m(\eta)\rho(\eta) \Gamma(\psi,L\psi)(\eta).
		\end{align*}
		Therefore, \eqref{eq:main-ineq} is equivalent to
		\begin{align}\label{eq:aritm-bak-eme}
			\sum_{\eta\in \Omega}m(\eta)\rho(\eta) \Gamma_2 \psi(\eta) \geq K \sum_{\eta\in \Omega}m(\eta)\rho(\eta) \Gamma \psi(\eta).
		\end{align} 
		From this, it is clear that $\cd\tond*{K,\infty}$ implies \eqref{eq:main-ineq} by choosing $f = \psi$. Conversely, choosing $\rho = \frac{d \delta_\eta}{dm}$ to be the density of a Dirac and $\psi = f$ in \eqref{eq:aritm-bak-eme} gives the converse implication.
	\end{proof}
	For more details about and consequences of the discrete Bakry--\'{E}mery curvature we refer the reader to \cite{fat-shu-2018} and the references therein.

	\section{Coupling rates and curvature lower bound}\label{sec:coup-rates}
	Coupling rates were introduced by Conforti in \cite{con-2022} as a tool to establish convex Sobolev inequalities. They are a modification of the usual notion of coupling, and they apply to continuous time Markov chains. Roughly speaking, they are a way of letting the generator $L$ act at the same time at two different states, in a way that is consistent with equation \eqref{eq:generator} when one looks separately at the two states. More precisely, for any pair of different states $\eta, \bar{\eta} \in \Omega$, we consider  coupling rates between them in the form of a function $\ccpl(\eta,\bar\eta,\cdot,\cdot)\colon G^*\times G^* \to \R_{\geq0}$    such that
	\begin{align*}
		\forall \gamma \in G, \quad \sum_{\bar{\gamma}\in G^*} \ccpl(\eta,\bar \eta, \gamma,\bar{\gamma}) &= c(\eta,\gamma),     \\
		\forall \bar{\gamma} \in G, \quad	\sum_{\gamma\in G^*} \ccpl(\eta,\bar \eta, \gamma,\bar{\gamma}) &= c(\bar \eta,\bar{\gamma}).
	\end{align*} 
	It can be seen easily seen that, for any fixed states $\eta\neq \bar \eta$,  coupling rates between them always exist; for example, one can consider the ``product coupling rates'', constructed as follows. Suppose without loss of generality (otherwise, exchange the role of $\eta,\bar \eta$) that 
	\[
	\sum_{\sigma \in G} c(\eta, \sigma) \le \sum_{\sigma\in G} c(\bar \eta, \sigma) \eqqcolon M,
	\]
	and notice $0<M<\infty$ (since $G$ is finite and the chain is irreducible). Then, 
	set $c(\eta, e) = M-\sum_{\sigma\in G}c(\eta,\sigma)$ and $c(\bar \eta,e) =0$, 
	where $e$ is the null move. Finally, for $\gamma,\bar\gamma \in G^*$, define
	\[
	\ccpl(\eta,\bar \eta, \gamma,\bar{\gamma}) = \frac{1}{M} c(\eta, \gamma) c(\bar \eta, \bar\gamma),
	\]
	which is easily seen to define appropriate coupling rates between $\eta$ and $\bar \eta$.

	From the definition of coupling rates, it follows immediately that one can jointly express the action of the generator \eqref{eq:generator} on a function $\psi$ at the states $\eta$ and $\bar \eta$ as follows:
	\begin{equation}\label{eq:coup-rates-action-gener}
		\begin{split}
			L\psi (\eta) =  \sum_{\gamma, \bar{\gamma}\in G^*} \ccpl(\eta,\bar \eta, \gamma,\bar{\gamma})\tond*{\psi(\gamma \eta)-\psi(\eta)},     \\
			L\psi(\bar \eta) = \sum_{\gamma, \bar \gamma\in G^*} \ccpl(\eta,\bar \eta, \gamma,\bar{\gamma}) \tond*{\psi(\bar\gamma \bar \eta)-\psi(\bar\eta)}.
		\end{split} 
	\end{equation}
	In \cite{con-2022}, Conforti showed that coupling rates are useful for organizing the terms appearing in the inequality \eqref{eq:bak-eme-method}, and thus (if one manages to establish it), in proving convex Sobolev inequalities via the Bakry--\'{E}mery argument.
	Heuristically, it turns out it is convenient to consider not arbitrary coupling rates, but rather ``contractive'' ones. Informally, this means that, for neigbouring states $\eta \neq \sigma\eta$ with $\eta\in \Omega,\sigma\in G$, the coupling rates $\ccpl(\eta,\sigma \eta, \gamma,\bar\gamma)$ between them are such that  ``as often as possible'' $\gamma \eta = \bar{\gamma} \sigma \eta$ (and, in particular, a fruitful choice is $(\gamma,\bar{\gamma}) = (\sigma, e)$ or $(e, \sigma^{-1})$).
	Indeed, when this is achieved, some terms cancellations going into the right direction occur when studying inequality \eqref{eq:bak-eme-method}.
 
	In the rest of this section, we show that similar considerations also hold when studying the stronger inequality \eqref{eq:main-ineq}.
	In particular, we will derive a lower bound for $\curlyB(\rho,\psi)$ using coupling rates under only Assumption \ref{ass: weak assumption theta} on $\theta$: this gives a sufficient condition for establishing the  inequality
	$\curlyB(\rho,\psi) \geq K \curlyA (\rho,\psi)$.
 In general, this is a more challenging situation compared to \eqref{eq:bak-eme-method}, since now we are dealing with two unknowns $(\rho,\psi)$ as opposed to just $\rho$, and we also have an additional non linear weight function $\theta$ to deal with; moreover, as discussed in Section \ref{sec:weight-conv}, inequality \eqref{eq:bak-eme-method}  corresponds to a particular case of \eqref{eq:main-ineq} with the choice $\psi = \phi'\circ \rho$ and $\theta$ as in Assumption \ref{ass:mild-theta-phi}. Next, we will conclude the section by discussing heuristically how contractions in the coupling rates can help proving the inequality \eqref{eq:main-ineq} too, similarly to what happened in \cite{con-2022}.
	
	We now proceed to show how arbitrary coupling rates can help rewrite the main inequality \eqref{eq:main-ineq} in a convenient way, by organizing the involved terms.
	Notice first that we can write, using coupling rates as in \eqref{eq:coup-rates-action-gener},
	\begin{equation}
		\label{eqn: new lower bound for C with rates}
		\begin{split}
			&\nabla \theta \tond*{\rho(\eta), \rho(\sigma \eta)} \cdot \begin{pmatrix}
				L \rho(\eta)\\
				L \rho(\sigma \eta)
			\end{pmatrix} \\
			=& 
			\nabla \theta \tond*{\rho(\eta), \rho(\sigma \eta)} \cdot\quadr*{\sum_{\gamma,\bar{\gamma}\in G^*} \ccpl(\eta,\sigma \eta, \gamma,\bar{\gamma}) \begin{pmatrix}
					\rho(\gamma \eta)-\rho(\eta)\\
					\rho(\bar{\gamma}\sigma \eta) - \rho(\sigma \eta)
			\end{pmatrix}} \\
			=& \sum_{\gamma,\bar{\gamma}\in G^*} \ccpl(\eta,\sigma \eta, \gamma, \bar{\gamma}) \nabla \theta(\rho(\eta),\rho(\sigma \eta)) \cdot\quadr*{\begin{pmatrix}
					\rho(\gamma \eta)\\
					\rho(\bar{\gamma}\sigma \eta) 
				\end{pmatrix} -
				\begin{pmatrix}
					\rho( \eta)\\
					\rho(\sigma \eta) 
			\end{pmatrix}}\\
			\geq& \sum_{\gamma, \bar{\gamma}\in G^*}\ccpl(\eta,\sigma \eta,\gamma,\bar{\gamma}) \quadr*{\theta\tond*{\rho(\gamma \eta),\rho(\bar{\gamma}\sigma \eta)} - \theta\tond*{\rho( \eta),\rho(\sigma \eta)}},
		\end{split}
	\end{equation}
	where we used concavity of $\theta$ in the last line.
	Therefore, for $\curlyC(\rho,\psi)$ we have the lower bound
	\begin{align*}
		\curlyC(\rho,\psi) &\geq
		\frac{1}{4}\sum_{(\eta,\sigma)\in S} \sum_{\gamma,\bar{\gamma}\in G^*} m(\eta)c(\eta, \sigma)\,\ccpl(\eta,\sigma \eta,\gamma,\bar{\gamma}) \theta \tond*{\rho(\gamma \eta),\rho(\bar{\gamma}\sigma \eta)} \quadr*{\psi(\eta)-\psi(\sigma \eta)}^2
		\\	 
		& -\frac{1}{4}\sum_{(\eta,\sigma)\in S} \sum_{\gamma,\bar{\gamma}\in G^*} m(\eta)c(\eta, \sigma) \ccpl(\eta, \sigma\eta,\gamma,\bar{\gamma}) \theta\tond*{\rho(\eta),\rho(\sigma \eta)} \quadr*{\psi(\eta)-\psi(\sigma \eta)}^2.
	\end{align*}

	As for the term $\curlyD(\rho, \psi)$, we can write
	\begin{align*}
		\curlyD(\rho,\psi) 
		= \,& \frac{1}{2}\sum_{(\eta,\sigma)\in S}\sum_{\gamma,\bar{\gamma}\in G^*} m(\eta) c(\eta, \sigma) \ccpl(\eta, \sigma\eta,\gamma,\bar{\gamma}) \theta(\rho(\eta),\rho(\sigma \eta)) 
		\\
		\cdot\, &\graf*{ \tond*{\psi(\eta)-\psi(\sigma \eta)} \tond*{\psi(\gamma \eta)-\psi(\bar{\gamma}\sigma \eta)}-
		 \tond*{\psi(\eta)-\psi(\sigma \eta)}^2}.
	\end{align*}

	Combining the bound for $\curlyC$  and the expression for $\curlyD$ we derive the following:
	\begin{lemma}
		\label{lem: lower bound for B with coupling rates} Let $\theta$ be a weight function satisfying Assumption \ref{ass: weak assumption theta}.
		We have 
		\begin{equation}
			\label{eqn: lower bound on B with couplings}
			\curlyB(\rho,\psi) \geq \frac{1}{4}\sum_{(\eta,\sigma)\in S} \sum_{\gamma,\bar{\gamma}\in G^*} m(\eta)c(\eta, \sigma)\,\ccpl(\eta, \sigma\eta,\gamma,\bar{\gamma}) J(\eta,\sigma,\gamma,\bar{\gamma})
		\end{equation}
		for all $\rho\colon \Omega\to \R_{> 0}$ and $\psi\colon \Omega \to \R$, where we define the function $J\colon \Omega \times G^*\times G^*\times G^* \to \R$ by
		\begin{align*}
			J(\eta,\sigma,\gamma,\bar{\gamma}) &\coloneqq  \graf*{\theta \tond*{\rho(\gamma \eta),\rho(\bar{\gamma}\sigma \eta)} + \theta \tond*{\rho( \eta),\rho(\sigma \eta)}} \quadr*{\psi(\eta)-\psi(\sigma \eta)}^2 
			\\
			&- 2 \theta(\rho(\eta),\rho(\sigma \eta)) \tond*{\psi(\eta)-\psi(\sigma \eta)} \tond*{\psi(\gamma \eta)-\psi(\bar{\gamma}\sigma \eta)}.
		\end{align*}
	\end{lemma}
 
	It is also  convenient to define the function $I\colon \Omega \times G^*\times G^*\times G^* \to \R$ by 
	\begin{align*}
		I(\eta,\sigma,\gamma,\bar{\gamma}) &= I_1(\eta,\sigma,\gamma,\bar{\gamma}) - I_2(\eta,\sigma,\gamma,\bar{\gamma}),\\
		I_1(\eta,\sigma,\gamma,\bar{\gamma}) & = \theta \tond*{\rho(\gamma \eta),\rho(\bar{\gamma}\sigma \eta)} \quadr*{\psi(\eta)-\psi(\sigma \eta)}^2 \geq 0, \\
		I_2(\eta,\sigma,\gamma,\bar{\gamma}) &= \theta(\rho(\eta),\rho(\sigma \eta)) \quadr*{\psi(\gamma \eta)-\psi(\bar{\gamma}\sigma \eta)}^2 \geq 0.
	\end{align*}
	Notice  that we have
	\begin{equation}
		\begin{split}
			\label{eqn: J vs I}
			J(\eta,\sigma, \gamma,\bar{\gamma}) &= I(\eta,\sigma, \gamma,\bar{\gamma})+{\theta}\tond*{\rho(\eta),\rho(\sigma \eta)} \quadr*{\psi(\eta)-\psi(\sigma\eta)-\psi({\gamma} \eta)+\psi(\bar{\gamma}\sigma \eta)}^2 
			\\
			&\geq I(\eta,\sigma,\gamma,\bar{\gamma}).
		\end{split}
	\end{equation}
	At this point, we can explain at least heuristically why it is useful to consider especially \emph{contractive} coupling rates.
	In view of Lemma \ref{lem: lower bound for B with coupling rates}, to establish inequality \eqref{eq:main-ineq} it suffices to prove that
	for some coupling rates 
	\begin{equation}\label{eq:suff-cond-main-ineq}
		\begin{split}
			&\frac{1}{2}\sum_{(\eta,\sigma)\in S}\sum_{\gamma,\bar{\gamma}\in G^*} m(\eta)c(\eta, \sigma)\,\ccpl(\eta, \sigma\eta,\gamma,\bar{\gamma}) J(\eta,\sigma,\gamma,\bar{\gamma})
			\\
			\ge &\, K \,\sum_{(\eta,\sigma)\in S} m(\eta) c(\eta, \sigma) \theta(\rho(\eta), \rho(\sigma \eta)) \quadr*{\psi(\sigma \eta)-\psi(\eta)}^2
			=  2K \, \curlyA(\rho,\psi).
		\end{split}
	\end{equation}
	Notice first of all that whenever $\gamma\eta = \bar{\gamma}\sigma \eta$  the second term in the definition of $J(\eta, \sigma \eta, \gamma, \bar{\gamma})$ is $0$, so $J$ is non-negative, suggesting a first lower bound for $\curlyB(\rho,\psi)$. More precisely, when $\gamma\eta = \bar{\gamma}\sigma \eta$, looking at the corresponding terms in the left-hand-side of  the inequality \eqref{eq:suff-cond-main-ineq}, we see that 
	\begin{equation}
		\label{eq:coupling-heuristic-contraction}
		\begin{split}
			& {\ccpl(\eta, \sigma\eta,\gamma,\bar{\gamma}) J(\eta, \sigma,\gamma, \bar\gamma)}
			\\
			=  &\, {\ccpl(\eta, \sigma\eta,\gamma,\bar \gamma)  } \graf*{\theta(\rho(\gamma\eta), \rho(\bar\gamma\sigma\eta))+\theta(\rho(\eta),\rho(\sigma\eta)) }\quadr*{\psi(\eta)-\psi(\sigma\eta)}^2
			\\
			\geq & \, {\ccpl(\eta, \sigma\eta,\gamma,\bar \gamma)  } {\theta(\rho(\eta),\rho(\sigma\eta)) }\quadr*{\psi(\eta)-\psi(\sigma\eta)}^2.
		\end{split}
	\end{equation}
	Hence, we recognise some terms appearing in the sum in the right-and-side of \eqref{eq:suff-cond-main-ineq} defining $\curlyA(\rho,\psi)$, multiplied by the factor $\ccpl(\eta, \sigma\eta,\gamma,\bar \gamma)$: therefore, if we have a uniform positive lower bound for 
	\begin{equation}\label{eq:inf-meeting-rates}
		\inf_{(\eta,\sigma)\in S} {\sum_{\substack{\gamma,\bar\gamma\in G^*\\ \gamma\eta=\bar\gamma \sigma \eta}}\ccpl(\eta, \sigma\eta,\gamma,\bar\gamma) } > 0,
	\end{equation}
	we are in a good position to prove the inequality \eqref{eq:suff-cond-main-ineq}, provided that 
	we can also accomplish the non-trivial task of dealing with the other terms appearing in the left-hand-side of \eqref{eq:suff-cond-main-ineq} (corresponding to the pairs of moves $(\gamma,\bar \gamma)$ not realising a contraction). This is indeed the general strategy that we will use in Section \ref{sec:applications},  where we analyse specific classes of Markov chains.
	
	A second point we wish to make is that sometimes, depending also on the weight function $\theta$, we can improve the bounds obtained with the strategy described before.  
	The first observation is that in \eqref{eq:coupling-heuristic-contraction} we have thrown away some non-negative terms, corresponding to
	\begin{equation}
		\label{eq:terms-thrown-away-general}
		\ccpl(\eta, \sigma\eta,\gamma,\bar \gamma)   \theta(\rho(\gamma\eta), \rho(\bar\gamma\sigma\eta)) \quadr*{\psi(\eta)-\psi(\sigma\eta)}^2.
	\end{equation}
	We now restrict our attention to two particular pairs of moves that are ``contractive'', given respectively 
	by $(e,\sigma^{-1})$ and $(\sigma, e)$. In this case, the terms in \eqref{eq:terms-thrown-away-general} sum up to
	\[
	\graf*{\ccpl(\eta, \sigma\eta,e,\sigma^{-1})\theta(\rho(\eta),\rho(\eta))  + \ccpl(\eta, \sigma\eta, \sigma, e)\theta(\rho(\sigma\eta),\rho(\sigma\eta))}\quadr*{\psi(\eta)-\psi(\sigma\eta)}^2.
	\]
	These terms could also be related to the ones appearing in the definition of $\curlyA(\rho,\psi)$, if we knew that for some constant $M_\theta$
	\[ \theta(\rho(\eta),\rho(\eta))+\theta(\rho(\sigma\eta),\rho(\sigma\eta)) \ge 2M_\theta \theta(\rho(\eta),\rho(\sigma\eta)), 
	\]
	and if we had a uniform positive lower bound 
	\begin{equation}\label{eq:inf-meeting-rates-stronger}
		\inf_{(\eta,\sigma)\in S} \min\graf*{\ccpl(\eta, \sigma\eta,\sigma,e), \ccpl(\eta, \sigma\eta, e, \sigma^{-1})} > 0,
	\end{equation}
	similarly to \eqref{eq:inf-meeting-rates}.
	For this reason, for a given weight function $\theta$ satisfying Assumption \ref{ass: weak assumption theta}, it is natural to define the quantity 
	\begin{equation}\label{eq:beta-theta}
		M_\theta \coloneqq \inf_{\substack{s,t> 0: \\ \theta(s,t)>0}} \frac{\theta(s,s)+\theta(t,t)}{2 \theta(s,t)}  \in [0,1],
	\end{equation}
	so that  for all $s,t\geq 0$
	\[
	2M_\theta \theta(s,t) \leq {\theta(s,s)+\theta(t,t)}.
	\]
	By choosing $s=t$ we can see that $M_\theta \leq 1$.
	The next proposition, whose proof is given in Appendix \ref{sec:comp-beta-theta}, provides the value of $M_\theta$ for the explicit examples of $\theta$ considered in Section \ref{sec:gen-ineq}.
	\begin{proposition}\label{prop:comp-beta-theta}
		\begin{itemize}
			\item For $\alpha \in [1,2]$  and $\theta_\alpha$ as in equations \eqref{eq:log-mean}, \eqref{eq:theta-alpha},  we have
			\[
			M_{\theta_\alpha} = \begin{cases}
				1 & \text{ if } \alpha \in \quadr*{1,\frac{3}{2}} \text{ or } \alpha = 2;
				\\
				\frac{1}{2(\alpha-1)} & \text{ if } \alpha \in \tond*{\frac{3}{2},2}.
			\end{cases}
			\]
			\item For the arithmetic mean we have $M_\theta = 1$. 
		\end{itemize}
	\end{proposition}

	As a concluding remark for this section, we emphasize that, while the method described in this section potentially applies to a wide variety of settings, in general it seems that some extra assumptions on the Markov chains are helpful to get the desired conclusions. In particular, reversibility of the model and an underlying symmetry of the structure of the Markov chain can help obtain useful terms cancellations to deal with the ``non-contractive'' pairs of moves $(\gamma,\bar\gamma)$ in the left-hand-side of \eqref{eq:suff-cond-main-ineq}.

	\section{Applications}\label{sec:applications}
	In this section, we apply Lemma \ref{lem: lower bound for B with coupling rates} to establish the general inequality of interest \eqref{eq:main-ineq} for most of the examples considered in \cite{con-2022}, under just Assumption \ref{ass: weak assumption theta}. In particular, Section \ref{sec: Glauber dynamics}, \ref{sec: Bernoulli Laplace model}, \ref{sec: Hardcore model} and \ref{sec: Interacting random walks} corresponds to Section $4$, $5.1$, $5.2$ and $3$ of \cite{con-2022} respectively. Not surprisingly, the proofs are similar to the ones of Conforti, and in all these examples the considered contractive coupling rates are the ones constructed in \cite{con-2022}. 
	The case of the interacting random walks of Section \ref{sec: Interacting random walks} is the only one where an additional assumption is present  compared to \cite{con-2022}: moreover, as done by Conforti, in that section a localization procedure is used to deal with the infinite cardinality of the state space.

	\subsection{Glauber dynamics}
	\label{sec: Glauber dynamics}
	We work in the setting of Section $4$ of \cite{con-2022} (i.e. Glauber dynamics) and we use the same notation, which we briefly recall.
	The state space is a finite set $\Omega$. We assume that  $\sigma = \sigma^{-1}$ and $\sigma \gamma \eta= \gamma \sigma \eta$  for all moves $\sigma,\gamma \in G$ and states $\eta \in \Omega$. Given an inverse temperature parameter $\beta>0$ and an Hamiltonian function $H\colon \Omega \to \R$, the rates are defined by
	\[
	c(\eta,\sigma) = \exp \tond*{-\frac{\beta}{2}\nabla_\sigma H(\eta)},
	\] 
	where we recall the notation $\nabla_\sigma H(\eta) = H(\sigma \eta) - H(\eta)$ for the discrete gradient.
	The reversible measure is then the Gibbs measure 
	\[
	m(\eta) = \frac{1}{Z_\beta} \exp\tond*{-\beta H(\eta)}
	\]
	where  $Z_\beta>0$ is the appropriate normalization constant.
	Finally we make the key assumption that $\kappa(\eta,\sigma) \geq 0$ for all states $\eta\in \Omega$ and moves $\sigma\in G$, where we define
	\[
	\kappa(\eta,\sigma) \coloneqq c(\sigma \eta, \sigma) - \sum_{\substack{\gamma\in G\\ \gamma \neq \sigma}} \max\graf*{-\nabla_\sigma c(\eta, \gamma),0}.
	\]
	This assumption is crucial for the construction of appropriate contractive coupling rates, for which we will apply Lemma \ref{lem: lower bound for B with coupling rates}.
	We also define the quantities
	\[
	\kappa_* \coloneqq \inf_{\substack{\eta\in \Omega\\ \sigma\in G}} \kappa(\eta,\sigma) + \kappa(\sigma\eta, \sigma), \qquad \overline{\kappa_*} \coloneqq \inf_{\substack{\eta\in \Omega\\ \sigma\in G}} \kappa(\eta,\sigma),
	\]
	which correspond respectively to the infimums in \eqref{eq:inf-meeting-rates} and \eqref{eq:inf-meeting-rates-stronger}. 
	Notice that $2\overline{\kappa_*} \leq \kappa_*$. 
	
	\begin{theorem}
		\label{thm: entropic curvature Glauber dynamics}
		With the previous notation, suppose that for all $\eta\in \Omega$ and $\sigma, \gamma\in G$ we have
		$\sigma\gamma=\gamma\sigma$, $\sigma = \sigma^{-1}$ and $\kappa(\eta,\sigma) \geq 0$. Let $\theta$ be a weight function satisfying Assumption \ref{ass: weak assumption theta}. Then the inequality \eqref{eq:main-ineq} holds with constant
		\[
		K =   {M_\theta\overline{\kappa_*}+\frac{\kappa_*}{2}}.
		\]
	\end{theorem}
	
	\begin{remark}[Comparison with \cite{con-2022}]
		In \cite[Thm. 4.1]{con-2022}, under the same assumptions on the model, Conforti establishes inequality \eqref{eq:bak-eme-method} and thus $\csi_\phi(2K)$ with constant $K$ equal to
		\begin{itemize}
			\item $\frac{\kappa_*}{2}$ for general convex $\phi$ satisfying convexity of \eqref{eq:conf-conv};
			\item $\frac{\kappa_*}{2} + {\overline{\kappa_*}}$ for $\phi = \phi_1$ (thus $\mlsi(\kappa_* + 2\overline{\kappa_*})$);
			\item $\frac{\alpha}{2} \kappa_*$ for $\phi = \phi_\alpha$ with $\alpha \in (1,2]$.		\end{itemize}
		Thus, by Proposition \ref{prop:comp-beta-theta} and by the discussion in Section \ref{sec:gen-ineq} (i.e. recalling for example Proposition \ref{prop:coneqs-entr-low-bound} and that \eqref{eq:bak-eme-method} is  particular case of \eqref{eq:main-ineq}), we obtain a stronger result for the case $\theta = \theta_1$ and complementary results for other choices of $\theta$.
	\end{remark}
	
	\subsubsection{Proof of Theorem \ref{thm: entropic curvature Glauber dynamics}}
	As done in \cite{con-2022}, we define
	\begin{align*}
		\Upsilon^<(\eta) = \graf*{(\sigma,\gamma) \in G\times G \mid \sigma \neq \gamma, \nabla_\sigma c(\eta,\gamma) <0},\\
		\Upsilon^>(\eta) = \graf*{(\sigma,\gamma) \in G\times G  \mid \sigma \neq \gamma, \nabla_\sigma c(\eta,\gamma) >0},\\
		\Upsilon^=(\eta) = \graf*{(\sigma,\gamma)\in G\times G  \mid \sigma \neq \gamma, \nabla_\sigma c(\eta,\gamma) =0},		
	\end{align*}
	where we recall the notation
	\begin{align*}
	\nabla_\sigma c(\eta,\gamma) &= c(\sigma\eta, \gamma)-c(\eta,\gamma) 
	\\
	&= \exp\tond*{-\frac{\beta}{2}\quadr*{H(\gamma\sigma \eta)-H(\sigma\eta)}} - \exp\tond*{-\frac{\beta}{2}\quadr*{H(\gamma\eta)-H(\eta)}}.
	\end{align*}
	We then define the same coupling rates: for $\eta\in \Omega, \sigma \in G$ set
	\[
	\ccpl(\eta,\sigma \eta, \gamma, \bar{\gamma}) = \left\{
	\begin{array}{ll}
		\min \graf{c(\sigma \eta,\gamma), c(\eta, \gamma)} & \mbox{ if } \gamma = \bar{\gamma} \mbox{ and } \sigma \neq \gamma, \gamma \in G, \\
		- \nabla_\sigma c(\eta,\gamma) & \mbox{ if } \bar{\gamma} = \sigma \mbox{ and } (\sigma,
		\gamma) \in \Upsilon^<(\eta),\\  
		\nabla_\sigma c(\eta,\bar{\gamma}) & \mbox{ if } {\gamma} = \sigma \mbox{ and } (\sigma, \bar{\gamma}) \in \Upsilon^>(\eta),\\
		\kappa(\sigma \eta,\sigma) & \mbox{ if } \gamma=\sigma, \bar{\gamma} = e,\\
		\kappa(\eta,\sigma) & \mbox{ if }
		{\gamma} = e, \bar{\gamma}=\sigma ,\\
		0 & \mbox{ otherwise}.		 
	\end{array}
	\right.
	\]
	Notice that these are indeed admissible coupling rates between $\eta$ and $\sigma \eta$, since by assumption $\kappa(\sigma\eta,\sigma), \kappa(\eta,\sigma)\ge 0$.
	With these coupling rates and using Lemma \ref{lem: lower bound for B with coupling rates} and \eqref{eqn: J vs I} we can write $\curlyB(\rho,\psi) \geq \frac{1}{4}(A+B+C+D)$ with
	\begin{align*}
		A &= \sum_{\substack{\eta\in \Omega, \sigma,\gamma \in G, \\ \sigma\neq \gamma}} m(\eta) c(\eta, \sigma) \min \graf*{c(\eta,\gamma),c(\sigma \eta, \gamma)} I(\eta,\sigma,\gamma,\gamma) ,\\
		B &= -\sum_{\substack{\eta\in \Omega,\\(\sigma,\gamma)\in \Upsilon^<(\eta)}} \metacsigma \nabla_\sigma c(\eta,\gamma) I(\eta,\sigma, \gamma,\sigma) ,\\
		C &= \sum_{\substack{\eta\in\Omega,\\(\sigma,\bar{\gamma})\in \Upsilon^>(\eta)}} \metacsigma \nabla_\sigma c(\eta,\bar{\gamma}) I(\eta,\sigma, \sigma,\bar{\gamma}) ,\\
		D &= \sum_{\eta\in\Omega,\sigma\in G} \metacsigma \tond*{\kappa(\sigma \eta,\sigma) J(\eta,\sigma,\sigma,e) + \kappa(\eta,\sigma) J(\eta,\sigma,e,\sigma)} .
	\end{align*}
	
	We show below that $A=B=C=0$ and that $D\geq (4M_\theta\overline{\kappa_*}+2\kappa_*) \curlyA(\rho,\psi)$, which concludes the proof of the theorem. 
	It is useful to have an auxiliary lemma:
	\begin{lemma}
		\label{lem: auxiliary lemma glauber dynamics conforti}
		For all $\eta \in \Omega$ and $\sigma, \gamma\in G$ with $\sigma \neq \gamma$ the following hold:
		\begin{enumerate}
			\item $c(\eta,\sigma) \nabla_\sigma c(\eta,\gamma) = c(\eta,\gamma) \nabla_\gamma c(\eta,\sigma)$.
			\item $c(\eta,\sigma) c(\sigma \eta, \gamma) = c(\eta,\gamma) c(\gamma \eta, \sigma)$.
			\item $\nabla_\sigma c(\sigma \eta,\gamma) = -\nabla_\sigma c(\eta,\gamma)$.
			\item $(\sigma, \gamma) \in \Upsilon^<(\eta) \iff (\gamma,\sigma)\in \Upsilon^<(\eta)$.
			\item $(\sigma, \gamma) \in \Upsilon^>(\gamma \eta) \iff (\sigma,\gamma)\in \Upsilon^<(\eta)$.
			\item $(\sigma, \gamma) \in \Upsilon^>(\sigma \eta) \iff (\sigma,\gamma)\in \Upsilon^<(\eta)$.
			\item $(\sigma, \gamma) \in \Upsilon^=(\gamma \eta) \iff (\sigma,\gamma)\in \Upsilon^=(\eta)$.
			\item $I(\eta, \sigma, \gamma, \sigma) = - I(\eta,\gamma,\sigma, \gamma)$.
			\item $I(\sigma \eta, \sigma, \sigma, \gamma) = - I(\gamma \eta, \gamma, \gamma, \sigma) $.
			\item $I(\eta, \sigma, \gamma, \gamma) = - I(\gamma\eta,\sigma, \gamma, \gamma)$.
		\end{enumerate}
	\end{lemma}
	\begin{proof}[Proof of Lemma]
		Statements 1--7 were already observed in the proof of \cite[Thm. 4.1]{con-2022} and are easy to check, while statements 8--10 are immediate from the definitions.
	\end{proof}
	
	\subparagraph*{Term D}
	We have
	\begin{align*}
		J(\eta,\sigma,\sigma,e) &= \graf*{\theta\tond*{\rho(\sigma \eta), \rho(\sigma \eta) }+ \theta(\rho(\eta),\rho(\sigma \eta)) }\quadr*{\psi(\eta)-\psi(\sigma \eta)}^2 \\
		J(\eta,\sigma,e,\sigma) &= \graf*{\theta\tond*{\rho( \eta),\rho( \eta)}+ \theta(\rho(\eta),\rho(\sigma \eta))} \quadr*{\psi(\eta)-\psi(\sigma \eta)}^2
	\end{align*}
	and so
	\begin{align*}
		&\kappa(\sigma \eta,\sigma) J(\eta,\sigma,\sigma,e) + \kappa(\eta,\sigma) J(\eta,\sigma,e,\sigma)  
		\\
		\geq & \graf*{\overline{\kappa_*} \quadr*{ \theta\tond*{\rho(\sigma \eta), \rho(\sigma \eta) } +\theta\tond*{\rho( \eta),\rho( \eta)}}
			+ \kappa_* \theta \tond*{\rho(\eta),\rho(\sigma \eta)}}\quadr*{\psi(\eta)-\psi(\sigma \eta)}^2
		\\ 
		\geq & (2M_\theta\overline{\kappa_*}+\kappa_*) \theta \tond*{\rho(\eta),\rho(\sigma \eta)}\quadr*{\psi(\eta)-\psi(\sigma \eta)}^2.
	\end{align*}
	
	Therefore we get $$D\geq (4M_\theta \overline{\kappa_*}+2{\kappa_*}) \curlyA(\rho,\psi).$$
	\subparagraph*{Term B}
	We have that 
	\[
	\begin{split}
		B = & -\sum_{\substack{\eta\in\Omega,\\(\sigma,\gamma)\in \Upsilon^<(\eta)}} \metacsigma \nabla_\sigma c(\eta,\gamma) I(\eta,\sigma, \gamma,\sigma)\\
		= &\,  -\sum_{\substack{\eta\in \Omega,\\(\gamma,\sigma)\in \Upsilon^<(\eta)}} \metacsigma \nabla_\sigma c(\eta,\gamma) I(\eta,\sigma, \gamma,\sigma)  \\
		= &\,  -\sum_{\substack{\eta\in \Omega,\\(\sigma,\gamma)\in \Upsilon^<(\eta)}} m(\eta) c(\eta, \gamma) \nabla_\gamma c(\eta,\sigma) I(\eta,\gamma, \sigma,\gamma) \\
		= & \,- \sum_{\substack{\eta\in \Omega,\\(\sigma,\gamma)\in \Upsilon^<(\eta)}} \metacsigma \nabla_\sigma c(\eta,\gamma) I(\eta,\gamma, \sigma,\gamma)
		\\
		= &\, \sum_{\substack{\eta\in \Omega,\\(\sigma,\gamma)\in \Upsilon^<(\eta)}} \metacsigma \nabla_\sigma c(\eta,\gamma) I(\eta,\sigma, \gamma,\sigma)\\
		= &\,- B
	\end{split}
	\]
	which implies that $B=0$.
	In the above, the second equality is by $4.$ of Lemma \ref{lem: auxiliary lemma glauber dynamics conforti}, the third by exchanging the role of $\sigma$ and $\gamma$, the fourth by $1.$ of Lemma \ref{lem: auxiliary lemma glauber dynamics conforti} and the fifth by $8.$ of Lemma \ref{lem: auxiliary lemma glauber dynamics conforti}.
	
	\subparagraph*{Term C}
	This is similar to term B using reversibility.
	We have 
	\[
	C = \sum_{\substack{\eta\in \Omega,\\(\sigma,\bar{\gamma})\in \Upsilon^>(\eta)}} \metacsigma \nabla_\sigma c(\eta,\bar{\gamma}) I(\eta,\sigma, \sigma,\bar{\gamma}).
	\] 
	Using the reversibility property \eqref{eqn: reversibility with moves} with $F(\eta,\sigma) = \sum_{\bar{\gamma}: (\sigma, \bar{\gamma})\in \Upsilon^>(\eta)} \nabla_\sigma c(\eta,\bar{\gamma}) I(\eta,\sigma, \sigma,\bar{\gamma})$, the  assumption $\sigma = \sigma^{-1}$ and  properties $3.$ and $6.$ of Lemma \ref{lem: auxiliary lemma glauber dynamics conforti}, we get
	\[
	C = -\sum_{\substack{\eta\in \Omega,\\(\sigma,\bar{\gamma})\in \Upsilon^<(\eta)}} \metacsigma \nabla_\sigma c(\eta,\bar{\gamma}) I(\sigma \eta, \sigma, \sigma, \bar{\gamma}).
	\]
	We want to show that this expression is $0$, analogously to what was done for $B$. Notice that
	\[
	\begin{split}
		C =& -\sum_{\substack{\eta\in \Omega,\\(\sigma,\bar{\gamma})\in \Upsilon^<(\eta)}} \metacsigma \nabla_\sigma c(\eta,\bar{\gamma}) I(\sigma \eta, \sigma, \sigma, \bar{\gamma})\\
		= &\,  -\sum_{\substack{\eta\in \Omega,\\(\bar{\gamma},\sigma)\in \Upsilon^<(\eta)}} \metacsigma \nabla_\sigma c(\eta,\bar{\gamma}) I(\sigma \eta, \sigma, \sigma, \bar{\gamma})  \\
		= &\,  -\sum_{\substack{\eta\in \Omega,\\(\sigma,\bar{\gamma})\in \Upsilon^<(\eta)}} m(\eta) c(\eta, \bar{\gamma}) \nabla_{\bar{\gamma}} c(\eta,\sigma) I(\bar{\gamma}\eta,\bar{\gamma}, \bar{\gamma},\sigma) \\
		= &\, -\sum_{\substack{\eta\in \Omega,\\(\sigma,\bar{\gamma})\in \Upsilon^<(\eta)}} \metacsigma \nabla_\sigma c(\eta,\bar{\gamma}) I(\bar{\gamma}\eta,\bar{\gamma}, \bar{\gamma},\sigma)\\
		= &\, \sum_{\substack{\eta\in \Omega,\\(\sigma,\bar{\gamma})\in \Upsilon^<(\eta)}} \metacsigma \nabla_\sigma c(\eta,\bar{\gamma}) I(\sigma \eta, \sigma, \sigma, \bar{\gamma})\\
		= &\, -C,
	\end{split}
	\]
	which implies that $C=0$. In the above, the second equality is by $4.$ of Lemma \ref{lem: auxiliary lemma glauber dynamics conforti}, the third by exchanging the role of $\sigma$ and $\bar \gamma$, the fourth by $1.$ of Lemma \ref{lem: auxiliary lemma glauber dynamics conforti} and the fifth by $9.$ of Lemma \ref{lem: auxiliary lemma glauber dynamics conforti}.

	\subparagraph*{Term A}
	We split $A$ into three different terms: $A= A_1 + A_2 +A_3$, where
	\begin{align*}
		A_1 &= \sum_{\substack{\eta\in \Omega,\\(\sigma, \gamma) \in \Upsilon^<(\eta)}} \metacsigma c(\sigma \eta,\gamma) I(\eta, \sigma, \gamma,\gamma),\\
		A_2 &= \sum_{\substack{\eta\in \Omega,\\(\sigma, \gamma) \in \Upsilon^>(\eta)}} \metacsigma c(\eta, \gamma) I(\eta, \sigma, \gamma,\gamma), \\
		A_3 & = \frac{1}{2}\sum_{\substack{\eta\in \Omega,\\(\sigma, \gamma) \in \Upsilon^=(\eta)}} \metacsigma \tond*{c(\sigma\eta, \gamma)+c(\eta,\gamma)} I(\eta, \sigma, \gamma,\gamma). 
	\end{align*}
	We want to show that $A_1+A_2=0$ and $A_3=0$.
	We have that
	\[
	\begin{split}
		A_2 = &  \sum_{\substack{\eta\in \Omega,\\(\sigma, \gamma) \in \Upsilon^>(\eta)}} \metacsigma c(\eta, \gamma) I(\eta, \sigma, \gamma,\gamma) \\
		= & \sum_{\substack{\eta\in \Omega,\\(\sigma, \gamma) \in \Upsilon^>(\gamma\eta)}} m(\eta) c(\eta, \gamma) c(\gamma \eta,\sigma) I(\gamma\eta, \sigma, \gamma,\gamma) \\
		= & -\sum_{\substack{\eta\in \Omega,\\(\sigma, \gamma) \in \Upsilon^<(\eta)}} m(\eta) c(\eta, \sigma) c(\sigma \eta,\gamma) I(\eta, \sigma, \gamma,\gamma)\\
		= & -A_1.
	\end{split}
	\]
	In the above, the second equality is by the reversibility property \eqref{eqn: reversibility with moves} with $F(\eta,\gamma) = \sum_{\sigma: (\sigma, \gamma)\in\Upsilon^>(\eta)} c(\eta,\sigma) I(\eta,\sigma,\gamma,\gamma)$ and the assumption $\gamma = \gamma^{-1}$, while the second equality is by the properties $2.$, $5.$ and $10.$ of Lemma \ref{lem: auxiliary lemma glauber dynamics conforti}.
	It follows that the contribution of the first two terms is $0$.
	\\
	It remains to show that $A_3 =0$, which is done in a similar way: notice that 
	\[
	\begin{split}
		&\sum_{\substack{\eta\in \Omega,\\(\sigma, \gamma) \in \Upsilon^=(\eta)}} \metacsigma c(\eta,\gamma) I(\eta, \sigma, \gamma,\gamma) \\
		= & \sum_{\substack{\eta\in \Omega,\\(\sigma, \gamma) \in \Upsilon^=(\gamma \eta)}}  m(\eta) c(\eta, \gamma)c(\gamma \eta,\sigma) I(\gamma\eta, \sigma, \gamma,\gamma) \\
		= & -\sum_{\substack{\eta\in \Omega,\\(\sigma, \gamma) \in \Upsilon^=(\eta)}}  m(\eta) c(\eta, \sigma)c(\sigma \eta,\gamma) I(\eta, \sigma, \gamma,\gamma).
	\end{split}
	\]
	In the above, the first equality is 
	by the reversibility property \eqref{eqn: reversibility with moves}  with  the function
	$F(\eta,\gamma) = \sum_{\sigma: (\sigma, \gamma)\in\Upsilon^=(\eta)} c(\eta,\sigma) I(\eta,\sigma,\gamma,\gamma)$ and by the assumption $\gamma = \gamma^{-1}$, while the second equality holds by properties $2.$, $7.$ and $10.$ of Lemma \ref{lem: auxiliary lemma glauber dynamics conforti}.
	It follows that $A_3 = 0$, thus concluding the proof of the theorem.

	\subsubsection{Examples of Glauber dynamics} \label{sec:appl-glauber}
	Below, following \cite{con-2022}, we present two examples of Glauber dynamics models satisfying the assumptions of Theorem \ref{thm: entropic curvature Glauber dynamics}.
	
	\paragraph*{Curie Weiss model}
	For the Curie Weiss model, the state space is the discrete hypercube $\Omega = \graf*{-1,1}^N$ for some integer $N>0$.
	The set of moves $G$ is given by $G = \graf*{\sigma_1,\ldots, \sigma_N}$ where $\sigma_i\colon \Omega \to \Omega$ corresponds to   flipping the $i$-th bit. Finally, the Hamiltonian function is 
	\[
	H(\eta) \coloneqq -\frac{1}{2N} \sum_{i,j=1}^N \eta_i \eta_j.
	\] 
	In this setting, the assumptions of Theorem \ref{thm: entropic curvature Glauber dynamics}  and the explicit values of $\kappa_*, \overline{\kappa_*}$ were checked by Conforti, who proves the following
	\begin{theorem}[Thm. 4.2 of \cite{con-2022}]\label{thm:curie-weiss}
		Assume that
		\[
		(N-1) \tond*{e^{\frac{2\beta}{N}} - 1 }\leq 1.
		\]
		Then the assumptions of Theorem \ref{thm: entropic curvature Glauber dynamics} are satisfied with
		\begin{align*}
			\kappa_* &= f_{\text{CW},\beta,N}\tond*{\floor*{\frac{N-1}{2}}},
			\\
			\overline{\kappa_*} & = e^{-\frac{\beta}{N}(N-1)} \quadr*{1-(N-1)\tond*{1-e^{\frac{2\beta}{N}}}},
		\end{align*}
		where $f_{\text{CW},\beta,N}\colon \N\to \R$ is defined by
		\begin{align*}
			f_{\text{CW},\beta,N}(m) &\coloneqq e^{-\frac{\beta}{N}(N-1-2m)} \quadr*{1-(N-1-m)\tond*{e^{\frac{2\beta}{N}}-1}}
			\\
			& + e^{\frac{\beta}{N}(N-1-2m)} \quadr*{1-m\tond*{e^{\frac{2\beta}{N}}-1}}.
		\end{align*}
	\end{theorem}
	\begin{remark}[Comparison with Cor. 4.5 of \cite{erbar-henderson-menz-tetali}]
		In particular, in the limit $N\to \infty$, the condition above reads $\beta\leq \frac{1}{2}$. Thus by choosing $\theta = \theta_1$ and combining Theorems \ref{thm: entropic curvature Glauber dynamics} and \ref{thm:curie-weiss} we have that as $N\to \infty$  the entropic curvature of the Curie--Weiss model satisfies
		\[
		\rice \geq (1-\beta) + (1-2\beta) e^{-\beta}
		\] 
		for $\beta \leq \frac{1}{2}$. This improves both in the estimate and in the range of admissible $\beta$ over \cite[Cor. 4.5]{erbar-henderson-menz-tetali}, where it was proved that, as $N\to \infty$,
		$\rice \geq 2(1-2\beta e^{2\beta}) e^{-\beta}$ for $\beta \lessapprox 0.284$.
	\end{remark}

	\paragraph*{Ising model}
	The second example of Glauber dynamics that we consider is the Ising model.
	For the Ising model, we let $\Lambda\subset \Z^d$ be a connected subset of $\Z^d$, endowed with the inherited graph structure $\sim$ of the discrete grid, and consider the state space $\Omega = \graf*{-1,1}^\Lambda$. The set of moves is $G = \graf*{\sigma_x}_{x\in \Lambda}$ where $\sigma_x\colon \Omega \to \Omega$ acts on a state $\eta$ by flipping the spin $\eta_x$ at site $x$. 
	Finally, the Hamiltonian is defined by
	\[
	H(\eta) \coloneqq -\frac{1}{2}\sum_{x\sim y} \eta_x \eta_y.
	\]
	Again, the assumptions and values of $\kappa_*, \overline{\kappa_*}$  in Theorem \ref{thm: entropic curvature Glauber dynamics} were checked in \cite{con-2022}, where the following result is proved.
	\begin{theorem}[Thm. 4.3 of \cite{con-2022}]\label{thm:ising}
		Assume that
		\begin{equation}
			\label{eqn: Conforti condition for Ising model}
			2d\tond*{1-e^{-2\beta}} e^{4d\beta} \leq 1.
		\end{equation}
		Then the assumptions of Theorem \ref{thm: entropic curvature Glauber dynamics} are satisfied and we have
		\begin{align*}
			\kappa_* &= 2  - 2d(1-e^{-2\beta}) e^{2\beta d}
			\\
			\overline{\kappa_*} & =  e^{-2\beta d} - 2d(1-e^{-2\beta}) e^{2\beta d}.
		\end{align*}
	\end{theorem}
	
	\begin{remark}[Comparison with Cor. 4.4 of \cite{erbar-henderson-menz-tetali}]
		By combining Theorems \ref{thm: entropic curvature Glauber dynamics} and \ref{thm:curie-weiss} and by choosing the logarithmic mean $\theta = \theta_1$, it follows that
		\begin{equation}\label{eq:entr-ising-new}
			\rice \geq 1 + e^{-2\beta d} - 3d(1-e^{-2\beta}) e^{2\beta d} \qquad \text{ if } \qquad  2d\tond*{1-e^{-2\beta}} e^{4d\beta} \leq 1.
		\end{equation}
		On the other hand, in \cite[Cor. 4.4]{erbar-henderson-menz-tetali}, it was proved for the Ising model that
		\begin{equation}\label{eq:entr-ising-old}
			\rice \geq  2 \quadr* {1-\tond*{2d-1} \tond*{1-e^{-2\beta}} e^{4\beta d}} e^{-2\beta d} \qquad \text{ if } \qquad \tond*{2d-1} \tond*{1-e^{-2\beta}} e^{4\beta d} \leq 1. 
		\end{equation}
		As observed by Conforti, the condition in \eqref{eq:entr-ising-new} is a bit more demanding then the one in \eqref{eq:entr-ising-old}, but when it is satisfied then the corresponding lower bound for the entropic Ricci curvature is better (for $d\geq 2$).
	\end{remark}

	\subsection{Bernoulli--Laplace model}
	\label{sec: Bernoulli Laplace model}
	In this subsection we analyze a simplified version of the Bernoulli--Laplace model, following Section 5.1 of \cite{con-2022}.
	Given integers $L>N\in \N$, where $L$ represents the number of sites and $N$ the number of particles, the state space is
	\[
	\Omega = \graf*{\eta \in \graf{0,1}^{[L]} \mid \sum_{i=1}^L \eta_i =  N },
	\]
	where $\quadr*{L} = \graf*{1,\ldots,L}$.
	Let $\delta_i \in \graf{0,1}^{[L]}$ be defined by
	$\delta_i(k) = \left\{\begin{array}{ll}
		1 &\mbox{ if } i=k,\\
		0 &\mbox{ otherwise }.
	\end{array}\right.$
	Then the set of moves is $G = \graf*{\sigma_{ik} \mid i, k \in [L]}$ where $\sigma_{ik}$ moves a particle from site $i$ to site $k$ if possible, i.e.
	\[
	\sigma_{ij}(\eta) = \left\{\begin{array}{ll}
		\eta-\delta_i+\delta_j &\mbox{ if } \eta_i(1-\eta_j) > 0,\\
		\eta & \mbox{ otherwise}.
	\end{array}\right.
	\]
	The transition rates are given by
	\[
	c(\eta, \sigma_{ij}) = \eta_i(1-\eta_j)
	\]
	and the reversible measure $m$ is the uniform one on $\Omega$.
	Finally, notice that we have $\sigma_{ij}^{-1} = \sigma_{ji}$.

	\begin{theorem}\label{thm:bern-lapl}
		For the Bernoulli--Laplace model and for all weight functions $\theta$ satisfying Assumption \ref{ass: weak assumption theta}, the inequality \eqref{eq:main-ineq} holds with constant
		\[
		K =  {M_\theta+ \frac{L}{2}}.
		\]
	\end{theorem}
	
	\begin{remark}[Comparison with \cite{con-2022}]
		In \cite[Thm. 5.1]{con-2022}, under the same assumptions on the model, Conforti establishes inequality \eqref{eq:bak-eme-method} and thus $\csi_\phi(2K)$ with constant $K$ equal to
		\begin{itemize}
			\item $\frac{L}{2}$ for general convex $\phi$ satisfying convexity of \eqref{eq:conf-conv};
			\item $\frac{L}{2} + 1$ for $\phi = \phi_1$, corresponding to $\mlsi(L+2)$;
			\item $\frac{\alpha L}{2}$ for $\phi = \phi_\alpha$ with $\alpha \in (1,2]$.		
		\end{itemize}
		Thus, by the discussion in Section \ref{sec:gen-ineq}, we obtain a stronger result for the case $\theta = \theta_1$ and complementary results for other choices of $\theta$.
	\end{remark}

	\begin{remark}
		The entropic curvature of the Bernoulli--Laplace model has been studied before \cite{erbar-maas-tetali}, also in the more general case of non-homogeneous rates \cite{fat-maa-2016}. In the homogeneous setting, our result for $\theta = \theta_1$ recovers the same (best known) lower bound of \cite{erbar-maas-tetali}. 
	\end{remark}
	
	\subsubsection{Proof of Theorem \ref{thm:bern-lapl}}
	Again, the proof is based on Lemma \ref{lem: lower bound for B with coupling rates} and adapts the arguments of \cite{con-2022}, from which we use the same coupling rates. Below, we reserve the letters $i,j,k,l$ for elements of $[L]$ and $\eta$ for a state in $\Omega$. For $(\eta,\sigma_{ij}) \in S$ (i.e. $\eta_i = 1,\eta_j=0$) set
	\[
	\ccpl(\eta,\sigma_{ij} \eta, \gamma, \bar{\gamma}) = \left\{
	\begin{array}{ll}
		\min \graf{c( \eta,\gamma), c(\sigma_{ij}\eta, \gamma)} & \mbox{ if } \gamma = \bar{\gamma}\in G, \\
		1 & \mbox{ if } {\gamma} = \sigma_{ij}, \bar{\gamma} = e \mbox{ or if } \gamma=e, \bar{\gamma} = \sigma_{ji},\\ 
		1-\eta_l & \mbox{ if } {\gamma} = \sigma_{il}, \bar{\gamma} = \sigma_{jl} \mbox{ and } l\notin \graf{i,j},\\
		\eta_k & \mbox{ if } {\gamma} = \sigma_{kj}, \bar{\gamma} = \sigma_{ki} \mbox{ and } k\notin \graf{i,j},\\
		0 & \mbox{ otherwise}.		 
	\end{array}
	\right.
	\]
	With these coupling rates and by \eqref{eqn: J vs I}, the right hand side of equation \eqref{eqn: lower bound on B with couplings} from Lemma \ref{lem: lower bound for B with coupling rates} is bounded from below by $ \frac{1}{4}(A+B+C+D)$, where we define
	\begin{align*}
		A &= \sum_{\eta,i,j,k,l} m(\eta) c(\eta, \sigma_{ij}) \min \graf*{c(\eta,\sigma_{kl}),c(\sigma_{ij} \eta, \sigma_{kl})} I(\eta,\sigma_{ij},\sigma_{kl},\sigma_{kl}),\\
		B &= \sum_{\eta,i,j} m(\eta) c(\eta, \sigma_{ij}) \quadr*{ J(\eta,\sigma_{ij},\sigma_{ij},e) + J(\eta,\sigma_{ij}, e, \sigma_{ji})} ,\\
		C &= \sum_{\eta,i,j} m(\eta) c(\eta, \sigma_{ij}) \quadr*{\sum_{l\neq i,j} (1-\eta_l) J(\eta,\sigma_{ij},\sigma_{il}, \sigma_{jl})},\\
		D &= \sum_{\eta,i,j} m(\eta) c(\eta,\sigma_{ij}) \quadr*{\sum_{k \neq i,j} \eta_k J(\eta,\sigma_{ij},\sigma_{kj}, \sigma_{ki}) }.
	\end{align*}
	
	We show that \begin{itemize}
		\item $A=0$,
		\item $B\geq (4+4M_\theta) \curlyA(\rho,\psi)$,
		\item $C \geq 2(L-N-1) \curlyA(\rho,\psi)$,
		\item $D \geq 2(N-1) \curlyA(\rho,\psi) $,
	\end{itemize} 
	from which the theorem follows by Lemma \ref{lem: lower bound for B with coupling rates}.
	It is convenient to prove first the following
	\begin{lemma}
		\label{lem: auxiliary lemma bernoulli laplace conforti}
		For all $\eta \in \Omega$ and $i,j,k,l\in [L]$ the following hold:
		\begin{enumerate}
			\item $ c(\eta, \sigma_{ij}) \min \graf*{c(\eta,\sigma_{kl}),c(\sigma_{ij} \eta, \sigma_{kl})} = \left\{
			\begin{array}{ll}
				1 & \mbox{ if } i\neq k, j\neq l, \eta_i = \eta_k =1, \eta_j = \eta_l =0, \\
				0 & \mbox{ otherwise }.
			\end{array}
			\right.$
			\item $c(\eta, \sigma_{ij}) \min \graf*{c(\eta, \sigma_{kl}), c(\sigma_{ij}\eta, \sigma_{kl})} = c(\eta, \sigma_{kl}) \min \graf*{c(\eta, \sigma_{ij}), c(\sigma_{kl}\eta, \sigma_{ij})}$. 
			\item $ I(\eta,\sigma_{ij},\sigma_{kl}, \sigma_{kl}) = - I(\sigma_{kl} \eta, \sigma_{ij}, \sigma_{lk}, \sigma_{lk})$ if $c(\eta, \sigma_{ij}) \min \graf*{c(\eta,\sigma_{kl}),c(\sigma_{ij} \eta, \sigma_{kl})} >0$.
		\end{enumerate}
	\end{lemma}
	\begin{proof}[Proof of Lemma]
		Statements 1--2 were already observed by Conforti in the proof of \cite[Thm 5.1]{con-2022} and are easy to check, while statement 3 is immediate from the definitions.
	\end{proof}
	
	\subparagraph*{Term A}
	We have, using $2.$ of Lemma \ref{lem: auxiliary lemma bernoulli laplace conforti},
	\[
	\begin{split}
		A = & \sum_{\eta,i,j,k,l} m(\eta) c(\eta, \sigma_{ij}) \min \graf*{c(\eta,\sigma_{kl}),c(\sigma_{ij} \eta, \sigma_{kl})} I(\eta,\sigma_{ij},\sigma_{kl},\sigma_{kl}) \\
		= & \sum_{\eta,i,j,k,l} m(\eta) c(\eta, \sigma_{kl}) \min \graf*{c(\eta,\sigma_{ij}),c(\sigma_{kl} \eta, \sigma_{ij})} I(\eta,\sigma_{ij},\sigma_{kl},\sigma_{kl}) \\
		= & \sum_{\eta,k,l} m(\eta) c(\eta,\sigma_{kl}) F(\eta,\sigma_{kl})
	\end{split}
	\]
	with 
	\[
	F(\eta,\sigma_{kl}) = \sum_{i,j} \min \graf*{c(\eta,\sigma_{ij}),c(\sigma_{kl} \eta, \sigma_{ij})} I(\eta,\sigma_{ij},\sigma_{kl},\sigma_{kl}).
	\]
	Next, using the reversibility property \eqref{eqn: reversibility with moves}, the fact that $\sigma_{lk} = \sigma_{kl}^{-1}$ for the first and second equality, and properties $2.$ and $3.$ of Lemma \ref{lem: auxiliary lemma bernoulli laplace conforti} respectively for the third and fourth equality, we deduce that
	\[
	\begin{split}
		A = & \sum_{\eta,k,l} m(\eta) c(\eta,\sigma_{kl}) F(\sigma_{kl} \eta,\sigma_{lk}) \\
		= & \sum_{\eta,i,j,k,l} m(\eta) c(\eta, \sigma_{kl}) \min \graf*{c(\eta,\sigma_{ij}),c(\sigma_{kl} \eta, \sigma_{ij})} I(\sigma_{kl}\eta,\sigma_{ij},\sigma_{lk},\sigma_{lk}) \\
		= & \sum_{\eta,i,j,k,l} m(\eta) c(\eta, \sigma_{ij}) \min \graf*{c(\eta,\sigma_{kl}),c(\sigma_{ij} \eta, \sigma_{kl})} I(\sigma_{kl}\eta,\sigma_{ij},\sigma_{lk},\sigma_{lk}) \\
		= & - \sum_{\eta,i,j,k,l} m(\eta) c(\eta, \sigma_{ij}) \min \graf*{c(\eta,\sigma_{kl}),c(\sigma_{ij} \eta, \sigma_{kl})} I(\eta,\sigma_{ij},\sigma_{kl},\sigma_{kl}) \\
		= & -A.
	\end{split}
	\]
	This implies that $A=0$, as desired.
	\subparagraph*{Term B}
	Notice that for $(\eta,\sigma_{ij})\in S$
	\begin{align*}
		J(\eta,\sigma_{ij},\sigma_{ij},e) &= \graf*{\theta\tond*{\rho(\sigma_{ij} \eta), \rho(\sigma_{ij} \eta)}+ \theta(\rho(\eta),\rho(\sigma_{ij} \eta)) }\quadr*{\psi(\eta)-\psi(\sigma_{ij} \eta)}^2 \\
		J(\eta,\sigma_{ij},e,\sigma_{ji}) &= \graf*{\theta\tond*{\rho( \eta), \rho( \eta)}+ \theta(\rho(\eta),\rho(\sigma_{ij} \eta))} \quadr*{\psi(\eta)-\psi(\sigma_{ij} \eta)}^2
	\end{align*}
	and so
	\begin{equation*}
		J(\eta,\sigma_{ij},\sigma_{ij},e) +  J(\eta,\sigma_{ij},e,\sigma_{ji}) \geq \graf*{2M_\theta +2}\theta \tond*{\rho(\eta),\rho(\sigma_{ij} \eta)} \quadr*{\psi(\eta)-\psi(\sigma_{ij} \eta)}^2.
	\end{equation*}
	
	Therefore we get $$B\geq 4(M_\theta+1) \curlyA(\rho,\psi).$$
	
	\subparagraph*{Term C}
	Notice that for $(\eta,\sigma_{ij})\in S$ we have $\eta_i=1, \eta_j =0$ and there are $L-N-1$ empty sites left. Moreover 
	when $l\neq i,j$ and $\eta_l=0$ we have that $\sigma_{jl}\sigma{ij}\eta = \sigma_{il}\eta$ and so
	\begin{align*}
		J(\eta,\sigma_{ij},\sigma_{il},\sigma_{jl}) \geq 
		\theta(\rho(\eta),\rho(\sigma_{ij} \eta)) \quadr*{\psi(\eta)-\psi(\sigma_{ij} \eta)}^2.
	\end{align*}

	Therefore we get $$C \geq 2(L-N-1) \curlyA(\rho,\psi).$$

	\subparagraph*{Term D}
	Notice that for $(\eta,\sigma_{ij})\in S$ we have $\eta_i=1, \eta_j =0$ and there are $N-1$  other occupied sites. Moreover 
	when $k\neq i,j$ and $\eta_k=1$ we have $\sigma_{ki}\sigma{ij}\eta = \sigma_{kj}\eta$ and so
	\begin{align*}
		J(\eta,\sigma_{ij},\sigma_{kj},\sigma_{ki}) \geq 
		\theta(\rho(\eta),\rho(\sigma_{ij} \eta)) \quadr*{\psi(\eta)-\psi(\sigma_{ij} \eta)}^2.
	\end{align*}

	Therefore we get $$D \geq 2(N-1) \curlyA(\rho,\psi).$$

	Combining these estimates for $A,B,C,D$, an application of Lemma \ref{lem: lower bound for B with coupling rates} gives 
	\[
	\curlyB(\rho,\psi) \geq \tond*{1 + M_\theta+
		\frac{L}{2}-1}\curlyA(\rho,\psi)  = \tond*{M_\theta+\frac{L}{2}} \curlyA(\rho,\psi).
	\]
	This concludes the proof of the theorem.

	\subsection{Hardcore model}

	\label{sec: Hardcore model}
	Following Section 5.2 of \cite{con-2022}, we consider the classical  hardcore model. Let $(V,E)$ be a simple, finite, connected graph and write $x\sim y$ if the vertices $x,y$ are connected (in the rest of this subsection,  $x,y$ always denote general elements of $V$). The state space is 
	\[
	\Omega = \graf*{\eta \in \graf*{0,1}^ V \mid \eta_x \eta_y =0 \mbox{ if } x \sim y }.
	\]
	In other words, each vertex can either be empty or be occupied by a particle, with the rule that if a site is occupied then its neighbors are all free.
	Let $N_x = \graf*{y\in V \mid x\sim y}$ be the set of neighbors of vertex $x$, $\bar{N}_x = N_x \cup \graf{x}$ and as before let $\delta_x \in \graf{0,1}^{V}$ be defined by
	$\delta_x(y) = \left\{\begin{array}{ll}
		1 &\mbox{ if } x=y,\\
		0 &\mbox{ otherwise }.
	\end{array}\right.$
	Then, the set of moves is given by $G = \graf*{\gxp, \gxm \mid x \in V}$ where $\gxp$ adds a particle to site $x$ if possible and $\gxm$ removes it if possible, i.e.
	\[
	\gxp(\eta) = \left\{\begin{array}{ll}
		\eta + \delta_x &\mbox{ if } \eta+\delta_x \in \Omega,\\
		\eta & \mbox{ otherwise },
	\end{array}\right.
	\]
	\[
	\gxm(\eta) = \left\{\begin{array}{ll}
		\eta - \delta_x &\mbox{ if } \eta - \delta_x \in \Omega,\\
		\eta & \mbox{ otherwise. }
	\end{array}\right.
	\]
	We also denote  $G^+ = \graf*{\gxp \mid x\in V}$ and $G^- = \graf*{\gxm \mid x\in V}$.
	For a given parameter $\beta\in (0,1)$, the transition rates are defined by
	\begin{align*}
		& c(\eta, \gxp) = \beta \prod_{y \in \bar{N}_x} (1-\eta_y),\\
		& c(\eta, \gxm) = \eta_x
	\end{align*}
	for all $\eta \in \Omega, x\in V$.
	With these choices, we have $\tond*{\gxp}^{-1} = \gxm$ and the reversible measure is given by
	\[
	m(\eta) = \frac{1}{Z} \mathds{1}_{\eta \in \Omega} \prod_{x\in V} \beta^{\eta_x},
	\]
	where $Z>0$ is the normalization constant.

	\begin{theorem}\label{thm:harcore-model}
		Let $\Delta$ be the maximum degree of $(V,E)$ and assume that
		\begin{equation}\label{eq:hardcore-condition}
			\beta \Delta \leq 1.
		\end{equation}
		Set
		\[
		\kappa_* = 1-\beta(\Delta-1), \qquad \overline{\kappa_*} = \min \graf*{\beta, 1-\beta \Delta}.
		\]
		Then, for all weight functions $\theta$ satisfying Assumption \ref{ass: weak assumption theta}, inequality \eqref{eq:main-ineq} holds with constant  
		\[
		K = {\frac{\kappa_*}{2}+ M_\theta\overline{\kappa_*}}.
		\]
	\end{theorem}
	\begin{remark}[Comparison with \cite{con-2022}]
		In \cite[Thm. 5.2]{con-2022}, under the same assumptions on the model, Conforti establishes inequality \eqref{eq:bak-eme-method} and thus $\csi_\phi(2K)$ with constant $K$ equal to
		\begin{itemize}
			\item $\frac{\kappa_*}{2}$ for general convex $\phi$ satisfying convexity of \eqref{eq:conf-conv};
			\item $\frac{\kappa_*}{2} + \overline{\kappa_*}$ for $\phi = \phi_1$, corresponding to $\mlsi(\kappa_* + 2\overline{\kappa_*})$;
			\item $\frac{\alpha \kappa_*}{2}$ for $\phi = \phi_\alpha$ with $\alpha \in (1,2]$.		
		\end{itemize}
		Thus, by the discussion in Section \ref{sec:gen-ineq}, we obtain a stronger result for the case $\theta = \theta_1$ and complementary results for other choices of $\theta$.
	\end{remark}

	\begin{remark}
		The entropic curvature of the hardcore model has been studied before in \cite{erbar-henderson-menz-tetali}, where a more general version of the model is considered. 
		When restricting to the classical version discussed in this section, it was proved in \cite[Cor. 4.8]{erbar-henderson-menz-tetali} that
		\[
		\rice \geq \frac{\kappa_*}{2}
		\]
		under condition \eqref{eq:hardcore-condition}. Therefore, in this setting, by choosing $\theta = \theta_1$ in Theorem  \ref{thm:harcore-model} we find a better lower bound for the entropic curvature.
	\end{remark}

	\subsubsection{Proof of Theorem \ref{thm:harcore-model}}
	The proof is based on Lemma \ref{lem: lower bound for B with coupling rates} and on the arguments in \cite{con-2022}, from which we use the same coupling rates.  For $(\eta,\gxp) \in S$ (i.e. $\eta\mid _{\bar{N}_x}= 0$) we set
	\[
	\ccpl(\eta,\gxp \eta, \gamma, \bar{\gamma}) = \left\{
	\begin{array}{ll}
		\min \graf{c( \eta,\gamma), c(\gxp\eta, \gamma)} & \mbox{ if } \gamma = \bar{\gamma} \in G, \\
		\beta & \mbox{ if } {\gamma} = \gyp, \bar{\gamma} = \gxm \mbox{ with } y \sim x, \eta|_{\bar{N}_y} = 0,\\ 
		\beta & \mbox{ if } {\gamma} = \gxp, \bar{\gamma} = e,\\
		1-\beta \abs*{\graf*{y : y\sim x, \eta |_{\bar{N}_y}=0}} & \mbox{ if } {\gamma} = e, \bar{\gamma} = \gxm,\\
		0 & \mbox{ otherwise}.		 
	\end{array}
	\right.
	\]
	If $(\eta, \gxm) \in S$ then also $(\gxm\eta,\gxp) \in S$, and so we can set
	\[
	\ccpl(\eta,\gxm \eta, \gamma, \bar{\gamma}) = \ccpl(\gxm \eta,\gxp \gxm \eta, \bar{\gamma}, \gamma) = \ccpl(\gxm \eta, \eta, \bar{\gamma}, \gamma).
	\]
	With these coupling rates the right hand side of equation  \eqref{eqn: lower bound on B with couplings}
	from Lemma \ref{lem: lower bound for B with coupling rates} reads
	\begin{align*}
		&\frac{1}{4} \sum_{(\eta, \gxp)\in S} \sum_{\gamma,\bar{\gamma}\in G^*} m(\eta)c(\eta,\gxp) \ccpl(\eta,\gxp \eta,\gamma,\bar{\gamma}) J(\eta, \gxp, \gamma,\bar{\gamma}) \\
		+ & \frac{1}{4} \sum_{(\eta, \gxm) \in S} \sum_{\gamma,\bar{\gamma}\in G^*} m(\eta)c(\eta,\gxm) \ccpl(\gxm\eta, \eta,\bar{\gamma},{\gamma}) J(\eta, \gxm, \gamma,\bar{\gamma}).
	\end{align*}
	
	Using the reversibility property \eqref{eqn: reversibility with moves} with 
	\[
	F(\eta, \sigma) = \mathds{1}_{G^-}(\sigma) \sum_{\gamma,\bar{\gamma}\in G^*} \ccpl(\sigma\eta, \eta,\bar{\gamma},{\gamma}) J(\eta, \sigma, \gamma,\bar{\gamma})
	\] 
	and that  $J(\eta,\sigma \eta, \gamma, \bar{\gamma}) = J(\sigma \eta, \sigma^{-1}, \bar{\gamma}, \gamma)$ (when $\sigma^{-1}\sigma \eta = \eta$) the second term is equal to the first, so we can rewrite the previous quantity as
	\begin{equation}
		\label{eqn: lower bound curly B for hardcore model}
		\frac{1}{2} \sum_{(\eta, \gxp)\in S} \sum_{\gamma,\bar{\gamma}\in G^*} m(\eta)c(\eta,\gxp) \ccpl(\eta,\gxp \eta,\gamma,\bar{\gamma}) J(\eta, \gxp, \gamma,\bar{\gamma}).
	\end{equation}

	Similarly we have
	
	\begin{align*}
		\curlyA(\rho,\psi) = &\frac{1}{2} \sum_{(\eta, \gxp)\in S}  m(\eta)c(\eta,\gxp)  \theta\tond*{\rho(\eta), \rho(\gxp \eta)}\quadr*{\psi(\eta) - \psi(\gxp\eta)}^2 \\
		+ & \frac{1}{2} \sum_{(\eta, \gxm) \in S}  m(\eta)c(\eta,\gxm) \theta\tond*{\rho(\eta), \rho(\gxm \eta)}\quadr*{\psi(\eta) - \psi(\gxm\eta)}^2
	\end{align*}
	and using again reversibility \eqref{eqn: reversibility with moves} the second term is equal to the first, so that we can write
	\[
	\curlyA(\rho,\psi) = \sum_{(\eta, \gxp)\in S}  m(\eta)c(\eta,\gxp)\theta\tond*{\rho(\eta), \rho(\gxp \eta)}  \quadr*{\psi(\eta) - \psi(\gxp\eta)}^2.
	\]
	
	We then have that (using that $\forall \eta,x,y$ $c(\eta, \gym) \leq c(\gxp\eta, \gym)$ and $c(\eta, \gyp) \geq c(\gxp\eta, \gyp)$) the quantity \eqref{eqn: lower bound curly B for hardcore model} (and in particular $\curlyB(\rho,\psi)$ too) is lower bounded by $ \frac{1}{2}(A+B+C)$ with
	\begin{align*}
		\begin{split}
			A & = \sum_{\substack{\eta\in \Omega \\ x, y\in V}} m(\eta) c(\eta,\gxp) c(\eta, \gym)  I(\eta,\gxp,\gym,\gym)\\
			& + \sum_{\substack{\eta\in \Omega \\ x, y\in V}} m(\eta) c(\eta, \gxp) c(\gxp \eta, \gyp) I(\eta,\gxp,\gyp,\gyp),
		\end{split}\\
		B & = \beta\sum_{\substack{\eta,x,y : \\ x\sim y, \eta|_{\bar{N}_y} = 0} } m(\eta) c(\eta, \gxp) I(\eta,\gxp,\gyp,\gxm) ,\\
		C &= \sum_{\substack{\eta\in \Omega \\ x, y\in V}} m(\eta) c(\eta, \gxp) \quadr*{\tond*{1-\beta\abs*{\graf*{y: y\sim x, \eta|_{\bar{N}_y} = 0}}} J(\eta, \gxp,e,\gxm) + \beta J(\eta,\gxp,\gxp,e)}.
	\end{align*}

	We will show that \begin{itemize}
		\item $A=0$,
		\item $B=0$,
		\item $C \geq \tond*{\kappa_* + 2M_\theta \overline{\kappa_*}} \curlyA(\rho,\psi)$.
	\end{itemize}
	An application of  Lemma \ref{lem: lower bound for B with coupling rates} then concludes the proof of the theorem.
	To do so, we 
	will use the following:
	\begin{lemma}
		\label{lem: auxiliary lemma hardcore model conforti}
		For all  $\eta\in \Omega$ and $x,y \in V$ the following hold:
		\begin{enumerate}
			\item $c(\eta, \gxp) c(\gxp\eta, \gyp) = c(\eta, \gyp) c(\gyp\eta, \gxp) = \left\{ \begin{array}{ll}
				\beta^2 & \mbox{ if } x\nsim y, \eta|_{\bar{N}_x \cup \bar{N}_y} = 0,\\
				0 & \mbox{ otherwise.}
			\end{array}\right. $
			\item $x\nsim y, \eta|_{\bar{N}_x \cup \bar{N}_y} = 0 \implies \gxp \gyp \eta = \gyp \gxp \eta$.
			\item If $\eta|_{\bar{N}_x\cup \bar{N}_y} = 0$ then $c(\eta,\gxp) = c(\eta, \gyp) = \beta$.
			\item $I(\gyp \eta, \gxp,\gym,\gym) = - I(\eta, \gxp,\gyp,\gyp)$ if $x\nsim y, \eta|_{\bar{N}_x \cup \bar{N}_y} = 0$ .
			\item $ I(\eta,\gxp,\gyp,\gxm) = - I(\eta,\gyp,\gxp,\gym)$ if $\eta|_{\bar{N}_x\cup \bar{N}_y} = 0$.
		\end{enumerate}
	\end{lemma}
	\begin{proof}[Proof of Lemma]
		Statements 1--3 were already in the proof of \cite[Thm. 5.2]{con-2022} and are easy to check, while statement 4--5 are immediate from the definitions.
	\end{proof}
	
	\subparagraph*{Term A}
	We look at the first term in the sum defining $A$: we can write it as
	\[
	\sum_{\substack{\eta\in \Omega \\ \gamma\in G}} m(\eta) c(\eta, \gamma) F(\eta, \gamma)
	\]
	with 
	\[
	F(\eta, \gamma) = \mathds{1}_{G^-}(\gamma) \sum_{x\in V}c(\eta,\gxp) I(\eta,\gxp,\gamma,\gamma).
	\]
	Using reversibility \eqref{eqn: reversibility with moves} we can rewrite it as 
	\[
	\sum_{\substack{\eta\in \Omega\\x,y\in V}} m(\eta) c(\eta,\gyp) c(\gyp\eta,\gxp) I(\gyp \eta, \gxp,\gym,\gym).
	\]
	Then we have
	\begin{align*}
		& \sum_{\substack{\eta\in \Omega \\ x,y\in V}} m(\eta) c(\eta,\gyp) c(\gyp\eta,\gxp) I(\gyp \eta, \gxp,\gym,\gym) \\
		= & \sum_{\substack{\eta\in \Omega \\ x,y\in V}} m(\eta) c(\eta,\gxp) c(\gxp\eta,\gyp) I(\gyp \eta, \gxp,\gym,\gym) \\
		= & - \sum_{\substack{\eta\in \Omega \\ x,y\in V}} m(\eta) c(\eta,\gxp) c(\gxp\eta,\gyp) I(\eta, \gxp,\gyp,\gyp)
	\end{align*}
	using properties $1.$ and $4.$ of Lemma \ref{lem: auxiliary lemma hardcore model conforti}.
	Therefore, the first term in the sum of $A$ is the opposite of the second, which implies $A=0$.

	\subparagraph*{Term B}
	We have 
	\begin{equation*}
		B  = \beta\sum_{\substack{\eta,x,y : \\ x\sim y, \eta|_{\bar{N}_y} = 0} } m(\eta) c(\eta, \gxp) I(\eta,\gxp,\gyp,\gxm).
	\end{equation*}
	Noticing as in \cite{con-2022} that $(\eta,\gxp) \in S$ if and only if $\eta|_{\bar{N}_x=0}$ we can write
	\begin{equation}
		\label{eqn: first expression B hardcore model}
		B  = \beta\sum_{\substack{\eta,x,y : \\ x\sim y, \eta|_{\bar{N}_y} = 0}, \\ \eta|_{\bar{N}_x} = 0} m(\eta) c(\eta, \gxp) I(\eta,\gxp,\gyp,\gxm).
	\end{equation}
	By exchanging $x,y$ we therefore also have 
	\begin{equation}
		\label{eqn: second expression B hardcore model}
		B  =\beta \sum_{\substack{\eta,x,y : \\ x\sim y, \eta|_{\bar{N}_y} = 0}, \\ \eta|_{\bar{N}_x} = 0}  m(\eta) c(\eta, \gyp) I(\eta,\gyp,\gxp,\gym).
	\end{equation}
	Summing these two expressions we get
	\[
	\begin{split}
		2B & = \beta\sum_{\substack{\eta,x,y : \\ x\sim y, \eta|_{\bar{N}_y} = 0}, \eta|_{\bar{N}_x} = 0}  m(\eta) \quadr*{ c(\eta, \gxp) I(\eta,\gxp,\gyp,\gxm)+c(\eta, \gyp) I(\eta,\gyp,\gxp,\gym)} \\
		& =  \beta\sum_{\substack{\eta,x,y : \\ x\sim y, \eta|_{\bar{N}_y} = 0}, \eta|_{\bar{N}_x} = 0}  m(\eta) \quadr*{ c(\eta, \gxp) - c(\eta, \gyp)} I(\eta,\gxp,\gyp,\gxm) \\
		& = 0,
	\end{split}
	\]
	where we used property $5.$ of Lemma \ref{lem: auxiliary lemma hardcore model conforti} and  that $c(\eta,\gxp)=c(\eta,\gyp) = \beta$ if $\eta|_{\bar{N}_y\cup\bar{N}_y} = 0$.
	\subparagraph*{Term C}
	We have 
	\[
	\begin{split}
		C &= \sum_{\substack{\eta\in\Omega\\x\in V}} m(\eta) c(\eta, \gxp) \quadr*{\tond*{1-\beta\abs*{\graf*{y: y\sim x, \eta|_{\bar{N}_y} = 0}}} J(\eta, \gxp,e,\gxm) + \beta J(\eta,\gxp,\gxp,e)} \\
		& = \sum_{\substack{\eta\in\Omega\\x\in V}} m(\eta) c(\eta, \gxp) \quadr*{\psi(\eta)-\psi(\gxp \eta)}^2 \\
		&\cdot  \Big[\tond*{1-\beta\abs*{\graf*{y: y\sim x, \eta|_{\bar{N}_y} = 0}}} \tond*{\theta\tond*{\rho(\eta),\rho(\eta)} + \theta\tond*{\rho(\eta),\rho(\gxp \eta)}}
		\\
		&+\beta \tond*{\theta\tond*{\rho(\gxp \eta),\rho(\gxp \eta)}+\theta \tond*{\rho(\eta),\rho(\gxp\eta)}}\Big]  \\
		& \geq \sum_{\substack{\eta\in\Omega\\x\in V}} m(\eta) c(\eta, \gxp) \theta\tond*{\rho(\eta),\rho(\gxp\eta)}\quadr*{\psi(\eta)-\psi(\gxp \eta)}^2 \tond*{\kappa_*+2M_\theta \overline{\kappa_*}} \\
		& = \tond*{\kappa_*+2M_\theta \overline{\kappa_*}} \curlyA(\rho,\psi).
	\end{split}
	\]

	This concludes the proof of the theorem.

	\subsection{Interacting random walks}

	\label{sec: Interacting random walks}
	Following Section $3$ of \cite{con-2022}, we now consider the case of interacting random walks. One motivation in this subsection is to find a discrete analogue of the following classical result.
	\begin{proposition}\label{prop:curv-gauss-bak-eme} Consider $\R^d$ equipped with the standard Euclidean distance $\dd$.
		Let $V\colon \R^d\to \R$ be \textbf{convex}, $\gamma_d$ be the law of a standard Gaussian in $\R^d$ and $Z>0$ be a normalizing constant so that $\frac{1}{Z}e^{-V} d\gamma_d$ is a probability measure. Then the metric measure space
		$\tond*{\R^d, \dd, \frac{1}{Z}e^{-V}d\gamma_d}$ has Ricci curvature $\ric \geq 1$ in the sense of the Lott--Sturm--Villani theory.
	\end{proposition}
	
	To find a discrete analogue, here the role of $\R^d$ is taken over by the discrete state space  $\Omega = \N^d$, while the Gaussian measure $\gamma_d$ is replaced by the multivariate Poisson distribution $\mu_\lambda$ given by the product measure of $d$ one-dimensional Poisson distribution of intensity $\frac{1}{\lambda}$, i.e. for $\eta \in \Omega$ we have
	\[
	\mu_\lambda\tond{\eta} = \prod_{i=1}^d e^{-\frac{1}{\lambda}} \frac{\lambda^{-\eta_i}}{\eta_i!}.
	\]
	It remains to define a Markov chain on this state space.
	We consider the set of moves $G$ containing  $\gip,\gim$ for $i\in [d]$, where
	\begin{align*}
		&\gip \eta = \eta + \mathbf{e_i},\\
		&\gim \eta = \eta - \mathbf{e_i} \, \mathds{1}_{\eta_i>0},
	\end{align*}
	and we denote as usual by $e$  the null move.
	Consider two potentials $V^+,V^-\colon \N^d \to \R$ and correspondingly define transition rates 
	\begin{align*}
		&c(\eta,\gip) = \exp\tond*{-\nabla_i^+ V^+(\eta)},\\
		&c(\eta,\gim) = \left\{ \begin{array}{ll}\exp\tond*{-\nabla_i^- V^-(\eta)} &\mbox{ if } \eta_i>0, \\
			0 &\mbox{ if } \eta_i=0,
		\end{array}\right.
	\end{align*}
	where we define $\nabla_i^{\pm} = \nabla_{\gamma_i^\pm}$.
	Then the reversible measure takes the form
	\[
	m = \frac{1}{Z}\exp(-V^+-V^-),
	\]
	where $Z$ is athe normalizing constant.  
	An interesting choice is  given by \begin{equation}\label{eq:v-min}
		V^-(\eta) = \sum_{i=1}^d \log(\lambda)\eta_i + \log(\eta_i!),
	\end{equation} which corresponds to $c(\eta,\gim) = \lambda \eta_i \indic{\eta_i>0} = \lambda \eta_i$. In this case, we write $V=V^+$ and the reversible measure becomes 
	\[
	\frac{1}{Z}e^{-V} d\mu_\lambda,
	\]
	which is reminiscent of the setting of Proposition \ref{prop:curv-gauss-bak-eme}. Therefore, to find an analogous discrete result, we are left to look for conditions on the potential $V$ that resemble convexity and yield positive entropic curvature of the corresponding Markov chain; we will do this in Section \ref{sec:exam-inter}, as an application of the main theorem of this section below.

	The first assumption that we make on the model is that for all $\eta\in \N^d$ and $i,j\in [d]$ we have
	\begin{align}
		\label{eqn: additional assumptions interacting random walks - 1}
		\nabla_i^+ c(\eta, \gjp) \leq 0;\\
		\label{eqn: additional assumptions interacting random walks - 2}
		\nabla_i^+ c(\eta,\gjm) \geq 0.
	\end{align}
	Notice that for $V^-$ as in \eqref{eq:v-min} the condition \eqref{eqn: additional assumptions interacting random walks - 2} is always satisfied. We remark that this assumption was not needed in \cite{con-2022}, but it will be useful later in the proof of the main theorem of this section to obtain some terms cancellations.
	
	Next, following \cite{con-2022}, we make the crucial assumptions that for all $\eta\in \N^d, i\in [d]$ the following quantities are non-negative: 
	\begin{alignat}{3}\label{eq:kappa_+_grid}
		&\kappa^+(\eta,i) \coloneqq -\nabla_i^+ c(\eta,\gip) \;\;& - \sum_{j\in[d], j\neq i} {\nabla_i^+ c(\eta,\gjm)}&\geq 0, \\
		\label{eq:kappa_-_grid}
		&\kappa^-(\eta,i) \coloneqq \nabla_i^+ c(\eta,\gim) &+ \sum_{j\in[d], j\neq i} {\nabla_i^+ c(\eta,\gjp)} \;\;& \geq 0.
	\end{alignat}
	Correspondingly we set
	\[
	\kappa_* = \inf_{\eta\in \N^d, i\in [d]} \kappa^+(\eta,i)+\kappa^-(\eta,i).
	\]
	It is natural to introduce the additional quantity
	\[
	\overline{\kappa_*} = \min\graf*{\inf_{\eta\in \N^d, i\in [d]}\kappa^+(\eta,i),   \inf_{\eta\in \N^d, i\in [d]} \kappa^-(\eta,i)}.
	\]
	As in \cite{con-2022} and in analogy with the previous examples, the assumptions in \eqref{eq:kappa_+_grid}, \eqref{eq:kappa_-_grid} are needed for the construction of appropriate contractive coupling rates; compared to \cite{con-2022}, the expressions are slightly simplified thanks to the additional assumptions \eqref{eqn: additional assumptions interacting random walks - 1}, \eqref{eqn: additional assumptions interacting random walks - 2}. With regard to the heuristic discussion in Section \ref{sec:coup-rates}, the quantities $\kappa_*$ and $\overline{\kappa_*}$ correspond respectively to \eqref{eq:inf-meeting-rates} and \eqref{eq:inf-meeting-rates-stronger}.
	
	It is also important to notice that this is the only example that we discuss where the cardinality of the state space is not finite and that, because of this,  there are some additional technical difficulties and not all the considerations of Section \ref{sec:gen-ineq} can be directly applied here.
	In this paper, to deal with the infinite cardinality of the state space, we proceed as in \cite{con-2022}  and make use of a localization argument, which we now briefly describe; more precisely, with this procedure and with $\phi$ and $\theta$ satisfying Assumption \ref{ass:mild-theta-phi}, we explain how to derive inequality \eqref{eq:conv-sob-ineq} from establishing \eqref{eq:main-ineq} for a localizing sequence of finite state space Markov chains. 
	Given an integer $N\ge 2$, let  $\Omega_N = \graf*{\eta \in \N^d \mid \eta_i \leq N \: \forall i\in [d] }$. On $\Omega_N,$ consider the Markov chain with generator $L_N$ described by the set of moves $G_N = \graf*{\gamma_i^{+,N}, \gamma_i^{-,N} \mid i\in [d] }$, where $\gamma_i^{+,N} (\eta) = \eta+\mathbf{e_i} \indic{\eta_i <N}$ and  $ \gamma_i^{-,N} (\eta) = \eta-\mathbf{e_i} \indic{\eta_i >0} = \gim\vert_{\Omega_N} \eta$ (and $\gim\vert_{\Omega_N}$ denotes the restriction of $\gim$ to $\Omega_N$), and by the transition rates $c_N(\eta,\gamma_i^{+,N}) = c(\eta, \gamma_i^+)\indic{\eta_i<N}$ and $c_N(\eta,\gamma_i^{-,N}) = c(\eta, \gamma_i^-)$. Clearly $\tond*{\gamma_i^{\pm,N}}^{-1} = \gamma_i^{\mp,N}$; moreover, as observed by Conforti, it is easy to check that this Markov chain is reversible with respect to the probability measure $m_N = \frac{m}{m(\Omega_N)}$ on $\Omega_N$ and \eqref{eqn: reversibility with moves} holds. Finally, denote by $\curlyB_N$ and $\curlyA_N$ the corresponding quantities $\curlyB$ and $\curlyA$ for this Markov chain, set $G_N^+=\graf*{\gamma_i^{+,N}\mid i\in [d] }$, $G_N^- = \graf*{ \gamma_i^{-,N} \mid i\in [d] } $, and, as usual, consider the enlarged set of moves $G^*_N=G_N\cup\graf{e}$.
	The main theorem of this section reads as follows.
	\begin{theorem}\label{thm:interacting}
		With the previous notation, suppose that for all $\eta\in \N^d$ and $i,j\in[d]$ the assumptions \eqref{eqn: additional assumptions interacting random walks - 1}, \eqref{eqn: additional assumptions interacting random walks - 2}, \eqref{eq:kappa_+_grid} and \eqref{eq:kappa_-_grid} are satisfied.
		Let also $\theta$ be a weight function satisfying Assumption \ref{ass: weak assumption theta}.
		Then  we have that
		\[
		\curlyB_N(\rho, \psi) \geq \tond*{\frac{\kappa_*}{2}+M_\theta \overline{\kappa_*}} \curlyA_N(\rho,\psi)
		\]
		for all integers $N\geq 2$ and for all  functions $\rho\colon \Omega_N \to \R_{>0}$ and  $\psi \colon \Omega_N \to \R$.
	\end{theorem}
	\begin{remark}
		With $\theta$ as in Assumption \ref{ass:mild-theta-phi}, by the discussion of Section \ref{sec:gen-ineq} the previous theorem allows to deduce the convex Sobolev inequality \eqref{eq:conv-sob-ineq} for all Markov chains $(\Omega_N, L_N, m_N)$ with uniform constant. As observed in \cite{con-2022}, this allows us to deduce the same convex Sobolev inequality for the original Markov chain by taking limits, when $\phi$ is lower bounded (as it is the case for $\phi = \phi_\alpha$ with $\alpha\in [1,2]$ in particular), cf. Corollary 2.1 of \cite{con-2022}.
	\end{remark}
	
	\begin{remark}[Comparison with \cite{con-2022}]
		In \cite[Thm. 3.1]{con-2022}, Conforti establishes inequality \eqref{eq:bak-eme-method} for the localizing sequence of Markov chain on $\Omega_N$ with a uniform constant $K$, and thus also  $\csi_\phi(2K)$ for the original Markov chain, with constant $K$ equal to
		\begin{itemize}
			\item $\frac{\kappa_*}{2}$ for general convex $\phi$ satisfying convexity of \eqref{eq:conf-conv};
			\item $\frac{\alpha \kappa_*}{2}$ for $\phi = \phi_\alpha$ with $\alpha \in (1,2]$.		
		\end{itemize}
		Compared to our assumptions on the model, he does not assume non-negativity in \eqref{eqn: additional assumptions interacting random walks - 1}, \eqref{eqn: additional assumptions interacting random walks - 2}; however, these additional assumptions are satisfied in the examples of Section \ref{sec:exam-inter}. By the discussion in Section \ref{sec:gen-ineq}, we therefore obtain a complementary result to \cite[Thm. 3.1]{con-2022}. 
	\end{remark}

	\subsubsection{Proof of  Theorem \ref{thm:interacting}}
	The proof adapts the one in \cite{con-2022}, with  slightly different notation and choices.
	Fix an integer $N\ge 2$ and consider the Markov chain described by the triple $(\Omega_N, L_N, m_N)$.
	Notice that we have 
	\begin{align*}
		S_N & \coloneqq \graf*{(\eta,\sigma)\in \Omega_N \times G_N \mid  c_N(\eta,\sigma) >0 } 
		\\
		& = \graf*{(\eta,\gipn) \mid \eta\in \Omega_N, i\in [d], \eta_i<N } \cup  \graf*{(\eta,\gimn) \mid \eta\in \Omega_N, i\in [d], \eta_i>0}.
	\end{align*}
	Notice that from our definitions it follows that if $\tond*{\eta,\gamma_i^{+,N}} \in S_N$ then $\gamma_i^{+,N} \eta = \gamma_i^+\vert_{\Omega_N} \eta$. To lighten the notation, with a slight abuse of notation, we will take advantage of this and drop the superscript $N$ in $\gamma_i^{\pm,N}$ (i.e. we just write $\gamma_i^\pm$), and we will also write $c_N(\eta,\gamma_i^{\pm}) = c_N(\eta,\gamma_i^{\pm,N})$. Similarly, for a function $\psi\colon \Omega_N \to \R$, a state $\eta\in \Omega_N$ and $i\in[d]$, we will write 
	$
	\nabla_i^\pm \psi(\eta) 
	$ instead of $\nabla_{\gamma_i^{\pm,N}} \psi(\eta)$. Again, this minor abuse of notation is justified by the fact that whenever $\nabla_i^+ \psi(\eta)$ appears in the computations below, it will be multiplied by a jump rate equal to $0$ if $\eta_i = N$.
	In analogy with \eqref{eq:kappa_+_grid}, for the localized Markov chain and for $(\eta,\gip)\in S_N$, we consider the quantity
	\[
	\kappa^{+,N}(\eta,i) \coloneqq -\nabla_i^{+} c_N(\eta,\gip) - \sum_{j\in[d], j\neq i} {\nabla_i^+ c_N(\eta,\gjm)},
	\]
	and we observe that
	\begin{equation}\label{eq:kappaNplus_kappa}
		\kappa^{+,N}(\eta,i) = \kappa^{+}(\eta,i) +c(\gip\eta,\gip)\indic{\eta_i = N-1} \geq \kappa^{+}(\eta,i) \geq 0,
	\end{equation}
	where the first equality is due to the fact that we have set $c_N(\tilde\eta,\gip) = 0 $ if $\tilde\eta_i =N$, as opposed to $c(\tilde \eta,\gip)$.
	Similarly, we define
	\[
	\kappa^{-,N}(\eta,i) \coloneqq \nabla_i^+ c_N(\eta,\gim) + \sum_{j\in[d], j\neq i} {\nabla_i^+ c_N(\eta,\gjp)} 
	\]
	and we notice that
	\begin{equation}\label{eq:kappaNminus_kappa}
		\kappa^{-,N}(\eta,i) = \kappa^{-}(\eta,i) -\sum_{{j\in[d], j\neq i}} \indic{\eta_j = N}\cdot{\nabla_i^+ c(\eta,\gjp)} \geq  \kappa^{-}(\eta,i)\geq 0,
	\end{equation}
	where the first equality follows from the fact that we have set $c_N(\tilde\eta,\gjp) = 0$ if $\tilde\eta_j = N$, as opposed to $c(\tilde\eta,\gjp)$,  and the first inequality is due to \eqref{eqn: additional assumptions interacting random walks - 1}.
	
	With these definitions, we are now ready to construct appropriate coupling rates, analogously to \cite{con-2022}.
	For $(\eta, \gip) \in S_N$ and $\gamma,\bar \gamma \in G_N^*$ set
	\[
	\ccpl(\eta,\gip \eta, \gamma, \bar{\gamma}) = \left\{
	\begin{array}{ll}
		\min \graf{c_N( \eta,\gamma), c_N(\gip \eta, \gamma)} & \mbox{ if } \gamma = \bar{\gamma}\neq e, \\
		\max \graf*{\nabla_i^+ c_N(\eta,\bar{\gamma}), 0 } & \mbox{ if } {\gamma} = \gip \mbox{ and } \bar{\gamma}\neq \gip,\gim,e,\\  
		\max \graf*{-\nabla_i^+ c_N(\eta,\gamma), 0 } & \mbox{ if } {\gamma} \neq \gip,\gim, e \mbox{ and } \bar{\gamma} =\gim,\\
		\kappa^{+,N}( \eta,i) & \mbox{ if } \gamma=\gip, \bar{\gamma} = e,\\
		\kappa^{-,N}(\eta,i) & \mbox{ if }
		{\gamma} = e, \bar{\gamma}=\gim ,\\
		0 & \mbox{ otherwise}.		 
	\end{array}
	\right.
	\]
	Next, notice that if $(\eta,\gim)\in S_N$ then $(\gim\eta,\gip)\in S_N$ and so we can set 
	$\ccpl(\eta,\gim\eta,\gamma,\bar{\gamma}) = \ccpl(\gim \eta, \gip\gim\eta,\bar{\gamma},\gamma) = \ccpl(\gim \eta, \eta,\bar{\gamma},\gamma)  $.

	By Lemma \ref{lem: lower bound for B with coupling rates}, to prove the theorem it suffices to show that
	\begin{equation}
		\begin{split}
			\label{eqn: lower bound for B with couplings - interacting}
			\frac{1}{4}\sum_{(\eta,\sigma)\in S_N}\sum_{\gamma,\bar{\gamma} \in G_N^*} m(\eta)c_N(\eta, \sigma) \ccpl(\eta,\sigma \eta,\gamma, \bar{\gamma}) J(\eta,\sigma, \gamma,\bar{\gamma}) 
			\geq  \tond*{\frac{\kappa_*}{2}+M_\theta \overline{\kappa_*}} \curlyA_N(\rho,\psi)
		\end{split}
	\end{equation}
	for all $\rho\colon \Omega_N \to \R_{>0}$ and  $\psi \colon \Omega_N \to \R$.

	The left hand side of equation \eqref{eqn: lower bound for B with couplings - interacting} reads 
	\begin{align*}
		&\frac{1}{4}{\sum_{(\eta,\gip)\in S_N}\sum_{\gamma,\bar{\gamma} \in G^*_N} m(\eta)c_N(\eta, \gip) \ccpl(\eta,\gip \eta,\gamma, \bar{\gamma}) J(\eta,\gip ,\gamma, \bar{\gamma})} \\
		+ & \frac{1}{4}{\sum_{(\eta,\gim)\in S_N}\sum_{\gamma,\bar{\gamma} \in G^*_N} m(\eta)c_N(\eta, \gim) \ccpl(\gim\eta, \eta,\bar{\gamma}, \gamma) J(\eta,\gim,\gamma, \bar{\gamma})}.
	\end{align*}

	Using reversibility \eqref{eqn: reversibility with moves} 
	and that
	$J(\eta, \sigma, \gamma, \bar{\gamma}) = J(\sigma \eta, \sigma^{-1},\bar{\gamma},\gamma)$ when $\sigma^{-1}\sigma \eta = \eta$, we get that the second summand is equal to the first and so we can rewrite our quantity as
	\[
	\frac{1}{2}{\sum_{(\eta,\gip)\in S_N}\sum_{\gamma,\bar{\gamma} \in G^*_N} m(\eta)c_N(\eta, \gip) \ccpl(\eta,\gip \eta,\gamma, \bar{\gamma}) J(\eta,\gip,\gamma, \bar{\gamma})}.
	\]
	
	With our explicit choice of coupling rates it follows that we can  write the left-hand side of \eqref{eqn: lower bound for B with couplings - interacting} as
	$\frac{1}{2}(\tilde{A}+\tilde{B}+\tilde{C}+{D})$, where 
	
	\begin{align*}
		\tilde{A} &= \sum_{(\eta,\gip)\in S_N}\sum_{\gamma\in G_N} m(\eta) c_N(\eta, \gip) \min \graf*{c_N(\eta,\gamma),c_N(\gip \eta, \gamma)} J(\eta,\gip,\gamma,\gamma) ,\\
		\tilde{B} &= \sum_{(\eta,\gip)\in S_N} \sum_{\substack{\gamma\in G_N\\ \gamma \neq \gamma_i^\pm}}  m(\eta) c_N(\eta, \gip) \max\graf{\nabla_i^+c_N(\eta,\gamma),0} J(\eta, \gip, \gip,\gamma) ,\\
		\tilde{C} &= \sum_{(\eta,\gip)\in S_N}\sum_{\substack{\gamma\in G_N\\ \gamma \neq \gamma_i^\pm}}  m(\eta) c_N(\eta, \gip) \max\graf{-\nabla_i^+c_N(\eta,\gamma),0} J(\eta, \gip, \gamma,\gim),\\
		{D} &= \sum_{(\eta,\gip)\in S_N} m(\eta) c_N(\eta, \gip)\quadr*{\kappa^{+,N}(\eta, i) J(\eta, \gip,\gip,e)+\kappa^{-,N}(\eta, i) J(\eta, \gip,e,\gim)}.
	\end{align*}
	This is then lower bounded (since $J\geq I$ by \eqref{eqn: J vs I}) by $\frac{1}{2}(A+B+C+D)$ where 
	\begin{align*}
		{A} &= \sum_{(\eta,\gip)\in S_N}\sum_{\gamma\in G_N} m(\eta) c_N(\eta, \gip) \min \graf*{c_N(\eta,\gamma),c_N(\gip \eta, \gamma)} I(\eta,\gip,\gamma,\gamma) ,\\
		{B} &= \sum_{(\eta,\gip)\in S_N} \sum_{\substack{\gamma\in G_N\\ \gamma \neq \gamma_i^\pm}}   m(\eta) c_N(\eta, \gip) \max\graf{\nabla_i^+c_N(\eta,\gamma),0} I(\eta, \gip, \gip,\gamma) ,\\
		{C} &= \sum_{(\eta,\gip)\in S_N} \sum_{\substack{\gamma\in G_N\\ \gamma \neq \gamma_i^\pm}}   m(\eta) c_N(\eta, \gip) \max\graf{-\nabla_i^+c_N(\eta,\gamma),0} I(\eta, \gip, \gamma,\gim).
	\end{align*}
	Next, using the assumptions \eqref{eqn: additional assumptions interacting random walks - 1} and \eqref{eqn: additional assumptions interacting random walks - 2} we can rewrite the expressions of $A,B,C$ as
	\begin{align*}
		\begin{split}
			A & = \sum_{(\eta, \gip)\in S_N} \sum_{j\in [d]} m(\eta) c_N(\eta, \gip) c_N(\gip \eta, \gjp) I(\eta, \gip, \gjp,\gjp) \\
			& + \sum_{(\eta, \gip)\in S_N} \sum_{j\in [d]} m(\eta) c_N(\eta, \gip) c_N( \eta, \gjm) I(\eta, \gip, \gjm,\gjm),
		\end{split}\\
		B & = \sum_{(\eta, \gip)\in S_N} \sum_{\substack{j\in [d]\\ j\neq i}} m(\eta) c_N(\eta,\gip) \nabla_i^+ c_N(\eta,\gjm) I(\eta,\gip,\gip,\gjm),\\
		C & = \sum_{(\eta, \gip)\in S_N} \sum_{\substack{j\in [d]\\ j\neq i}} m(\eta) c_N(\eta,\gip) \graf{-\nabla_i^+ c_N(\eta,\gjp)} I(\eta,\gip,\gjp,\gim).
	\end{align*}
	
	Finally, we notice that we can write
	\begin{equation*}
		\begin{split}
			\curlyA_N(\rho,\psi) = & \frac{1}{2}\sum_{(\eta,\gip)\in S_N}  m(\eta)c_N(\eta, \gip) \theta \tond*{\rho(\eta),\rho(\gip \eta)} \quadr*{\psi(\eta)-\psi(\gip \eta)}^2 \\
			+  &\frac{1}{2}\sum_{(\eta,\gim)\in S_N} m(\eta)c_N(\eta, \gim) \theta \tond*{\rho(\eta), \rho(\gim \eta)}
			\quadr*{\psi(\eta)-\psi(\gim \eta)}^2\\
			= & \sum_{(\eta,\gip)\in S_N}  m(\eta)c_N(\eta, \gip) \theta \tond*{\rho(\eta),\rho(\gip \eta)} \quadr*{\psi(\eta)-\psi(\gip \eta)}^2
		\end{split}
	\end{equation*}	
	using reversibility \eqref{eqn: reversibility with moves} again for the last equality.
	
	In what follows, we will show that $D\ge(2 M_\theta \overline{\kappa_*}+\kappa_*) \curlyA_N(\rho,\psi)$ and that $A=B=C=0$, thus verifying \eqref{eqn: lower bound for B with couplings - interacting} and concluding the proof of the theorem.
	
	\paragraph*{Term D}
	We have 
	\[
	J(\eta,\gip,\gip,e) = \tond*{\theta \tond*{\rho(\gip \eta),\rho(\gip \eta)} + \theta(\rho(\eta),\rho(\gip \eta))} \quadr*{\psi(\eta)-\psi(\gip \eta)}^2
	\]
	and
	\[
	J(\eta,\gip,e,\gim) = \tond*{\theta\tond*{\rho(\eta),\rho(\eta)}+\theta(\rho(\eta),\rho(\gip \eta))} \quadr*{\psi(\eta)-\psi(\gip \eta)}^2
	\]
	and so, remembering also \eqref{eq:kappaNplus_kappa} and \eqref{eq:kappaNminus_kappa},
	\begin{align*}
		&\kappa^{+,N}(\eta, i) J(\eta, \gip,\gip,e)+\kappa^{-,N}(\eta, i) J(\eta, \gip,e,\gim)
		\\ 
		\geq &\graf*{\overline{\kappa_*} \quadr*{\theta \tond*{\rho(\gip \eta),\rho(\gip \eta)} + \theta\tond*{\rho(\eta),\rho(\eta)}} + \kappa_* \theta \tond*{\rho(\eta),\rho(\gip \eta)}}\quadr*{\psi(\eta)-\psi(\gip \eta)}^2
		\\ 
		\geq & (2M_\theta\overline{\kappa_*}+\kappa_*) \theta \tond*{\rho(\eta),\rho(\gip \eta)}\quadr*{\psi(\eta)-\psi(\gip \eta)}^2.
	\end{align*}
	Therefore it follows that 
	\[
	D\geq (2 M_\theta \overline{\kappa_*}+\kappa_*) \curlyA_N(\rho,\psi).
	\]
	\paragraph*{Other terms}
	We now show that each one of the other terms is $0$, concluding the proof of the theorem.
	To show that $A=B=C=0$ we proceed similarly to \cite{con-2022}.
	It is useful to have an auxiliary lemma.
	\begin{lemma}
		\label{lem: auxiliary lemma interacting random walks}
		For all $\eta\in \N^d$, $\tilde \eta \in \Omega_N$ and $i,j\in [d]$
		\begin{enumerate}
			\item $c(\eta, \gjp) c(\gjp \eta, \gip) = c(\eta, \gip) c(\gip \eta, \gjp)$ and similarly \\  $c_N(\tilde\eta, \gjp) c_N(\gjp \tilde\eta, \gip) = c_N(\tilde\eta, \gip) c_N(\gip\tilde \eta, \gjp)$.
			\item $c(\eta,\gim) \nabla_i^- c(\eta, \gjm) = c(\eta,\gjm) \nabla_j^- c(\eta, \gim)$ and similarly \\$c_N(\tilde\eta,\gim) \nabla_i^- c_N(\tilde\eta, \gjm) = c_N(\tilde\eta,\gjm) \nabla_j^- c_N(\tilde\eta, \gim)$.
			\item $c(\eta,\gip) \nabla_i^+ c(\eta, \gjp) = c(\eta,\gjp) \nabla_j^+ c(\eta, \gip)$ and similarly when $\eta_i,\eta_j<N$ \\ $c_N(\tilde\eta,\gip) \nabla_i^+ c_N(\tilde\eta, \gjp) = c_N(\tilde\eta,\gjp) \nabla_j^+ c_N(\tilde\eta, \gip)$.
			\item $I(\gjp\eta,\gip,\gjm,\gjm) = -I(\eta,\gip,\gjp,\gjp)$.
			\item $I(\gim\eta,\gip,\gip,\gjm) = -I(\gjm\eta,\gjp,\gjp,\gim)$ if $\eta_i,\eta_j>0$.
			\item $I(\eta,\gip,\gjp,\gim) = - I(\eta,\gjp,\gip,\gjm)$.

		\end{enumerate}
	\end{lemma}
	
	\begin{proof}[Proof of Lemma]
		Statements 1--3 were already observed in the proof of \cite[Thm. 3.1]{con-2022} and are easy to check, while statements 4--6 are immediate from the definitions and the assumptions on the model.
	\end{proof}
	
	\subparagraph*{Term A}
	Using the reversibility property \eqref{eqn: reversibility with moves} in the second summand defining $A$ with
	\[
	F(\eta, \sigma) = \mathds{1}_{G_N^-}(\sigma) \sum_{i\in [d]} c_N(\eta, \gip) I(\eta,\gip,\sigma,\sigma).
	\] we find that
	\begin{align*}
		& \sum_{\substack{\eta\in \Omega_N\\ i,j\in[d]}} m(\eta) c_N(\eta, \gip) c_N( \eta, \gjm) I(\eta, \gip, \gjm,\gjm) \\
		= &  \sum_{\substack{\eta\in \Omega_N\\ i,j\in[d]}} m(\eta) c_N(\eta, \gjp) c_N(\gjp\eta, \gip)  I(\gjp\eta, \gip, \gjm,\gjm) \\
		= & -\sum_{\substack{\eta\in \Omega_N\\ i,j\in[d]}} m(\eta)c_N(\eta, \gip) c_N(\gip \eta, \gjp) I(\eta,\gip,\gjp,\gjp),
	\end{align*}
	where in the last equality we have used properties $1.$ and $4.$ of Lemma \ref{lem: auxiliary lemma interacting random walks}. 
	This implies that $A=0$, as desired.

	\subparagraph*{Term B}
	First, using the reversibility property \eqref{eqn: reversibility with moves} with 
	\[
	F(\eta,\sigma) = \mathds{1}_{G^+_N}(\sigma) \sum_{\gamma\in G^-_N, \gamma \neq \sigma,\sigma^{-1}} \nabla_\sigma c_N(\eta, \gamma) I(\eta,\sigma,\sigma,\gamma),
	\]
	and the fact that $\gip\gim\eta =\eta$ if $\eta_i>0$ we find
	\begin{equation}
		\label{eq:temp-expr-b-inter}
		\begin{split}
			B & = \sum_{\substack{\eta\in \Omega_N,\\i\neq j\in [d]}}m(\eta) c_N(\eta,\gim) {\nabla_i^+ c_N(\gim\eta, \gjm)} I(\gim\eta,\gip,\gip,\gjm) 
			\\
			& = \sum_{\substack{\eta\in \Omega_N,\\i\neq j\in [d]}} m(\eta) c_N(\eta,\gim) \graf{-\nabla_i^- c_N(\eta, \gjm)} I(\gim\eta,\gip,\gip,\gjm).
		\end{split}
	\end{equation}
	By exchanging $i,j$ first and then using properties $5.$ and $2.$ of Lemma \ref{lem: auxiliary lemma interacting random walks} we deduce that
	\begin{align*}
		B & = \sum_{\substack{\eta\in \Omega_N,\\i\neq j\in [d]}} m(\eta) c_N(\eta,\gjm) \graf{-\nabla_j^- c_N(\eta, \gim)} I(\gjm\eta,\gjp,\gjp,\gim)
		\\
		& = -\sum_{\substack{\eta\in \Omega_N,\\i\neq j\in [d]}} m(\eta) c_N(\eta,\gjm) \graf{-\nabla_j^- c_N(\eta, \gim)} I(\gim\eta,\gip,\gip,\gjm)
		\\
		& = -\sum_{\substack{\eta\in \Omega_N,\\i\neq j\in [d]}} m(\eta) c_N(\eta,\gim) \graf{-\nabla_i^- c_N(\eta, \gjm)} I(\gim\eta,\gip,\gip,\gjm).
	\end{align*}
	Comparing this with the expression of $B$ in \eqref{eq:temp-expr-b-inter} we deduce that $B=-B$, hence $B=0$.
	\subparagraph*{Term C}
	We have
	\begin{align*}
		C &= \sum_{\substack{\eta\in \Omega_N,\\i\neq j\in [d]}} m(\eta) c_N(\eta,\gip) \graf{-\nabla_i^+ c_N(\eta,\gjp)} I(\eta,\gip,\gjp,\gim)
		\\
		&= \sum_{\substack{\eta\in \Omega_N,\\i\neq j\in [d]}} m(\eta) c_N(\eta,\gjp) \graf{-\nabla_j^+ c_N(\eta,\gip)} I(\eta,\gjp,\gip,\gjm)
		\\
		&= -\sum_{\substack{\eta\in \Omega_N,\\i\neq j\in [d]}} m(\eta) c_N(\eta,\gjp) \graf{-\nabla_j^+ c_N(\eta,\gip)} I(\eta,\gip,\gjp,\gim)
		\\
		&= -\sum_{\substack{\eta\in \Omega_N,\\i\neq j\in [d]}} m(\eta) c_N(\eta,\gip) \graf{-\nabla_i^+ c_N(\eta,\gjp)} I(\eta,\gip,\gjp,\gim)
	\end{align*}
	where  we have exchanged $i,j$ in the first equality and used properties $6.$ and $3.$ of Lemma \ref{lem: auxiliary lemma interacting random walks} in the last two.
	This shows that $C=-C$, hence $C=0$.
	
	This concludes the proof of the theorem.

	\subsubsection{Examples of interacting random walks}\label{sec:exam-inter}
	
	As anticipated, as an application of Theorem \ref{thm:interacting} and looking for a discrete analogue of Proposition \ref{prop:curv-gauss-bak-eme}, we now revisit some particular examples of interacting random walks considered in \cite{con-2022} .
	In this subsection, we stick to the particular choice of $V^-$ given in \eqref{eq:v-min} (for which \eqref{eqn: additional assumptions interacting random walks - 2} is satisfied) and we simply write $V = V^+$. Notice first of all that our assumption \eqref{eqn: additional assumptions interacting random walks - 1} can be written equivalently as
	\begin{equation}\label{eq:discr-hess-non-neg}
		\nabla_i^+ \nabla_j^+ V (\eta) \geq 0 \quad \text{ for all } \eta\in \N^d, i,j\in [d].
	\end{equation}
	
	Interestingly, while in Proposition \ref{prop:curv-gauss-bak-eme} the key assumption was the convexity of $V$  (i.e. the positive semi-definiteness of the Hessian $\nabla^2 V$), here we see that the non-negativity of the entries of a ``discrete Hessian'' of the potential $V$ comes into play.
	
	Under this assumption \eqref{eq:discr-hess-non-neg}, the conditions \eqref{eq:kappa_+_grid} and \eqref{eq:kappa_-_grid} were checked by Conforti \cite{con-2022}, and in particular we have the following
	\begin{corollary}[Cor 3.2 of \cite{con-2022}]
			With the notation of this section, suppose that  \eqref{eq:discr-hess-non-neg} holds and that for all $\eta\in \N^d, i \in [d]$
		\[
		\lambda - \sum_{\substack{j = 1,\\j\neq i}}^d \quadr*{ e^{-\nabla_j^+ V(\eta)} - e^{-\nabla_j^+ V(\gip\eta)}} \geq 0.
		\]
		Then, the assumptions of Theorem \ref{thm:interacting} are satisfied and 
		\[
		\kappa_* = \inf_{\substack{\eta\in \N^d\\ i \in [d]}} \lambda + \quadr*{ e^{-\nabla_i^+ V(\eta)} - e^{-\nabla_i^+ V(\gip\eta)}} - \sum_{\substack{j = 1,\\j\neq i}}^d \quadr*{ e^{-\nabla_j^+ V(\eta)} - e^{-\nabla_j^+ V(\gip\eta)}}.
		\]
		If in addition 
		\begin{equation}\label{eq:kappa-second-inter}
			\min_{i\in [d]} \lambda - \sum_{\substack{j = 1\\ j\neq i}}e^{-\nabla_j^+ V(\mathbf{0})} \geq 0
		\end{equation}
		then $\kappa_*$ is bounded from below  by the expression in \eqref{eq:kappa-second-inter}.
	\end{corollary}

	We consider now a particular example of potential $V$ satisfying \eqref{eq:discr-hess-non-neg}. Given a function $h\colon \R_{\ge 0} \to \R$ and some $\beta>0$, set
	\[
	V(\eta) \coloneqq \beta h\tond*{\abs*{\eta}}, 
	\]
	where $\abs*{\eta} = \sum_{i} {\eta_i}$ for $\eta\in \N^d$.
	Using the notation $\nabla^+ h(m) \coloneqq h(m+1)-h(m)$, Conforti observed the following
	\begin{corollary}[Cor 3.1 of \cite{con-2022}]
		With the notation of this section, suppose that $h$ is \textbf{convex} and that 
		\[
		\inf_{m\in \N} \lambda -(d-1) \quadr*{e^{-\beta \nabla^+ h(m)} - e^{-\beta \nabla^+ h(m+1)} } \geq 0.
		\]
		Then the assumptions of Theorem \ref{thm:interacting} are satisfied and 
		\[
		\kappa_* = \inf_{m\in \N} \lambda -(d-2) \quadr*{e^{-\beta \nabla^+ h(m)} - e^{-\beta \nabla^+ h(m+1)} }.
		\]
		In particular, if $h(1)>h(0)$ and 
		\[
		\beta \geq \frac{\log(d-1) - \log(\lambda)}{h(1)-h(0)}
		\]
		then the assumptions of Theorem \ref{thm:interacting} are satisfied with 
		\[
		\kappa_* \ge  \lambda -(d-2) e^{-\beta \nabla^+ h(0)} .
		\]
	\end{corollary}
	Interestingly, as before, a notion of convexity of the potential is naturally involved.
	Note  also  that the necessary condition of a lower bound on $\beta$ means this is a non perturbative criterion,  since the resulting reversible measure is far from being a product measure. It is known that product measures behave well with the entropic curvature, i.e. they tensorize (see Theorem 6.2 of \cite{erb-maa-2012}). Therefore, it is particularly interesting to have conditions implying positive entropic curvature for a Markov chain whose stationary measure is not close to a product measure.

	\section{Couplings and coarse Ricci curvature}\label{sec:coup-curv}
	
	In this section, we recall some well-known related definitions of curvature for discrete and continuous time Markov chains (generally referred to as ``coarse Ricci curvature''), give some natural generalizations and discuss the relations among them and with the concept of coupling. As an application, we show that in all the examples of Section \ref{sec:applications} (except for Section \ref{sec: Interacting random walks} - which was however covered in  \cite[Sec. 3.3]{con-2022} - since we restrict to the case of finite state space Markov chains), when the assumptions of the respective main theorems are satisfied then the coarse curvature is positive, and for all starting probability densities we have exponential contraction of the $p$-Wasserstein distances along the heat flow $P_t$ (see the precise details later). 
	The coarse Ricci curvature was first introduced by Ollivier for discrete time Markov chains (see \cite{oll-2009, oll-2010}) and later modified to apply to continuous time models (see \cite{lin-lu-yau-2011, veysseire_coarse_cts}). 
	Compared to the entropic Ricci curvature, it is often easier to establish positive curvature for this notion. However, it is not known whether positive coarse Ricci curvature implies some functional inequalities, and in particular the modified log-Sobolev inequality (see Section \ref{sec:conjecture} for a discussion).
	
	Throughout this section, we switch to a more standard notation, and we don't employ the description of the Markov chain in terms of its allowed moves.
	As anticipated, we always assume in this section that we are working with irreducible and reversible Markov chains on a finite state space. We use the letter $\Omega$ for the state space, $x,y,v,w,z$ for elements of $\Omega$, $P$ for a stochastic transition matrix and $L,Q$ for a generator $L$ with transition rates $Q$, so that the action of $L$ on a function $\psi \colon \Omega \to \R$ is given by
	\[
	L\psi(x) = \sum_{y\in \Omega} Q(x,y) \tond*{\psi(y) - \psi(x)}.
	\]
	Here we have $Q(x,y)\geq 0$ and we do not assume $Q(x,x) = 0$, and in fact we often change the value of $Q(x,x)$ (without loss of generality) depending on convenience; on the other hand, we sometimes identify $L$, as a linear operator, with a matrix, in which case by construction $L(x,x)= -\sum_{y\neq x} Q(x,y)$.
	We typically identify measures with row vectors: in particular, let $\delta_x$ be the row vector with entry $1$ corresponding to $x$ and $0$ everywhere else, which is identified with the Dirac measure at $x$; on the other hand, densities with respect to $\pi$ and other functions on the state space are identified with column vectors. We also introduce a simple graph structure, where  $x\sim y$ if and only if $x\neq y$ and  $P(x,y)>0$ or $Q(x,y)>0$ respectively, and correspondingly consider the unweighted graph distance $d$. With respect to this graph distance $d$, we will consider the $p$-Wasserstein  distances $W_p$. Couplings for the transition measures/rates from starting points $x\neq y\in \Omega$ will be described by non-negative functions $\Pi\tond*{x,y,\cdot,\cdot}, C\tond*{x,y,\cdot,\cdot} \colon \Omega \times \Omega \to \R_{\geq 0} $ respectively in discrete and continuous time, so that 
	\begin{align}
		&\begin{cases}
			\sum_{w\in \Omega}\Pi (x,y,w,z) &= \:P(y,z) \text{ for all } z\in \Omega,\\
			\sum_{z\in \Omega}\Pi (x,y,w,z) &= \:P(x,w) \text{ for all } w\in \Omega,
		\end{cases}\\
		&\begin{cases}
			\sum_{w\in \Omega} C (x,y,w,z) &=\: Q(y,z) \text{ for all } z\in \Omega,\\
			\sum_{z\in \Omega} C (x,y,w,z) &=\: Q(x,w) \text{ for all } w\in \Omega.
		\end{cases}
	\end{align}
	Similarly to what we observed in Section \ref{sec:coup-rates}, for fixed $x\neq y\in \Omega$ the set of such admissible couplings is non-empty, provided that in continuous time one redefines $Q(x,x)$ appropriately without loss of generality: notice indeed that the existence of the coupling rates implies (by summing over $w,z\in \Omega$) that 
	$\sum_w Q(x,w) = \sum_z Q(y,z) \eqqcolon Z>0$. If this holds, the ``product'' coupling rates $C(x,y,w,z) = \frac{1}{Z} Q(x,w) Q(y,z)$ are admissible.
	Notice also that, to compare with the notation of the previous sections of this work, we  could write, for states $x,y,w,z\in \Omega$ and enlarged set of moves $G^*$,
	\[C(x,y,w,z) = \sum_{\substack{\gamma,\bar\gamma \in G^*\\ \gamma x = w, \bar\gamma y = z}} \ccpl(x,y,\gamma, \,\bar \gamma).\]
	\subsubsection*{Preliminaries definitions}
	Before turning to a detailed discussion of the coarse Ricci curvature, we recall a few definitions from optimal transport, which will be needed in the sequel.  Throughout this section, we denote by $\pp(X)$ the set of probability measures on a space $X$. 
	
	In what follows, we let $X, Y$ be two finite sets. Given two probability measures $\mu \in \pp(X)$, $\nu\in\pp(Y)$, we denote by $\Gamma(\mu,\nu)$ the family of couplings between $\mu$ and $\nu$, i.e. the family of probability measures $\gamma\in \pp\tond*{X\times Y}$ having marginals $\mu$ and $\nu$ respectively.
	
	Given a cost function $c\colon X\times Y\to \R_{\ge 0}$, the optimal transport cost $\mathcal{T}_c(\mu,\nu)$ between two probability measures $\mu\in\pp(X), \nu\in \pp(Y)$ is defined by
	\begin{equation}
		\mathcal{T}_c(\mu,\nu) = \inf_{\gamma\in\Gamma(\mu,\nu)} \sum_{x\in X,y\in Y}c(x,y)\gamma(x,y).
	\end{equation}
	
	When $X=Y$ and we are given a distance $d$ on $X$, we can define the Wasserstein distance of any order $p\in(1,\infty)$ as follows: for $\mu,\nu\in \pp(X)$ 
	\begin{equation*}
		W_p^p(\mu,\nu) = \mathcal{T}_{d^p}(\mu,\nu) = \inf_{\gamma\in\Gamma(\mu,\nu)} \sum_{x\in X,y\in Y}d(x,y)^p\gamma(x,y).
	\end{equation*}
	When discussing the coarse Ricci curvature of a Markov chain with finite state space $\Omega$, we will consider the Wasserstein distances with respect to the natural graph distance $d$ mentioned before.
	
	Finally, we will consider also the total variation distance between two probability measures $\mu,\nu\in \pp(X)$, which is defined by
	\begin{equation*}
		\tv{\mu- {\nu}} = \frac12\sum_{x\in X}\abs*{\mu(x)-\nu(x)}\in [0,1].
	\end{equation*}
	
	\subsection{Discrete time}
	An important notion of curvature for Markov chains is the \emph{coarse Ricci curvature} introduced by Ollivier (see \cite{oll-2009} and \cite{oll-2010}).
	\begin{definition}
		Given $p\geq 1$ and $x\neq y$, we say that the Markov chain has (discrete time) $p$-coarse Ricci curvature $\kollp(x,y) $ in direction $(x,y)$ if
		\[
		W_p\tond*{\delta_x P, \delta_y P} = (1-\kollp(x,y)) d(x,y).
		\]
	\end{definition}
	\begin{remark}
		Ollivier focused in particular on the case $p=1$; in this paper, however, it will be useful to consider also other values of $p$.
	\end{remark}
	We also give the following definition, inspired by the properties of the couplings constructed in the previous sections.
	\begin{definition}
		For $x\neq y$, we define the (discrete time) $\infty$-coarse Ricci curvature $\kollinf(x,y)$ in direction $(x,y)$ to be the supremum of all $K\in \R$ such that there exists a coupling $\Pi(x,y,\cdot,\cdot)$ of $\tond*{\delta_x P, \delta_y P}$ satisfying:
		\begin{itemize}
			\item $\sum_{w,z\in\Omega} \Pi(x,y,w,z)\indic{d(w,z)> d(x,y)} = 0$;
			\item $\sum_{w,z\in\Omega} \Pi(x,y,w,z){ d(w,z)} \leq (1-K) d(x,y)$.
		\end{itemize}
		We use the convention that $\sup \,\emptyset = -\infty$.	
	\end{definition} 
	\begin{remark}
		We have that $\kollinf(x,y) \in \R_{\geq 0} \cup\graf{-\infty}$; if $\kollinf(x,y)\geq 0$ then the supremum in the definition  is attained.
		Notice that, equivalently,  $\kollinf(x,y) \geq 0$  means that there exists a coupling $(X,Y)$ of the one-step probability distributions $(\delta_x P, \delta_y P)$ which proves $\kollone(x,y) \geq \kollinf(x,y)$ (i.e. $\expe{d(X,Y)}\leq \tond*{1-\kollinf(x,y)}d(x,y)$) and at the same time satisfies $d(X,Y) \leq d(x,y)$ almost surely.
	\end{remark}
	
	For $p\in [1,\infty]$ we write 
	\[
	\ricollp \geq K
	\]
	if $\kollp(x,y)\geq K$ for all $x\neq y \in \Omega$.
	The next proposition collects some useful results.
	\begin{proposition}\label{prop:properties-kollp}
		The following hold:
		\begin{enumerate}
			\item\label{it:coarseDiscreteNeighbours} For $p\in [1,\infty]$, if $\kollp(x,y)\geq K$ for all $x\sim y$ then $\ricollp\geq K$.
			\item \label{it:coarseDiscreteDifferentP} For $1\leq p \leq q <\infty$  we have $\kollp(x,y) \geq  K_{dc,q}(x,y)$. Moreover if $x\sim y$  we have 
			\[
			\kollp(x,y) \geq 1-\tond*{1-\kollinf(x,y)}^{\frac{1}{p}}\geq \frac{\kollinf(x,y)}{p}.
			\]
			\item \label{it:coarseDiscreteInfinityLowerBound}
			If $x\sim y$ and $\lim_{p\to \infty} \kollp(x,y) \geq 0$ we have \[
			\kollinf(x,y)  \geq 1-e^{-\limsup_{p\to \infty} p \cdot \kollp(x,y)}.
			\]
			\item \label{it:CoarseDiscreteContraction} For $p\in[1,\infty)$, if $\ricollp\geq K$ then for any starting probability measures $\mu,\nu$ and $n\geq 0$ we have
			\[
			W_p(\mu P^n, \nu P^n) \leq (1-K)^n W_p(\mu,\nu).
			\]
		\end{enumerate}
		
	\end{proposition}
	\begin{proof}
		\begin{enumerate}
			\item If $p<\infty$, this is done as in \cite[Prop. 19]{oll-2009}: suppose $d(x,y)=n$ and let $x=z_0\sim z_1 \sim \ldots \sim z_n =y$. Then
			\[
			W_p(\delta_x P,\delta_yP) \leq \sum_{i=0}^{n-1} W_p\tond*{\delta_{z_{i-1}}P, \delta_{z_{i}}P} \leq (1-K) n = (1-K) d(x,y).
			\]
			Suppose now that $p = \infty$: if $K = -\infty$ the conclusion is trivial, hence assume that $ K\geq 0$. Let again  $n=d(x,y)$ and $x=z_0\sim z_1 \sim \ldots \sim z_n =y$: we prove the claim by induction over $n$. The base case $n=1$ follows directly by the assumption. Now suppose $n>1$ and that the inductive hypothesis holds. Let $\Pi(x,z_{n-1}, \cdot, \cdot)$ and $\Pi(z_{n-1},y,\cdot,\cdot)$ be such that
			\begin{align*}
				&\sum_{v,w\in \Omega} \Pi(x,z_{n-1}, v, w) \indic{d(v,w)> d(x,z_{n-1})} = 0,
				\\
				& \sum_{v,w\in \Omega} \Pi(x,z_{n-1}, v, w) d(v,w) \leq (1-K) d(x,z_{n-1}), 
				\\
				& \sum_{v,w\in \Omega} \Pi(z_{n-1},y, v, w) \indic{d(v,w) > 1} = 0,
				\\
				& \sum_{v,w\in \Omega} \Pi(z_{n-1},y, v, w) d(v,w) \leq (1-K).
			\end{align*}
			By the Gluing lemma \cite{vil-2009}, there exists 
			$\hat{\Pi}(\cdot,\cdot,\cdot)=\hat{\Pi}(x,z_{n-1},y,\cdot,\cdot,\cdot) \in \pp(\Omega \times \Omega \times \Omega)$ such that $p_{1,2}\#\hat{\Pi} = \Pi(x,z_{n-1},\cdot,\cdot) $ and $p_{2,3}\#\hat{\Pi} = \Pi(z_{n-1},y,\cdot,\cdot)$, where $p_{i,j}$ is the projection on coordinates $i,j$ and $\#$ denotes the pushforward of a measure via a map. The measure $\Pi(x,y,\cdot,\cdot) \coloneqq p_{1,3} \# \hat{\Pi}$ then realizes a coupling with the desired properties, since, given that
			$d(v,w) \leq d(v,s) + d(s,w) $, we have
			\begin{align*}
				&\sum_{v, w\in \Omega}\Pi(x,y,v,w) d(v,w) 
				\\
				=& \sum_{v, w, s\in \Omega}\hat{\Pi}(x,z_{n-1},y,v,s,w) d(v,w)
				\\
				\leq &\sum_{v, w, s\in \Omega}\hat{\Pi}(x,z_{n-1},y,v,s,w)  \tond*{d(v,s) + d(s,w)}
				\\
				=& \sum_{v, s\in \Omega}\Pi(x,z_{n-1},v,s) {d(v,s)} + \sum_{s, w\in \Omega}\Pi(z_{n-1},y,s,w) {d(s,w)}
				\\
				\leq & (1-K) (d(x,z_{n-1}) + d(z_{n-1},y))
				\\
				= & (1-K) d(x,y),
			\end{align*}
			and similarly 
			\begin{align*}
				&\sum_{v, w\in \Omega}\Pi(x,y,v,w) \indic{d(v,w)>d(x,y)} 
				\\
				= & \sum_{v, w, s\in \Omega}\hat{\Pi}(x,z_{n-1},y,v,s,w) \indic{d(v,w)>d(x,y)}
				\\
				\leq &\sum_{v, w, s\in \Omega}\hat{\Pi}(x,z_{n-1},y,v,s,w)  \indic{d(v,s)+d(s,w) > d(x,z_{n-1})+d(z_{n-1}, y)}
				\\
				\leq &\sum_{v, w, s\in \Omega}\hat{\Pi}(x,z_{n-1},y,v,s,w) \tond*{ \indic{d(v,s) > d(x,z_{n-1})}+
					\indic{d(s,w) > d(z_{n-1}, y)} }
				\\
				=& \sum_{v, s\in \Omega}\Pi(x,z_{n-1},v,s) \indic{d(v,s) > d(x,z_{n-1})} + \sum_{s, w\in \Omega}\Pi(z_{n-1},y,s,w)\indic{d(s,w) > d(z_{n-1}, y)} 
				\\
				= &\, 0.
			\end{align*}
			\item The first statement follows by the inequality $W_p(\mu, \nu) \leq W_q(\mu,\nu)$ for $1\leq p\leq q<\infty$. For the second statement, suppose that $x\sim y$ and $K \coloneqq \kollinf(x,y)\geq 0$ . Let $\Pi$ be the optimal coupling in the definition of $\kollinf(x,y)$. Then notice that
			\begin{align*}
				W_p(\delta_x P, \delta_y P) &\leq \quadr*{1-K}^{\frac{1}{p}} \leq 1-\frac{K}{p},
			\end{align*}
			from which the conclusion follows.
		
            \item Let $K = \limsup_{p\to \infty} p \cdot {\kollp(x,y)}\geq 0$ and consider a sequence $(p_n)_n \subset [1,\infty)$ such that $p_n \to +\infty$ as $n\to \infty$, and denote by $\Pi_n(x,y,\cdot,\cdot)$ associated couplings of $(\delta_x P, \delta_y P)$ that show $W_{p_n}(\delta_x P, \delta_y P) \leq (1-K_n)$, i.e.
            \[
			\tond*{\sum_{w,z} \Pi_n(x,y,w,z)d(w,z)^{p_n}}^\frac{1}{p_n}\leq 1-K_n.
            \]
            By viewing $\Pi_n(x,y,\cdot,\cdot)$ as elements of $[0,1]^{\Omega\times \Omega}$ and by a compactness argument (recall that $\Omega$ is finite), we can pass to a subsequence (which we denote in the same way) and assume that there exists $\Pi \in [0,1]^{\Omega\times \Omega}$ such that $\Pi_n(x,y,\cdot,\cdot) \rightarrow \Pi(x,y,\cdot,\cdot)$ entrywise as $n\to \infty$. It is straightforward to check that $\Pi$ is still a coupling of $(\delta_x P, \delta_y P)$, and we claim that $\Pi$ has the desired properties. Indeed, by construction for $n\ge 1$ we have
			\begin{align*}
				& 2^{p_n} \sum_{w,z}\Pi_n(x,y,w,z)\indic{d(w,x)\geq 2} + \sum_{w,z}\Pi_n(x,y,w,z)\indic{d(w,x)=1 } 
				\\
				\leq & W_{p_n}^{p_n}(\delta_x P, \delta_y P)
				\\
				\leq & (1-K_n)^{p_n}
				\\
				\leq & e^{- K_n p_n}. 
			\end{align*}
			Letting $n\to \infty$ in the previous we deduce 
			\[
			\begin{cases}
				&\sum_{w,z}\Pi(x,y,w,z)\indic{d(w,x)\geq 2} = 0,
				\\
				& \sum_{w,z}\Pi(x,y,w,z)\indic{d(w,x) = 1} \leq e^{-K},
			\end{cases}
			\]
			which yields the desired conclusion.
			\item  Let $\Pi_{\mu,\nu}$ and $\Pi_{x,y}$ be the optimal couplings in the definitions of $W_p(\mu,\nu)$ and  $\kollp(x,y)$ respectively. Then
			\begin{align*}
				W_p^p(\mu P,\nu P) &\leq \sum_{x,y} d(x,y)^p \sum_{w,z} \Pi_{\mu,\nu}(w,z) \Pi_{w,z}(x,y) \\
				&\leq (1-K)^p \sum_{w,z} \Pi_{\mu,\nu}(w,z) d(w,z)^p
				\\
				&\leq (1-K)^p W_p^p(\mu,\nu).
			\end{align*}
			The conclusion follows by induction over $n$.
		\end{enumerate}
	\end{proof}
	
	\subsection{Continuous time}
	In this subsection, we describe the analogous notions of coarse Ricci curvature for continuous time Markov chains with generator $L$ on a finite state space (see \cite{veysseire_coarse_cts, lin-lu-yau-2011, Mun-Woj-2017}).
	Recalling that we identify $L$ with a matrix with zero row sums, we use the notation 
	\[
	\tilde{P}_t = I+tL,
	\]
	so that for $t>0$ small enough $\tilde{P}_t$ is a stochastic matrix too and we expect it to approximate the transition matrix $P_t$ (given it corresponds to the first order Taylor expansion of $P_t = e^{tL}$).
	The next definition is motivated by the study of the idleness function in \cite{Bou-Cus-Liu-Mun-Pey-2018}.
	\begin{definition}
		Let $T = \tond*{\max_{x \in \Omega} -L(x,x)}^{-1} = \tond*{\max_{x \in \Omega} \sum_{y\neq x} Q(x,y)}^{-1}$. Fix $x \neq y \in \Omega$ and let $p\in [1,\infty)$.
		Then we define the  function $I_{p,x,y}\colon \quadr*{0,T} \to \R$ by
		\[
		I_{p,x,y}(t) \coloneqq W_p^p(\delta_x \tilde{P}_t, \delta_y \tilde{P}_t).
		\]
	\end{definition}
	For $0\le t \le T$ we have that $\ptil_t$ is a stochastic matrix, so $\delta_x \ptil_t$ and $\delta_y \ptil_t$ are probability measures on $\Omega$ and $I_{p,x,y}(t)$ is well-defined.
	For simplicity, we drop the subscripts $x,y$ when there is no confusion. Notice that $I_p(0) = d(x,y)^p$.
	The next propositions and lemma are straightforward adaptations of the results in \cite{Bou-Cus-Liu-Mun-Pey-2018}, where the case $p=1$ was considered.
	Our motivation here is to establish linearity of $t\to I_p(t)$ for small values of $t$, hence the differentiability at $t=0$ (cf. Proposition \ref{pro: idleness function linear close to origin}): this allows us to consider a first definition of coarse Ricci curvature in continuous time, cf. Definition \ref{def:p-lly-curvature}. This definition was first given for $p=1$ in \cite{lin-lu-yau-2011} for combinatorial graphs (i.e. simple random walks on graphs), and later studied in a more general setting in \cite{Mun-Woj-2017}.
	\begin{proposition}
		The function $t\to I_p(t)$ is convex.
	\end{proposition}
	\begin{proof}
		We use Kantorovich duality: for a function $\psi\colon \Omega \to \R$, we denote by $\psi^{c,p}$ its $d^p$-transform, i.e.
		\[
		\psi^{c,p} (x) \coloneqq \inf_{y} \graf*{d(x,y)^p - \psi(y)}.
		\]
		Then we have for $t\in [0,T]$
		\begin{align*}
			& W_p^p(\delta_x \tilde{P}_t, \delta_y \tilde{P}_t) 
			\\
			= & \sup_\psi \graf*{ \sum_{z\in \Omega} \psi(z)  \cdot \tond*{\delta_x \tilde{P}_t}(z) +\psi^{c,p}(z)  \cdot \tond*{\delta_y \tilde{P}_t}(z) }
			\\
			= &\sup_\psi \Bigg\{\psi(x)\tond*{1-t\sum_z Q(x,z)} + \psi^{c,p}(y) \tond*{1-t\sum_z Q(y,z)}  
			\\
			+ &  t \quadr*{\sum_z Q(x,z)\psi(z) + Q(y,z) \psi^{c,p}(z)}\Bigg\}.
		\end{align*}
		This is the supremum of affine functions of $t$, hence $t\to 	W_p^p(\delta_x \tilde{P}_t, \delta_y \tilde{P}_t) = I_{p,x,y}(t)$ is convex.
	\end{proof}

	\begin{lemma}\label{lem:cond-linearity-laziness-function}
		Let $0\leq t_1<t_2\le T$ and suppose that $(\psi,\psi^{c,p})$ is a pair of optimal Kantorovich potentials in the definition of both $W_p(\delta_x\ptil_{t_1},\delta_y\ptil_{t_1})$ and $W_p(\delta_x\ptil_{t_2},\delta_y\ptil_{t_2})$. Then  $t \to I_p(t)$ is linear over $[t_1,t_2]$.
	\end{lemma}
	\begin{proof}
		We already know the function $t \to W_p^p\tond*{\delta_x\ptil_t,\delta_y\ptil_t}$ is convex, hence is suffices to show it is also concave over $[t_1,t_2]$. This follows by the assumption using Kantorovich duality, since for $\alpha \in [0,1]$
		\begin{align*}
			&W_p^p\tond*{\delta_x \ptil_{\alpha t_1 +(1-\alpha) t_2}, \delta_y \ptil_{\alpha t_1 +(1-\alpha) t_2}} 
			\\
			\geq & \sum_{z\in \Omega} \psi(z)  \cdot \tond*{\delta_x \tilde{P}_{\alpha t_1 +(1-\alpha) t_2}}(z) +\psi^{c,p}(z)  \cdot \tond*{\delta_y \tilde{P}_{\alpha t_1 +(1-\alpha) t_2}}(z) 
			\\
			= & \psi(x)\tond*{1-(\alpha t_1 + (1-\alpha) t_2)\sum_z Q(x,z)} + \psi^{c,p}(y) \tond*{1-(\alpha t_1 + (1-\alpha) t_2)\sum_z Q(y,z)} 
			\\
			+& (\alpha t_1 + (1-\alpha) t_2) \quadr*{\sum_z Q(x,z)\psi(z) + Q(y,z) \psi^{c,p}(z)}
			\\
			= & \alpha W_p^p\tond*{\delta_x \ptil_{t_1}, \delta_y \ptil_{t_1}} + (1-\alpha) W_p^p\tond*{\delta_x \ptil_{t_2}, \delta_y \ptil_{t_2}}.
		\end{align*}
	\end{proof}

	\begin{proposition}\label{pro: idleness function linear close to origin}
		The  function $I_p$ is linear over $[0,\frac{T}{2}]$.
	\end{proposition}
	\begin{proof}
		Fix $t<\frac{T}{2}$ and let $\psi$ and $\gamma$ be a couple of optimal Kantorovich potential and optimal transport plan for $W_p\tond*{\delta_x\ptil_t,\delta_y\ptil_t}$. Then, it is easy to see that $\gamma(x,y)>0$ (since the laziness  of $\ptil_t$ is at least $1- t/T>\frac{1}{2}$, and so $\delta_x \ptil_t(x), \delta_y\ptil_t(y) > \frac{1}{2}$). 
		It follows that $\psi(x)+\psi^c(y)=d(x,y)^p$ (cf. \cite[Thm. 5.10]{vil-2009} for example). This implies that $(\psi,\psi^c)$ is optimal also for $W_p(\delta_x\ptil_0,\delta_y\ptil_0) = W_p(\delta_x,\delta_y) = d(x,y)$. The conclusion follows by Lemma \ref{lem:cond-linearity-laziness-function} and by letting $t \to \frac{T}{2}$ (by continuity at $\frac{T}{2}$, which follows from the convexity of $I_p$).
	\end{proof}
	
	Since linear functions are easily differentiated, the previous proposition immediately implies the following corollary, which gives an expression for the derivative of $I_p(t)$ at $t=0$.
	\begin{corollary}
		For $x\neq y$  there exists
		\begin{align}
			\frac{d}{dt}\bigg|_{t=0} W_p^p \tond*{\delta_x \ptil_t, \delta_y \ptil_t} &= - \frac{1}{s} \tond*{d(x,y)^p-W_p^p \tond*{\delta_x \ptil_{s}, \delta_y \ptil_{s}}}
		\end{align}
		for all $0<s\leq \frac{T}{2}$.
	\end{corollary}

	\begin{definition}\label{def:p-lly-curvature}
		The continuous time $p$-coarse Ricci curvature in direction $(x,y)$ is defined via
		\begin{align}
			\kllyp (x,y) &\coloneqq  -\frac{1}{d(x,y)}\frac{d}{dt}\bigg|_{t=0} {W_p\tond*{\delta_x \ptil_t, \delta_y \ptil_t}} \nonumber
			\\ \nonumber
			& =  -\frac{1}{ p \cdot d(x,y)^p} \frac{d}{dt}\bigg|_{t=0} {W_p^p\tond*{\delta_x \ptil_t, \delta_y \ptil_t}}
			\\\label{eq:def-p-lly} 
			& = \frac{ 1}{sp}  \tond*{1-\frac{W_p^p \tond*{\delta_x \ptil_{s}, \delta_y \ptil_{s}}}{d(x,y)^p}}
		\end{align}
		for all $0<s\leq \frac{T}{2}$.
	\end{definition}

	\begin{remark}
		The original definition in \cite{lin-lu-yau-2011} (see also \cite{Mun-Woj-2017}) focused on the case $p=1$, but in this paper it will be useful to consider also other values of $p$.
	\end{remark}

	Next, we prove a few preliminary results needed for the proof of Corollary \ref{cor:equiv-vey}, which will show that we can define $\kllyp$ also by involving $P_t$ instead of $\ptil_t$. For $p=1$, this shows the equivalence with the notion of continuous time coarse Ricci curvature defined in \cite{veysseire_coarse_cts}, which was first proved in \cite[Thm. 5.8]{Mun-Woj-2017}.
	\begin{lemma}\label{lem:clos-coup-tv}
		Let $X,Y$ be finite sets and $\mu, \tilde{\mu} \in \mathcal{P}(X), \nu \in \mathcal{P}(Y)$ be probability measures.
		Suppose that $\gamma \in \Gamma(\mu,\nu)$ is a coupling between $\mu$ and $\nu$.
		Then, there exists a coupling $\tilde{\gamma}\in \Gamma(\tilde{\mu},\nu)$ between $\tilde \mu$ and $\nu$ that satisfies
		\[
		\tv{\tilde{\gamma}-\gamma} = \tv{\tilde{\mu}-\mu}.
		\]
	\end{lemma}
	\begin{proof}
		We will construct a coupling $\tilde{\gamma} \in \Gamma(\tilde{\mu},\nu)$ such that
		\begin{equation}\label{eq:tilde-gamma-bigger}
			\tilde{\gamma}(x,y) \geq \gamma(x,y) \text{ if and only if } \tilde{\mu}(x) \geq \mu(x).
		\end{equation}
		Notice that this implies the thesis since we would have
		\begin{align*}
			\tv{\tilde{\gamma}-\gamma} &= \frac{1}{2} \sum_{x\in X, y\in Y}\abs*{\tilde{\gamma}(x,y)-\gamma(x,y)}
			\\ 
			& = \frac{1}{2}\sum_{x,y: \tilde{\mu}(x)\geq \mu(x)}\quadr*{\tilde{\gamma}(x,y)-\gamma(x,y)} +  \frac{1}{2}\sum_{x,y: \tilde{\mu}(x)< \mu(x)}\quadr*{\gamma(x,y)-\tilde{\gamma}(x,y)}
			\\
			& = \frac{1}{2}\sum_{x: \tilde{\mu}(x)\geq \mu(x)}\quadr*{\tilde{\mu}(x)-\mu(x)} +  \frac{1}{2}\sum_{x: \tilde{\mu}(x)< \mu(x)}\quadr*{\mu(x)-\tilde{\mu}(x)}
			\\
			& = \frac{1}{2}\sum_{x\in X}\abs*{\tilde{\mu}(x)-\mu(x)}
			\\
			& = \tv{\tilde{\mu}-\mu}.
		\end{align*}
		Let us thus construct $\tilde{\gamma}$ that satisfies \eqref{eq:tilde-gamma-bigger}.
		To do that, for every $x\in X, y\in Y$ set 
		\begin{align*}
			\alpha(x) &= \tilde{\mu}(x)-\tond*{\tilde{\mu}\wedge\mu}(x),
			\\
			\beta(y) & = \nu(y) - \sum_{x\in X: \tilde{\mu}(x)<\mu(x)}\gamma(x,y) \frac{\tilde{\mu}(x)}{\mu(x)} - \sum_{x\in X: \tilde{\mu}(x)\ge\mu(x)}\gamma(x,y).
		\end{align*}
		It is easy to see that $\sum_{x\in X}\alpha(x)  = \sum_{y\in Y}\beta(y) = \tv{\mu-\tilde{\mu}}$. 
		Then we define
		\[
		\tilde{\gamma}(x,y)\coloneqq \indic{\tilde{\mu}(x) < \mu (x)} \frac{\tilde{\mu}(x)}{\mu(x)} \gamma(x,y) + \indic{\tilde{\mu}(x) \geq \mu(x)} \gamma(x,y) + \frac{1}{\tv{\mu-\tilde{\mu}} } \alpha(x) \beta(y).
		\]
		It is easy to check that $\tilde{\gamma} \in \Gamma(\tilde{\mu}, \nu)$ and that $\tilde{\gamma}$ satisfies \eqref{eq:tilde-gamma-bigger}.
	\end{proof}

	\begin{corollary}\label{cor:cont-ot-cost-wrt-tv}
		Let $X,Y$ be finite sets, $\mu_1,\mu_2 \in \pp(X)$, $\nu_1,\nu_2\in \pp(Y)$ and $c\colon X\times Y \to \R_{\ge 0}$ be a cost function. For any probability measures $\mu\in\pp(X),\nu\in \pp(Y)$, let $\ttt{\mu,\nu} = \mathcal{T}_c\tond{\mu,\nu}$ be the optimal transport cost associated to the cost $c$. Then 
		\[
		\abs*{\ttt{\mu_1,\nu_1}- \ttt{\mu_2,\nu_2}} \leq 2\max_{x\in X, y\in Y}c(x,y) \cdot \tond*{\tv{\mu_1-\mu_2} + \tv{\nu_1-\nu_2}}.
		\]
	\end{corollary}
	
	\begin{proof}
		Let $M\coloneqq  \max_{x\in X, y\in Y} c(x,y) \geq 0$.
		Notice that
		\[
		\abs*{\ttt{\mu_1,\nu_1}- \ttt{\mu_2,\nu_2}} \leq 
		\abs*{\ttt{\mu_1,\nu_1}- \ttt{\mu_2,\nu_1}} + \abs*{\ttt{\mu_2,\nu_1} - \ttt{\mu_2,\nu_2}}, 
		\]
		hence it suffices to show that 
		\begin{align*}
			\abs*{\ttt{\mu_1,\nu_1}- \ttt{\mu_2,\nu_1}}  &\le 2M\cdot  \tv{\mu_1-\mu_2},
			\\
			\abs*{\ttt{\mu_2,\nu_1} - \ttt{\mu_2,\nu_2}} & \le 2M \cdot  \tv{\nu_1-\nu_2}.
		\end{align*}
		We prove only the first inequality, since the proof of the second one is similar.
		Let $\gamma$ be an optimal coupling in the definition of $\ttt{\mu_1,\nu_1}$ and $\tilde{\gamma}$ be the coupling  for $\tond*{\mu_2,\nu_1}$ given by Lemma \ref{lem:clos-coup-tv}. Then it follows that
		\begin{align*}
			\ttt{\mu_2,\nu_1} &\leq \sum_{x,y} c(x,y) \tilde{\gamma}(x,y) 
			\\
			& \leq  \sum_{x,y} c(x,y) \tond*{ {\gamma}(x,y)+\abs*{\tilde{\gamma}- \gamma}(x,y) }
			\\
			&\le \ttt{\mu_1,\nu_1} + 2\,M \tv{\tilde{\gamma}- \gamma}
			\\
			& \le \ttt{\mu_1,\nu_1} + 2\,M \tv{\mu_1-\mu_2}.
		\end{align*}
		Similarly one shows that
		\[
		\ttt{\mu_1,\nu_1} \leq		\ttt{\mu_2,\nu_1} + 2\,M \tv{\mu_1-\mu_2},
		\]
		which concludes the proof of the corollary.
	\end{proof}
	
	\begin{corollary}\label{cor:wass-p-bound}
		For any probability measures $\mu, \nu \in \pp(\Omega)$ and $t>0$ small enough we have
		\[
		\abs*{W_p^p\tond*{\mu P_t,\nu P_t} - W_p^p\tond*{\mu \ptil_t, \nu \ptil_t}} = O(t^2)
		\]
		as $t\to 0$.
	\end{corollary}
	
	\begin{proof}
		Fix $\mu,\nu \in \pp(X)$ and $t>0$ small and let $D \coloneqq \diam(\Omega)$. Since $\Omega$ is finite, we have that $D<\infty$ and $\tv{\mu \ptil_t-\mu P_t}, \tv{\nu \ptil_t-\nu P_t}  = O(t^2)$.
		By applying Corollary \ref{cor:cont-ot-cost-wrt-tv} for $t>0$ small enough with 
		$\mu_1 = \mu P_t, \mu_2 = \mu \ptil_t, \nu_1 = \nu P_t, \nu_2 = \nu \ptil_t$, we find
		\[
		\abs*{W_p^p\tond*{\mu P_t,\nu P_t} - W_p^p\tond*{\mu \ptil_t, \nu \ptil_t}} \leq 2 D^p \cdot \tond*{\tv{\mu \ptil_t- \mu P_t} + \tv{\nu \ptil_t- \nu P_t} } = O(t^2),
		\]
		as desired.

	\end{proof}

	We can finally give an equivalent definition of $\kllyp$ using $P_t$ instead of $\ptil_t$.
	\begin{corollary}\label{cor:equiv-vey}
		For $x\neq y \in \Omega$ we have
		\begin{equation} \label{eq:vey-p-curv}
			\kllyp(x,y) = -\frac{1}{d(x,y)}  \frac{d}{dt}\bigg|_{t=0}  {W_p\tond*{\delta_xP_t, \delta_y P_t}}.
		\end{equation}
	\end{corollary}
	\begin{proof}
		Recalling Definition \ref{def:p-lly-curvature}, it suffices to show that
		\[
		\frac{d}{dt}\bigg|_{t=0} {W_p\tond*{\delta_x \ptil_t, \delta_y \ptil_t}} = \frac{d}{dt}\bigg|_{t=0} {W_p\tond*{\delta_x P_t, \delta_y P_t}}.
		\] 
		Equivalently, since $W_p\tond*{\delta_x P_0, \delta_y P_0} = W_p\tond*{\delta_x \ptil_0, \delta_y \ptil_0} = d(x,y) >0$, it is enough to show
		\[
		\frac{d}{dt}\bigg|_{t=0} {W_p^p\tond*{\delta_x \ptil_t, \delta_y \ptil_t}} = \frac{d}{dt}\bigg|_{t=0} {W_p^p\tond*{\delta_x P_t, \delta_y P_t}},
		\]
		but this follows from Corollary \ref{cor:wass-p-bound}.
	\end{proof}

	\begin{remark}
		As anticipated, for $p=1$ the right hand side in \eqref{eq:vey-p-curv} corresponds to the definition of curvature for continuous time Markov chain introduced by Veysseire (see \cite{veysseire_coarse_cts}). For $p=1$, the identity in the corollary was established in \cite{Mun-Woj-2017} (in a more general setting).
	\end{remark}
	As in discrete time, we define also the curvature for $p =\infty$.

	\begin{definition}\label{def:inf-lly-curvature}
		For $x\neq y$, we define the continuous time  $\infty$-coarse Ricci curvature $\kllyinf(x,y)$ in direction $(x,y)$ to be the supremum of all $K\in \R$ such that there exist coupling rates $C(x,y,\cdot,\cdot)$ satisfying:
		\begin{itemize}
			\item $\sum_{v,z\in\Omega} C(x,y,v,z) \indic{d(v,z)> d(x,y)} = 0$;
			\item $\sum_{v,z\in\Omega} C(x,y,v,z)d(v,z) \leq \tond*{\sum_{v,z} C(x,y,v,z)  -K} d(x,y)$.
		\end{itemize}
	\end{definition} 
	
	\begin{remark}
		We have that $\kllyinf(x,y) \in \R_{\geq0} \cup\graf{-\infty}$; if $\kllyinf(x,y)\geq 0$, the supremum in the definition  is attained.
		Notice that if $C(x,y,\cdot,\cdot)$ gives optimal coupling rates, then we can change the value $C(x,y,x,y)$ arbitrarily obtaining other optimal coupling rates.
	\end{remark}
	
	As in discrete time, for $p\in [1,\infty]$ we write 
	\[
	\ricllyp \geq K
	\]
	if $\kllyp(x,y)\geq K$ for all $x\neq y \in \Omega$.

	The next result shows that, given coupling rates from $x\neq y$, we can construct a coupling for the probability measures $\tond*{\delta_x \ptil_t, \delta_t \ptil_t}$ for any $t>0$ small enough. This will be useful to connect the notions of coarse Ricci curvature for discrete and continuous time Markov chains, when we consider a natural correspondence between stochastic transition matrices $P$ and generators $L$, see Section \ref{sec:disc-vs-cts}.
	\begin{lemma}\label{lem:coup-from-coup-rates}
		Let $x\neq y\in \Omega$ and $C(x,y,\cdot,\cdot)$ denote coupling rates from $x,y$, and set 
		\[
		M \coloneqq \sum_{v,w} C(x,y,v,w) = \sum_{v}Q(x,v) = \sum_{w} Q(y,w).
		\] 
		Then, for any $0< t\le \frac{1}{M-C(x,y,x,y)}$ we have that $\gamma_t$ is a coupling for $\tond*{\delta_x \ptil_t, \delta_t \ptil_t}$, where we define
		\[
		\gamma_t(v,w) = 
		\begin{cases}
			C(x,y,v,w) \cdot t  &\text{ if } (v,w) \neq (x,y),
			\\
			1 - M\cdot t + C(x,y,x,y)\cdot t &\text{ if } (v,w) = (x,y).
		\end{cases}		
		\]
	\end{lemma}
	
	\begin{proof}
		Clearly for $0<t\le \frac{1}{M-C(x,y,x,y)}$ and $v,w\in\Omega$  we have that $0\leq \gamma_t(v,w) \leq 1$.
		Now, if $v\neq x$ we have that
		\[
		\sum_{w}\gamma_t(v,w) = t \sum_{w}C(x,y,v,w) = t Q(x,v) = (\delta_x \ptil_t)(v).
		\]
		If $v = x$ instead then
		\begin{align*}
			\sum_{w}\gamma_t(x,w) &= t \sum_{w\neq y}C(x,y,x,w) + 1 - M\cdot t + C(x,y,x,y)\cdot t 
			\\
			& = 1 +  Q(x,x)\cdot t - M\cdot t 
			\\
			& = (\delta_x \ptil_t)(x).
		\end{align*}
		Hence, this shows that the first marginal of $\gamma_t$ is $\delta_x \ptil_t$. 
		Similarly, one checks that the second marginal of $\gamma_t$ is $\delta_y \ptil_t$, as desired.
	\end{proof}
	
	The next proposition collects some useful results, and it is a continuous time analogue of Proposition \ref{prop:properties-kollp}.
	\begin{proposition}\label{prop:properties-llyp}
		The following hold:
		\begin{enumerate}
			\item\label{it:coarseContinuousNeighbours} For $p\in [1,\infty]$, if $\kllyp(x,y)\geq K$ for all $x\sim y$ then $\ricllyp\geq K$.
			\item \label{it:coarseContinuousDifferentP} For $1\leq p \leq q <\infty$  we have $\kllyp(x,y) \geq K_{cc,q}(x,y)$. Moreover if $x\sim y$  we have that 
			\[
			\kllyp(x,y) \geq \frac{\kllyinf(x,y)}{p}.
			\]
			\item \label{it:CoarseContinuousContraction} For $p\in[1,\infty)$, if $\ricllyp\geq K$ then for any starting probability measures $\mu,\nu \in \pp(\Omega)$ and any $t\geq 0$ we have that
			\[
			W_p(\mu P_t, \nu Pt) \leq e^{-Kt} W_p(\mu,\nu).
			\]
		\end{enumerate}
		
	\end{proposition}
	\begin{proof}
		\begin{enumerate}

			\item Suppose first that $p<\infty$ and let $n = d(x,y)$ with $x=z_0\sim z_1 \sim \ldots \sim z_n =y$. Then for $t>0$ small enough we have that
			\[
			- W_p(\delta_x \ptil_t,\delta_y \ptil_t) \geq \sum_{i=0}^{n-1}- W_p\tond*{\delta_{z_{i}}\ptil_t, \delta_{z_{i+1}}\ptil_t}
			\]
			by the triangle inequality. Therefore adding $n= d(x,y)$ and dividing by $t$
			\[
			\frac{d(x,y)-W_p(\delta_x \ptil_t,\delta_y \ptil_t)}{t} \geq \sum_{i=0}^{n-1} \frac{1-W_p\tond*{\delta_{z_{i}}\ptil_t, \delta_{z_{i+1}}\ptil_t}}{t}.
			\]
			Letting $t\to 0$ and using the assumption gives the conclusion.

			Suppose now that $p = \infty$: if $K = -\infty$ the conclusion is trivial, hence assume that $ K\geq 0$. Let again  $n=d(x,y)$ and $x=z_0\sim z_1 \sim \ldots \sim z_n =y$: we prove the claim by induction over $n$. The base case $n=1$ follows directly by the assumption. Now suppose $n>1$ and that the inductive hypothesis holds. Let $C(x,z_{n-1}, \cdot, \cdot)$ and $C(z_{n-1},y,\cdot,\cdot)$ be such that
			
			\begin{align*}
				&\sum_{v,w\in \Omega} C(x,z_{n-1}, v, w) \indic{d(v,w)> d(x,z_{n-1})} = 0,
				\\
				& \sum_{v,w\in\Omega} C(x,z_{n-1}, v, w) d(v,w) \leq \tond*{\sum_{v,w\in\Omega} C(x,z_{n-1}, v, w) - K} d(x,z_{n-1}), 
				\\
				& \sum_{v,w\in\Omega} C(z_{n-1},y, v, w) \indic{d(v,w) > 1} = 0,
				\\
				& \sum_{v,w\in\Omega} C(z_{n-1},y, v, w) d(v,w) \leq \tond*{ \sum_{v,w\in\Omega} C(z_{n-1},y, v, w) -  K }.
			\end{align*}
			Without loss of generality, by changing if needed the values of $C(x,z_{n-1},x,z_{n-1})$, $C(z_{n-1}, y, z_{n-1}, y)$, $Q(x,x)$, $Q(z_{n-1}, z_{n-1})$, $Q(y,y)$, we can assume that 
			\begin{align*}
				& \sum_{s\in\Omega} Q(x,s)  =\sum_{s\in\Omega} Q(z_{n-1}, s) = \sum_{s\in \Omega} Q(y,s) 
				\\
				= & \sum_{v,w\in\Omega} C(x,z_{n-1},v,w) =\sum_{v,w\in\Omega} C(z_{n-1},y , v, w)  
				\\
				\eqqcolon & M >0.
			\end{align*}
			Therefore, we can apply the Gluing lemma (which easily extends to measures having the same total mass) to conclude that there exists $\hat{C}(\cdot,\cdot,\cdot) = \hat{C}(x,z_{n-1},y,\cdot,\cdot,\cdot) \colon \Omega \times \Omega \times \Omega \to \R_{\geq 0}$ such that $p_{1,2}\#\hat{C} = C(x,z_{n-1},\cdot,\cdot) $ and $p_{2,3}\#\hat{C} = C(z_{n-1},y,\cdot,\cdot)$, where $p_{i,j}$ is the projection on coordinates $i,j$ and $\#$ is the pushforward, so that
			\begin{align*}
				&\sum_{w\in\Omega} \hat{C}(x,z_{n-1},y,v,s,w) = C(x,z_{n-1},v,s),
				\\
				& \sum_{v\in\Omega} \hat{C}(x,z_{n-1},y,v,s,w) = C(z_{n-1},y,s,w).
			\end{align*}
			Defining $C(x,y,\cdot,\cdot) \coloneqq p_{1,3} \# \hat{C}$ gives coupling rates with the desired properties, analogously to the proof of Proposition \ref{prop:properties-kollp}.
			Indeed, we have that 
			\begin{align*}
				\sum_{v,w}C(x,y,v,w) d(v,w) =& \sum_{v,w,s}\hat{C}(x,z_{n-1},y,v,s,w) d(v,w)
				\\
				\leq &\sum_{v,w,s}\hat{C}(x,z_{n-1},y,v,s,w)  \tond*{d(v,s) + d(s,w)}
				\\
				=& \sum_{v,s}C(x,z_{n-1},v,s) {d(v,s)} + \sum_{s,w}C(z_{n-1},y,s,w) {d(s,w)}
				\\
				\leq & (M-K) (d(x,z_{n-1}) + d(z_{n-1},y))
				\\
				= & \tond*{\sum_{v,w} C(x,y,v,w)-K} d(x,y),
			\end{align*}
			and similarly 
			\begin{align*}
				&\sum_{v,w}C(x,y,v,w) \indic{d(v,w)>d(x,y)} 
				\\
				= & \sum_{v,w,s}\hat{C}(x,z_{n-1},y,v,s,w) \indic{d(v,w)>d(x,y)}
				\\
				\leq &\sum_{v,w,s}\hat{C}(x,z_{n-1},y,v,s,w)  \indic{d(v,s)+d(s,w) > d(x,z_{n-1})+d(z_{n-1}, y)}
				\\
				\leq &\sum_{v,w,s}\hat{C}(x,z_{n-1},y,v,s,w) \tond*{ \indic{d(v,s) > d(x,z_{n-1})}+
					\indic{d(s,w) > d(z_{n-1}, y)} }
				\\
				=& \sum_{v,s}C(x,z_{n-1},v,s) \indic{d(v,s) > d(x,z_{n-1})} + \sum_{s,w}C(z_{n-1},y,s,w)\indic{d(s,w) > d(z_{n-1}, y)} 
				\\
				= & \,0.
			\end{align*}

			\item The first statement follows by the inequality $W_p(\mu, \nu) \leq W_q(\mu,\nu)$: indeed it implies
			\[
			\frac{d(x,y)-W_p(\delta_x \ptil_t, \delta_y \ptil_t)}{t\,d(x,y)} \geq \frac{d(x,y)-W_q(\delta_x \ptil_t, \delta_y \ptil_t)}{t\,d(x,y)},
			\] 
			from which the conclusion follows by letting $t\to 0$.
			For the second statement, suppose $\kllyinf(x,y)\geq 0$ for $x\sim y$. Let $C(x,y,\cdot,\cdot)$ be optimal coupling rates in the definition of $\kllyinf(x,y)$. 
			Then for $t>0$ small enough consider the coupling $\gamma_t$ given by Lemma \ref{lem:coup-from-coup-rates}. It is easy to see that this coupling is such that
			\begin{align*}
				\sum_{v,w} \gamma_t\tond{v,w} \indic{d(v,w)>1} &= 0;
				\\
				\sum_{v,w} \gamma_t\tond{v,w} \indic{d(v,w)= 0} &\geq \kllyinf(x,y)\cdot t.
			\end{align*}
			Therefore, it shows that
			\[
			W_p\tond*{\delta_x \ptil_t, \delta_y \ptil_t} \leq \tond*{1-\kllyinf(x,y)\cdot t}^{\frac{1}{p}} \leq 1-\frac{\kllyinf(x,y)\cdot t}{p}.
			\]
			Hence we have
			\[
			\frac{1 - W_p(\delta_x \ptil_t, \delta_y \ptil_t)}{t} \geq  \frac{\kllyinf(x,y)}{p},
			\]
			from which the conclusion follows by letting $t\to 0$.

			\item Let $\gamma_{\mu,\nu}$ and $\gamma_{x,y,t}$ be  optimal couplings in the definitions of $W_p(\mu,\nu)$ and  $W_p(\delta_x P_t, \delta_y P_t)$ respectively. Then notice that
			\begin{align*}
				W_p^p(\mu P_t,\nu P_t) &\leq \sum_{x,y} d(x,y)^p \sum_{w,z} \gamma_{\mu,\nu}(w,z) \gamma_{w,z,t}(x,y) \\
				& = \sum_{w,z}\gamma_{\mu,\nu}(w,z) W_p^p(\delta_w P_t, \delta_z P_t),
			\end{align*}
			and so
			\[
			\frac{W_p^p(\mu P_t,\nu P_t)-W_p^p(\mu,\nu)}{t}\leq \sum_{w,z}\gamma_{\mu,\nu}(w,z) \frac{ W_p^p(\delta_w P_t, \delta_z P_t) - d(w,z)^p}{t}.
			\]
			
			Taking the $\limsup_{t\to 0^+}$ and denoting by $\frac{d^+}{dt}$ the upper Dini derivative
			(cf. Appendix \ref{sec:gronwall-dini}) we find that
			\begin{align*}
				& \frac{d^+}{dt}\bigg|_{t=0} W_p^p(\mu P_t,\nu P_t)
				\\
				\leq & \sum_{w,z} \gamma_{\mu,\nu}(w,z) \frac{d}{dt}\bigg|_{t=0}W_p^p(\delta_w P_t, \delta_z P_t)
				\\
				= & \sum_{w,z} \gamma_{\mu,\nu}(w,z) \,p\, d(w,z)^{p-1}\frac{d}{dt}\bigg|_{t=0}W_p(\delta_w P_t, \delta_z P_t)
				\\
				\leq & -K\,p \,\sum_{w,z} \gamma_{\mu,\nu}(w,z) d(w,z)^{p}
				\\
				= &-K\,p\,W_p^p(\mu,\nu),
			\end{align*}
			where we also used Corollary \ref{cor:equiv-vey}.
			Therefore, 
			\[
			\frac{d^+}{dt}\bigg|_{t=0} W_p^p(\mu P_t,\nu P_t) \leq -Kp W_p^p(\mu,\nu).
			\]
			and by Markovianity this extends to every $\bar{t}>0$, i.e.
			\[
			\frac{d^+}{dt}\bigg|_{t = \bar{t}} W_p^p(\mu P_t,\nu P_t) \leq -Kp W_p^p(\mu P_{\bar{t}},\nu P_{\bar{t}}).
			\]
			Noticing also that $t\to W_p^p\tond{\mu P_t, \nu P_t}$ is continuous by Corollary \ref{cor:cont-ot-cost-wrt-tv}, we can apply Lemma \ref{lem:gronwall-dini} to conclude that
			\[
			W_p^p(\mu P_t,\nu P_t) \leq e^{-Kpt} W_p^p(\mu ,\nu ),
			\]
			for any $t\ge 0$, as desired.
		\end{enumerate}
	\end{proof}

	\subsection{Comparison discrete and continuous time}
	\label{sec:disc-vs-cts}
	There is a natural way to construct a continuous time Markov chain from a discrete time one: namely, for $\lambda >0$ and a stochastic matrix $P$, the generator is defined by $L=\lambda(P-I)$. On the other hand, it is readily seen that, given a generator $L$, for any $\lambda>0$ big enough there exists a corresponding stochastic matrix $P = I+\frac{1}{\lambda}L$ (recall we are assuming finiteness of the state space, so the entries of $L$ are bounded).
	\begin{remark}\label{rmk:equiv-oll-vey}
		Let $P$ be a stochastic matrix and $\lambda >0$.
		In view of Definition \ref{def:p-lly-curvature}, we see that for $p=1$ the $\ricoll_{,1}$ and $\riclly_{,1}$ notions of curvature are essentially equivalent for the continuous time Markov chain with generator
		$L = \lambda (P-I)$ (and transition semigroup $P_t$) and the \textbf{lazy} discrete time Markov chain with transition matrix $\ptil \coloneqq \frac{I+P}{2}$. Indeed, it follows from the  identity in \eqref{eq:def-p-lly} with $s=\frac{1}{2\lambda}$ that 
		\begin{equation}\label{eq:equiv-oll-vey}
			\kollone^{(\ptil)}(x,y) = \frac{1}{2\lambda} \kllyone(x,y),
		\end{equation}
		where above $\kollone^{(\ptil)}(x,y)$ is the curvature $\kollone(x,y)$ for the Markov chain with transition matrix $\ptil$, not $P$.
		Notice that here it is fundamental to consider the lazy version with transition matrix $\ptil$. To see why, consider the following simple example: the state space is the two-point space $\Omega = \graf*{x,y}$, the stochastic matrix is
		\[
		P = \begin{pmatrix}
			0 & 1
			\\
			1 & 0
		\end{pmatrix}
		\]
		and $\lambda= 1$. Then we have $\kollone^{(P)}(x,y)=0$ for $P$, $\kollone^{(\ptil)}(x,y)= 1$ for $\ptil$ and $\kllyone(x,y) = 2$ for $L = P-I$. Hence \eqref{eq:equiv-oll-vey} is satisfied only when $\kollone(x,y)$ is defined for $\ptil$ and not for $P$.
	\end{remark}
	The next proposition 
	shows that an analogous relation as the one described in the above remark holds true also for $\ricollinf$ and $\ricllyinf$.

	\begin{proposition}\label{prop:equiv-ollinf-llyinf}
		Suppose $L=\lambda(P-I)$ for a stochastic matrix $P$ and $\lambda >0$. Then the following hold for any $x\neq y$:
		\begin{enumerate}
			\item $\kllyinf(x,y) \geq \lambda \kollinf^{(P)}(x,y)$.
			\item For the lazy Markov chain with transition matrix $\tilde{P} = \frac{P+I}{2}$, we have $\kollinf^{(\ptil)}(x,y) \geq \frac{1}{2\lambda} \kllyinf(x,y)$.  
		\end{enumerate}
	\end{proposition}

	\begin{proof}
		\begin{enumerate}
			\item Assume $\kollinf^{(P)}(x,y)\geq 0$, otherwise the claim is trivial. Let $\Pi(x,y,\cdot,\cdot)$ be an optimal coupling in the definition of $\kollinf^{(P)}(x,y)$ and define then the coupling rates $C(x,y,w,z) = \lambda \cdot \Pi(x,y,w,z)$. It is easy to check that these coupling rates are admissible and yield the first conclusion.

			\item Again, assume $\kllyinf(x,y) \geq 0$, otherwise the conclusion is trivial. Let $C(x,y,\cdot,\cdot)$ be optimal coupling rates in the definition of $\kllyinf(x,y)$. 
			Notice that for $L = \lambda(P - I)$ and $t = \frac{1}{2\lambda}$ we have
			\[
			\ptil_t = \ptil_{\frac{1}{2\lambda}} = I +\frac{1}{2\lambda}\cdot\lambda(P-I) = \ptil. 
			\]
			Therefore, since
			\begin{align*}
				M\coloneqq \sum_{v,w} C(x,y,v,w) &= C(x,y,x,y) +	\sum_{v,w: (v,w)\neq (x,y)} C(x,y,v,w) 
				\\
				&\leq C(x,y,x,y)+\sum_{v,w : v\neq x} C(x,y,v,w) + \sum_{v,w : w\neq y} C(x,y,v,w)
				\\
				& = C(x,y,x,y)+ \sum_{v\neq x} Q(x,v) + \sum_{w\neq y} Q(y,w)
				\\
				&\leq C(x,y,x,y)+  2\lambda,
			\end{align*}
			we can apply Lemma \ref{lem:coup-from-coup-rates} to obtain a coupling $\Pi(x,y,\cdot,\cdot)$ for $\tond*{\delta_x\ptil,\delta_y \ptil}$, i.e.
			\[
			\Pi(x,y,v,w) =
			\begin{cases}
				C(x,y,v,w) \cdot \frac{1}{2\lambda}  &\text{ if } (v,w) \neq (x,y),
				\\
				1 - M\cdot \frac{1}{2\lambda} + C(x,y,x,y)\cdot \frac{1}{2\lambda} &\text{ if } (v,w) = (x,y).
			\end{cases}		
			\]
			Clearly, by the assumptions on $C(x,y,\cdot,\cdot)$ we have that
			\[
			\sum_{v,w\in\Omega} \Pi(x,y,v,w)\indic{d(v,w)> d(x,y)} = 0.
			\]
			Moreover 
			\begin{align*}
				\sum_{v,w\in\Omega} \Pi(x,y,v,w){ d(v,w)}  = & \tond*{1-\frac{M}{2\lambda}}d(x,y) + \frac{1}{2\lambda}\sum_{v,w} C(x,y,v,w)d(v,w)
				\\
				\leq & \tond*{1-\frac{M}{2\lambda}}d(x,y) + \frac{1}{2\lambda}(M-\kllyinf(x,y))d(x,y)
				\\
				= & \tond*{1-\frac{\kllyinf(x,y)}{2\lambda}}d(x,y).
			\end{align*}
			
			This shows  that $\kollinf^{(\ptil)}(x,y)\geq \frac{1}{2 \lambda} \kllyinf(x,y)$ for the lazy Markov chain $\tilde{P}$, as desired.
		\end{enumerate}

	\end{proof}
	
	\subsection{Applications and related problems}\label{sec:conjecture}
	
	Recalling Definition \ref{def:inf-lly-curvature} and Proposition \ref{prop:properties-llyp}, we see that for a continuous time Markov chain on a finite state  space $\Omega$ we have that $\ricllyinf \geq K>0$ if and only if for all \emph{neighbouring} states $x\sim y\in \Omega$ there exist coupling rates $C(x,y,\cdot,\cdot)$ satisfying  
	\begin{itemize}
		\item $\sum_{v,z\in\Omega} C(x,y,v,z) \indic{d(v,z)> 1} = 0$;
		\item $\sum_{v,z\in\Omega} C(x,y,v,z) \indic{d(v,z) = 0} \geq K$.
	\end{itemize}
	Bearing this mind, we can see that in the examples of Section \ref{sec:applications} the constructed coupling rates immediately yield positive $\ricllyinf$ curvature: more precisely, we have the following result.
	\begin{theorem}\label{thm:pos-lly-inf-curv}
		The following hold:
		\begin{itemize}
			\item  Under the assumptions of Theorem \ref{thm: entropic curvature Glauber dynamics}, we have $\ricllyinf \geq \kappa_*$ for  Glauber dynamics.
			\item Under the assumptions of Theorem \ref{thm:bern-lapl}, we have $\ricllyinf \geq L$ for the Bernoulli--Laplace model.
			\item Under the assumptions of Theorem \ref{thm:harcore-model}, we have $\ricllyinf \geq \kappa_*$ for the hardcore model.
		\end{itemize}
		In particular, by Proposition \ref{prop:properties-llyp},  under the respective  theorems' assumptions for all probability measures $\mu,\nu\in 	\pp(\Omega)$, $p\geq 1$ and $t\geq 0$ we have
		\begin{equation}\label{eq:p-wass-contraction-conj}
			W_p\tond*{\mu P_t, \nu P_t} \leq e^{-\frac{K}{p}t} W_p(\mu,\nu),
		\end{equation}
		with $K = \begin{cases}
			\kappa_* &\text{ for Glauber dynamics},
			\\
			L &\text{ for the Bernoulli--Laplace model},
			\\
			\kappa_* &\text{ for the hardcore model}.
		\end{cases}$
	\end{theorem}
	We remark here that the estimate \eqref{eq:p-wass-contraction-conj} was already established in \cite{con-2022} only for the specific case of interacting random walks on the grid of Section \ref{sec: Interacting random walks} (see Theorem 3.2 of \cite{con-2022}), which doesn't follow directly from the arguments of this section since we restricted our discussion to finite state space Markov chains.\\
	We also remark that Theorem \ref{thm:pos-lly-inf-curv} shows that, for some Markov chains, assumptions that are strictly connected to positive $\ricllyinf$ curvature are useful for establishing the modified log-Sobolev inequality, positive entropic curvature, positive discrete Bakry--\`{E}mery curvature and other related inequalities (cf. the discussion in Section \ref{sec:gen-ineq}). This suggests interesting connections with some other open problems in the theory of functional inequalities and discrete curvature for Markov chains, as we discuss next.
	\paragraph*{Peres--Tetali Conjecture} 
	One notable example is the following  important unpublished conjecture by Peres and Tetali, which links coarse Ricci curvature to the modified log-Sobolev inequality in the setting of lazy simple random walks on finite graphs (see also \cite[Con. 3.1]{Eld-Lee-Leh-2017}, \cite[Con. 4]{fat-2019}, \cite[Rmk. 1.1]{bla-cap-che-par-ste-vig-2022}).
	\begin{conjecture}
		There exists a universal constant $\alpha >0$ such that the following holds.
		Let $\Omega$ be a finite unweighted graph and consider the stochastic matrix $P$ associated with the simple random walk on this graph, $\ptil = \frac{P+I}{2}$ associated to the lazy simple random walk and  the generator $L= P-I$. If 	$\ricollone \geq K >0$ for the lazy stochastic matrix $\ptil$ (or, equivalently, if $\ricllyone\geq 2K$ for $L$), then $\mlsi(\alpha K)$ holds.
	\end{conjecture}
	In all the examples of Theorem \ref{thm:pos-lly-inf-curv} we have positive $\ricllyinf$ curvature, which implies in particular positive $\riclly_{,1}$ curvature by Proposition \ref{prop:properties-llyp}. Therefore, it is natural to study the following  problem related to the above conjecture: assuming a strictly positive lower bound of $\ricllyinf$ (and under some additional assumptions), is it possible to deduce a lower bound of the same order for the $\mlsi$ constant? In particular, if the additional assumptions are that we are in the setting of simple random walks on finite graphs, this problem constitutes a weaker form of the Peres--Tetali Conjecture.

	\paragraph*{Coarse and entropic curvature}
	Another important open problem consists in comparing the different notions of discrete curvature. For example, it is not known when a positive lower bound for the coarse Ricci curvature implies a positive lower bound of the same order for the entropic curvature, or vice versa (and similarly for the discrete Bakry--\'{E}mery curvature).
	In light of the results of this paper, the following is also a natural question: assuming a strictly positive lower bound of $\ricllyinf$ (and under some additional assumption), is it possible to deduce a lower bound of the same order for the entropic curvature?
	Interestingly, we remark that positive lower bounds for the entropic curvature are linked to exponential contraction with respect to the metric $\curlyW$. Precisely, if $\rice\geq K$ then
	\[
	\curlyW(\mu P_t,  \nu P_t )\leq e^{-Kt} \curlyW(\mu, \nu)
	\]
	for all $t \geq 0$ and starting probability measures $\mu, \nu$
	(see \cite[Prop 4.7]{erb-maa-2012} and cf. Proposition \ref{prop:coneqs-entr-low-bound}). This is of course reminiscent of the exponential decay of \eqref{eq:p-wass-contraction-conj}, which follows from $\ricllyinf \geq K$. However, not much is known about the relationship between $\curlyW$ and $W_p$ (see \cite[Prop. 2.12, 2.14]{erb-maa-2012} for a lower(/upper) bound of $\curlyW$ in terms of $W_1 (/W_2)$), so it is not clear how these exponential decay estimates are connected.

	\appendix

	\section{Proof of Proposition \ref{prop:rel-bak-em-esti}}
		\label{sec:proof-prop-rel-bak-eme-estim}

		Here, we work in the setting of Section \ref{sec:coup-curv} and with that notation; recall in particular that the state space is finite and the Markov chain is irreducible and reversible.

		\begin{proof} 
			Suppose $\Omega$ is finite and Assumption \ref{ass:mild-theta-phi} is satisfied and that \eqref{eq:main-ineq} holds for all 
			$\rho\colon \Omega \to \R_{>0}$ and $\psi \colon \Omega \to \R$.
			Fix now $\rho\colon \Omega \to \R_{>0}$ and choose $\psi = \phi'\circ  \rho$. Notice then that
			\begin{align*}
				&	\curlyA( \rho,\phi'\circ  \rho) 
				\\ 
				=\, &  \frac{1}{2} \sum_{x, y: \psi(x)\neq \psi(y)} \pi(x) Q(x,y) \frac{ \rho(x)- \rho(y)}{(\phi'\circ  \rho)(x) - (\phi'\circ  \rho)(y)} \quadr*{(\phi'\circ  \rho)(x) - (\phi'\circ  \rho)(y) }^2
				\\
				= \,& \en \tond*{ \rho, \phi'\circ  \rho}.
			\end{align*}
			Moreover if $\psi(x) \neq \psi(y)$ we have that
			\[
			\nabla \theta ( \rho(x),  \rho(y)) = \frac{1}{\psi(x)-\psi(y)} \quadr*{1-{\phi''( \rho(x))}\,{\theta( \rho(x), \rho(y))} , -1+ {\phi''( \rho(y))}\,{\theta( \rho(x), \rho(y)) }}.
			\]
			It follows that
			\begin{align*}
				\curlyC( \rho,\phi'\circ  \rho)   &
				=\frac{1}{4} \sum_{x,y} \pi(x) Q(x,y) \quadr*{(\phi'\circ  \rho)(x)-(\phi'\circ  \rho)(y)} \\
				&\cdot \graf*{ \quadr*{1-{\phi''( \rho(x))}\,{\theta( \rho(x), \rho(y))} , -1+ {\phi''( \rho(y))}\,{\theta( \rho(x), \rho(y)) }} \cdot \begin{pmatrix}
						L \rho(x)
						\\
						L \rho(y)
				\end{pmatrix}}
				\\
				& = \frac{1}{2} \en\tond*{L  \rho,\phi'\circ  \rho} -\frac{1}{2} \en\tond*{ \rho, (\phi''\circ  \rho) \cdot L  \rho}
				\\
				& = \frac{1}{2} \en\tond*{  \rho,L\tond*{\phi'\circ  \rho}} -\frac{1}{2} \en\tond*{  \rho, (\phi''\circ  \rho) \cdot L  \rho},
				\\
				\curlyD( \rho, \phi'\circ  \rho) &= \frac{1}{2}\sum_{x,y} \pi(x) Q(x,y) ( \rho(x) -  \rho(y)) (L(\phi'\circ  \rho)(x) - L(\phi'\circ  \rho)(y)) 
				\\
				&= \en( \rho, L(\phi'\circ  \rho)).
			\end{align*}
			
			Therefore 
			\begin{align*}
				\curlyB(  \rho,(\phi'\circ  \rho)) = -\frac{1}{2} \en\quadr*{  \rho, (\phi''\circ   \rho) \cdot L   \rho + L\tond*{\phi'\circ   \rho}}.
			\end{align*}
			
			From this, letting $\rho_t = P_t \rho$, we see that
			\[
			\frac{d}{dt} \Bigg|_{t=0} \en\tond*{  \rho_t, \phi'\circ   \rho_t} = - 2 \curlyB(  \rho,\phi'\circ  \rho).
			\]
			Therefore the inequality $\curlyB(\rho,\phi'\circ  \rho) \geq K \curlyA(\rho,\phi'\circ  \rho) $ is equivalent to inequality \eqref{eq:bak-eme-method}, from which $\csi_\phi(2K)$ follows.
		\end{proof}
		\begin{remark}
			In the particular case when $\theta$ is the logarithmic mean, for any constant $Z>0$ we have $\curlyB(Z\cdot \rho, \psi) = Z \curlyB(\rho,\psi)$ and $\curlyA(Z\cdot \rho, \psi) = Z \curlyA(\rho, \psi)$, so it suffices to consider the case where $\rho$ is a density with respect to $\pi$ when proving inequality \eqref{eq:main-ineq}.
		\end{remark}
		
		\section{Proof of Proposition \ref{prop:comp-beta-theta}}\label{sec:comp-beta-theta}
		In this section we prove Proposition \ref{prop:comp-beta-theta}, which gives the value of $M_\theta$ from \eqref{eq:beta-theta} for some of the weight functions considered in Section \ref{sec:gen-ineq}.

		\begin{proof}
			Recall the definition 
			\[
			M_\theta \coloneqq \inf_{\substack{s,t\geq 0: \\\theta(s,t)>0}} \frac{\theta(s,s)+\theta(t,t)}{2 \theta(s,t)}  \in [0,1].
			\]
			The statement about the arithmetic mean is trivial, while the one about the logarithmic mean $\theta_1$ follows from equation $(2.1)$ in \cite{erb-maa-2012}. Let us therefore consider the case of the weight function $\theta_\alpha$, which was defined in \eqref{eq:theta-phi} with $\phi = \phi_\alpha$ as in \eqref{eq:beckner-func}, corresponding to Beckner functionals. In other words, we have 
			\[
			\theta(s,t) = \frac{\alpha-1}{\alpha} \frac{s-t}{s^{\alpha-1}-t^{\alpha-1}}
			\]
			for $s\neq t > 0$ and $\theta(s,s) = \frac{1}{\alpha}s^{2-\alpha}$.
			Again, for $\alpha = 2$ the result is trivial, so we assume henceforth $1<\alpha<2$. Without loss of generality we can minimize over $s>t>0$, so that substituting the expression of $\theta_\alpha$ from \eqref{eq:theta-alpha} the problem reduces to computing 
			\begin{equation}
				\begin{split}
					M_{\theta_\alpha} &= \inf_{s>t>0} \frac{1}{2(\alpha-1)} \frac{\tond*{s^{2-\alpha}+t^{2-\alpha}}\cdot\tond*{s^{\alpha-1}-t^{\alpha-1}} }{s-t}
					\\
					& = \frac{1}{2(\alpha-1)} \graf*{1 + \inf_{s>t>0} \frac{s^{\alpha-1}t^{2-\alpha} - s^{2-\alpha}t^{\alpha-1}}{s-t}}
					\\
					& = \frac{1}{2(\alpha-1)} \graf*{1 + \inf_{\lambda > 1} \frac{\lambda^{\alpha-1}-\lambda^{2-\alpha}}{\lambda-1}}
				\end{split}
			\end{equation}	
			where we set $\lambda \coloneqq \frac{s}{t} >1$.
			If $\alpha\geq \frac{3}{2}$, the conclusion follows by noticing that 
			$\frac{\lambda^{\alpha-1}-\lambda^{2-\alpha}}{\lambda-1} \geq 0$ for $\lambda >1$ and by  letting $\lambda \to \infty$.
			Suppose hence now that $\alpha \in \tond*{1,\frac{3}{2}}$: to conclude that in this case $M_{\theta_\alpha} = 1$ it is enough to show that for $\lambda >1$
			\begin{equation}\label{eq:min-lambda-function}
				\frac{\lambda^{\alpha-1}-\lambda^{2-\alpha}}{\lambda-1} \geq 2\alpha -3.
			\end{equation}
			Notice that equality holds as $\lambda \to 1^+$.
			By density of $\Q$ in $\R$ and by a continuity argument, it suffices to show that for all $\lambda >1 $ and all \emph{even} integers $p,s \in \N$ with $p>s$ we have that 
			\[
			\frac{\lambda^{\frac{s}{p+s}} - \lambda^{\frac{p}{p+s}}}{\lambda-1} \geq -\frac{p-s}{p+s},
			\]
			where we used the substitution $\alpha = 1+\frac{s}{p+s}$.
			Rearranging this and renaming $\lambda \leftarrow \lambda^{\frac{1}{p+s}}$ we need to prove equivalently that for all $\lambda >1$ 
			\[
			\frac{\lambda^{p+s}-1}{p+s} \geq \frac{\lambda^{p}-\lambda^{s}}{p-s}.
			\]
			Equivalently, dividing both sides by $\lambda-1$ and denoting by $\am$ the arithmetic mean, we need to prove that
			\[
			\am \tond*{1,\ldots, \lambda^{p+s-1}} \geq \am \tond*{\lambda^s, \ldots, \lambda^{p-1}}.
			\] 
			This last inequality holds true, since for $\lambda >1$ and integers $0<i<j$ we have that
			\[
			\lambda^{i-1} + \lambda^{j+1} \geq \lambda^i + \lambda^j
			\]
			by the classical rearrangement inequality.
		\end{proof}

		\section{Gr\"{o}nwall's lemma with Dini derivative}
		\label{sec:gronwall-dini}
		For a continuous function $f\colon I \to \R$, where $I\subset \R$ is an interval, we consider the \emph{upper/lower Dini derivative}, which are defined respectively by 
		\begin{align*}
			\frac{d^+}{dt} \bigg|_{t_0} f(t) &= \limsup_{h\to 0^+} \frac{f(t+h)-f(t)}{h},
			\\
			\frac{d^-}{dt} \bigg|_{t_0} f(t) &= \liminf_{h\to 0^+} \frac{f(t+h)-f(t)}{h}.
		\end{align*}
		Clearly, $\frac{d^-}{dt}  f(t) \leq \frac{d^+}{dt}  f(t)$.
		It may be useful to apply Gr\"{o}nwall's lemma in absence of differentiability, by considering instead the Dini derivatives. In particular, the following variant holds.
		\begin{lemma}[Gr\"onwall's lemma]\label{lem:gronwall-dini}
			Let $I=[a,b) \subset \R$ be an interval,with  $ a<b\leq \infty$, and consider real valued continuous functions $u,\beta \colon I \to \R$.
			Suppose that for all $t\in I$
			\begin{equation}\label{eq:gronwall-ineq}
				\frac{d^-}{dt} u(t) \leq \beta(t) u(t).
			\end{equation}
			Then for all $t\in I$ 
			\[
			u(t) \leq \exp\tond*{\int_0^t\beta(s)ds} u(a).
			\]
		\end{lemma}
		
		\begin{proof}
			Set $v(t) = \exp\tond*{-\int_0^t \beta (s) ds}$, which satisfies $v(t)> 0$, $v'(t) = -\beta(t)v(t)$.
			Notice that for all $t\in I$
			\begin{align*}
				\frac{d^-}{dt} \quadr*{u(t)v(t)} = & \liminf_{h\to 0^+} \frac{\quadr*{u(t+h)-u(t)}v(t+h)}{h} +\frac{u(t)\quadr*{v(t+h)-v(t)}}{h} 
				\\
				= &\quadr*{\frac{d^-}{dt}u(t)}v(t) + u(t)v'(t)
				\\
				\leq & \beta(t)u(t)v(t) -u(t)\beta(t)v(t)
				\\
				= & 0.
			\end{align*}
			Next, we notice as in \cite[Eqn. (3.8)]{dan-sav-2008} that this implies that the continuous function $u(t)v(t)$ is non-increasing on $I$. In particular, we have $ u(t) v(t) \leq u(a)$, which implies the thesis by substituting the expression for $v(t)$ and rearranging.
		\end{proof}

	\subsubsection*{Acknowledgments}
		The author  warmly thanks Jan Maas for suggesting the project and for his guidance, and  Melchior Wirth and Haonan Zhang for useful discussions. The author is also grateful to an anonymous reviewer for carefully reading the manuscript and providing many valuable suggestions.
	
	\subsubsection*{Funding}
        The author gratefully acknowledges support by the European Research Council (ERC) under the European Union's Horizon 2020 research and innovation programme (grant agreement No.~716117) and by the Austrian Science Fund (FWF), Project SFB F65.

    \bibliographystyle{alpha}  
    \bibliography{bibliography}
\end{document}